\documentclass[14pt]{amsart}

\usepackage{bm}
 \newcommand{\tv}[1]{\bm {\mathcal{#1}}}

 \usepackage{upgreek}

\usepackage{stmaryrd} 
\usepackage[T1]{fontenc}
\usepackage{textcomp}

\usepackage{amsopn, amsthm, amsgen, amscd,amsmath,amssymb}
\usepackage[bbgreekl]{mathbbol}
\usepackage[mathcal]{euscript}
\DeclareMathAlphabet{\mathbbm}{U}{bbm}{m}{n}
\usepackage{color}
\usepackage{tikz}
  \usepackage{graphicx}
\usepackage{epstopdf}
\usepackage[small,nohug,heads = LaTeX]{diagrams} \diagramstyle[labelstyle=\scriptstyle]

\definecolor{CadetBlue}{cmyk}{0.62, 0.57, 0.23, 0 }
\definecolor{black}{cmyk}{1, 0.5, 0, 0 }
\definecolor{RedViolet}{cmyk}{0.07, 0.9, 0, 0.34 }
\definecolor{SeaGreen}{cmyk}{0.69, 0, 0.5, 0}

\DeclareMathAlphabet{\mathpzc}{OT1}{pzc}{m}{it}

\newcommand{\R}{\mathbb R}
\newcommand{\C}{\mathbb C}

\newcommand{\K}{\mathbb K}
\newcommand{\N}{\mathbb N}
\newcommand{\PR}{\mathbb P}
\newcommand{\Q}{\mathbb Q}
\newcommand{\Z}{\mathbb Z}

\newcommand{\SI}{\mathbb S}

\newcommand*\circled[1]{\tikz[baseline=(char.base)]{
            \node[shape=circle,draw,inner sep=.5pt] (char) {#1};}}

\newtheorem{theo}{Theorem}
\newtheorem{lemm}{Lemma}
\newtheorem{prop}{Proposition}
\newtheorem{coro}{Corollary}

\newtheorem*{conj}{Conjecture}
\newtheorem*{ques}{Question}

\theoremstyle{definition}

\theoremstyle{remark}

\newtheorem{exam}{Example}

\newtheorem{note}{Note}
\newtheorem*{aside}{Aside}

\newtheorem{obse}{Observation }

\newcommand{\bast}{{}^{\ast}}
\newcommand{\bbull}{{}^{\bullet}}

\newcommand{\bstar}{{}^{\star}}

\title[Arithmetic of Diophantine Approximation Groups I]{The Arithmetic of Diophantine Approximation Groups I: \\ Linear Theory}
\author{T.M. Gendron}
\address{Instituto de Matem\'{a}ticas -- Unidad Cuernavaca, Universidad
Nacional Autonoma de M\'{e}xico, Av. Universidad S/N, C.P. 62210
Cuernavaca, Morelos, M\'{E}XICO}
\email{tim@matcuer.unam.mx}
\date{28 June 2014}
\subjclass[2000]{Primary, 11J99, 11U10, 11R99}
\keywords{Diophantine approximation groups, approximate ideal arithmetic, Mahler classification, resultant arithmetic }
\begin{document}
\vspace{2cm}
\begin{abstract} 
A paradigm for a global algebraic number theory of the reals
 is formulated with the purpose of providing a unified setting for 
algebraic and transcendental number theory.  This is achieved through the
study of subgroups of nonstandard models of Dedekind domains called {\it diophantine approximation groups}.
The arithmetic of diophantine approximation groups 
is defined
in a way which extends the ideal-theoretic arithmetic of algebraic number theory, using the structure of an {\it approximate ideal}: a bifiltration by subgroups
 along which partial products may be performed.  
 \end{abstract}
 \maketitle
\tableofcontents
\section*{Introduction}

This is the first paper in a series of two introducing a paradigm within which a {\it global algebraic number theory}
for $\R$ may be formulated, in such a way as to make possible the synthesis of
algebraic and transcendental number theory into a coherent whole.   
This synthesis is made possible by passing to nonstandard models of well known
arithmetic objects, and while no deep model theory is brought to bear, it indicates the utility of model theoretic constructions in the advancement of certain
mathematical ideas.

Algebraic number theory is based upon the arithmetic of ideals in Dedekind domains; 
we incorporate transcendental number theory into this theory by introducing a generalized notion of ideal which we call a {\it diophantine approximation group}.
Diophantine approximation groups occur as subgroups of nonstandard models of classical Dedekind domains and their relatives.
 In particular, to  $\uptheta\in \R$ we may associate various Diophantine approximation groups 
depending on how one approximates $\uptheta$ -- by rational integers, by algebraic integers, by polynomials.    In this paper we will consider diophantine approximation groups
of the first two varieties, which together make up the {\it linear theory}.

Diophantine approximation groups come with natural filtrations -- called {\it approximate ideal structures} -- along which one can partially define products: the study of which gives rise to an arithmetic extending 
the usual arithmetic of ideals.   The heart of this paper then
consists of an extensive investigation as to how the arithmetic of diophantine approximation groups 
\begin{itemize}
\item reflects the class of the real number $\uptheta$ with respect to the linear classification:
rational, badly approximable, (very) well approximable and Liouville, 
\item introduces invariants which make possible finer distinctions amongst real numbers and
\item allows one to merge transcendental number theory within the theoretical framework of algebraic number theory.  
\end{itemize}  In the sequel \cite{GeArith2} we consider diophantine approximation groups consisting of polynomials and study
their arithmetic according to the nonlinear (Mahler) classification.

We now give a more detailed accounting of what is to be found here.  Fix $\mathfrak{u}\subset {\sf 2}^{\N}$ a nonprincipal ultrafilter on $\N$ and denote the ultrapower 
$ \bast\Z:=\Z^{\N}/\mathfrak{u}$.  By definition, $\bast\Z$ consists of equivalence classes of sequences in $\Z$, where sequences are identified
if they agree on subsequences indexed by some $X\in \mathfrak{u}$.  The ring $\bast\Z$ is a model of $\Z$ in the sense that its first order theory agrees with that
of $\Z$, see \S 1.

Given $\uptheta\in\R$,  the {\bf {\em diophantine approximation group}}  \[ \bast\Z (\uptheta)\subset \bast\Z\]
is the subgroup of $\bast n\in\bast\Z$ for which there exists $\bast n^{\perp}\in\bast\Z$ such that
\[ \bast n\uptheta -\bast n^{\perp} \simeq 0,\]
where $\simeq$ is the relation of being asymptotic to $0$ (infinitesimal) in the field $\bast\R := \R^{\N}/\mathfrak{u}$.  
The dual element $\bast n^{\perp}$ 
is uniquely determined by $\bast n$ and we refer to $(\bast n^{\perp},\bast n)$ as a ``numerator denominator pair'',  denoting it here using a suggestive pseudo fractional notation
\[ 
 \begin{array}{l}
 \bast n^{\perp} \\
 \widetilde{\bast n\;\;}
 \end{array}. 
       \]
       When $\uptheta = a/b\in\Q$ then 
  \begin{itemize}
\item[-]  $\bast\Z (\uptheta)=\bast(b)$ = the ultrapower of the ideal $(b)$.
\item[-] For every $\bast n\in \bast(b)$, $\bast n^{\perp}/\bast n=a/b$. 
\end{itemize}
Otherwise, if $\uptheta\not\in\Q$, $\bast\Z (\uptheta)$ is only a group and $\bast\Z (\uptheta )\cap\Z=0$:
  that is, to ``observe'' $\uptheta$ by way of $\Z$ it is essential that we leave the standard model.   
  
  The group $\bast\Z (\uptheta)$ was first introduced in \cite{Ge1} where it
appears as a generalized fundamental group for the Kronecker foliation of slope $\uptheta$; in this manifestation, it plays a central role in the definition of the quantum modular invariant \cite{CaGe}.
In \cite{Ge4},
variants of $\bast\Z (\uptheta)$ are considered, in which $\Z$ is replaced by the ring of integers $\mathcal{O}$ of a finite extension $K/\Q$ or by the polynomial ring $\Z[X]$,
or $\uptheta$ is replaced by a real matrix $\Uptheta$.  The focus of
that study is the relationship between diophantine approximation groups, Kronecker foliations and linear/algebraic independence.
In this paper we turn to the issue of arithmetic, motivated by a desire to answer the following 
\begin{ques}  Let $\uptheta,\upeta\in \R$ and
\[ \begin{array}{l}
 \bast m^{\perp} \\
 \widetilde{\bast m\;\;}
 \end{array},\quad
  \begin{array}{l}
 \bast n^{\perp} \\
 \widetilde{\bast n\;\;} 
  \end{array}\]
be numerator denominator pairs associated to
$\bast m\in \bast\Z (\uptheta)$,
$\bast n\in \bast\Z (\upeta)$.
Under what conditions  can they be manipulated as ring
elements via {\bf  fractional arithmetic}: that is, when do  
\[    \begin{array}{l}
 \bast m^{\perp} \\
 \widetilde{\bast m\;\;}
 \end{array} 
 \cdot 
  \begin{array}{l}
 \bast n^{\perp} \\
 \widetilde{\bast n\;\;} 
  \end{array}:=  \begin{array}{l}
\bast m^{\perp}\cdot \bast n^{\perp} \\
 \widetilde{\bast m\cdot\bast n\;\;\;\;}
 \end{array},\quad 
  \begin{array}{l}
 \bast m^{\perp} \\
 \widetilde{\bast m\;\;}
 \end{array} 
 \pm
  \begin{array}{l}
 \bast n^{\perp} \\
 \widetilde{\bast n\;\;} 
  \end{array}:=
   \begin{array}{l}
\left(\bast m\bast n^{\perp}\pm\bast m^{\perp}\bast n\right) \\
 \widetilde{{}\quad\quad\bast m\cdot\bast n\;\;\;\;\quad{}}
 \end{array}\]
define numerator denominator pairs corresponding to diophantine approximations of \[ \uptheta \upeta,\quad \uptheta\pm\upeta?\]
\end{ques}


As it turns out, our response to this question is closely related to the problem of determining conditions under which we may form a partial 
product of diophantine approximation groups in a way which generalizes the
product of ideals
in algebraic number theory.  

There are two quantitative measures of a diophantine approximation $\bast n\in\bast\Z (\uptheta )$ that have defined the field of Diophantine Approximation since
the time of Dirichlet and Liouville: 
\begin{itemize}
\item[1.] The growth of the denominator $\bast n$. 
\item[2.] The decay of the error term \[ \upvarepsilon (\bast n):=\uptheta \bast n-\bast n^{\perp}\in \bast\R_{\upvarepsilon}  \] 
where  $\bast\R_{\upvarepsilon}$ is the subgroup
of infinitesimals in $\bast\R$.
\end{itemize}
We measure these in the following way. Let 
\[ \langle\cdot\rangle:\bast\R\longrightarrow \bstar\PR\R:=\bast\R/\bast\R_{\rm fin}^{\times}\] 
be the Krull valuation on $\bast\R$ associated to the local subring $\bast\R_{\rm fin}\subset\bast\R$ of bounded nonstandard reals.  The ordered valuation group $\bstar\PR\R$  is a tropical semi ring with respect to operations $\cdot,+$ induced from 
their counterparts on $\bast\R$, the {\bf {\em growth-decay semi ring}},  
 \S \ref{tropical}.

For $\bast n\in\bast\Z (\uptheta )$ we define its {\bf {\em growth}} to be 
\[ \upmu(\bast n):=  \langle\bast n^{-1}\rangle\] and its {\bf {\em decay}} to be 
\[ \upnu (\bast n):=  \langle \upvarepsilon (\bast n)\rangle.\]
Then for each pair $\upmu,\upnu\in\bstar\PR\R_{\upvarepsilon} =$ the infinitesimal part of $\bstar\PR\R$,  
\[ \bast\Z^{\upmu}_{\upnu}(\uptheta ) = \left\{\bast n\in\bast\Z (\uptheta )\left|\;\; \upmu <\upmu(\bast n), \; \upnu (\bast n)\leq \upnu\right.\right\}\]
is a subgroup of $\bast\Z (\uptheta )$.  The
bi-filtered group \[ \bast\Z (\uptheta )=\{\bast\Z^{\upmu}_{\upnu}(\uptheta )\}\] is referred to as an {\bf {\em approximate ideal}}. 

The concept of an approximate ideal generalizes naturally that of ideal as follows.  If we consider just the growth filtration
$\bast\Z = \{\bast \Z^{\upnu}\}$ where $\bast\Z^{\upnu}=\{\bast n|\; \upnu<\upmu(\bast n)\}$ then for each $\upmu,\upnu\in \bstar\PR\R_{\upvarepsilon}$, 
\[ \bast  \Z^{\upnu}\cdot \bast\Z^{\upmu}_{\upnu}(\uptheta )\subset \bast\Z^{\upmu\cdot\upnu}(\uptheta ).\] 
See Proposition \ref{ideological}, \S \ref{gdarith}.  By forgetting the indices one recovers  the usual definition of an ideal.  

Determining when the subgroup $\bast\Z^{\upmu}_{\upnu}(\uptheta )$ is non trivial is the first problem which must be addressed. 
The {\bf {\em nonvanishing spectrum}} of $\uptheta$ is 
\[ {\rm Spec}(\uptheta)=\{ (\upmu,\upnu)|\; \bast\Z^{\upmu}_{\upnu}(\uptheta )\not=0\}\]
a ${\rm PGL}_{2}(\Z )$ invariant of $\uptheta$.
In \S \ref{nonvanspec}, we characterize the linear classification of the reals -- rational, badly approximable, (very) well approximable and Liouville -- in terms of their
nonvanishing spectra, see Figure 1 of \S \ref{nonvanspec}.
The intersection ${\rm Spec}_{\rm flat}(\uptheta )$
of ${\rm Spec}(\uptheta)$ with the line $\upmu=\upnu$ represents a critical divide called the {\bf {\em flat spectrum}}, whose study is taken up in \S \ref{flatsection}.  The flat
spectrum reflects properties of the partial fraction decomposition of $\uptheta$ rather than its exponent.

The Question posed above is answered in \S \ref{gdarith} using the approximate ideal structure: there is a bilinear map 
\begin{align}\label{gdprodintro} \bast\Z^{\upmu}_{\upnu}(\uptheta )\times\bast\Z^{\upnu}_{\upmu}(\upeta )\longrightarrow \bast\Z^{\upmu\cdot\upnu}(\uptheta \upeta)\cap\bast\Z^{\upmu\cdot\upnu}(\uptheta +\upeta) \cap\bast\Z^{\upmu\cdot\upnu}(\uptheta -\upeta)
\end{align}
defined by the ordinary product in $\bast\Z$.  
This means that whenever $\bast m\in  \bast\Z^{\upmu}_{\upnu}(\uptheta )$ and $\bast n\in \bast\Z^{\upnu}_{\upmu}(\upeta )$, then their numerator denominator pairs may be multiplied and 
added/subtracted exactly as formulated in the Question.  
When $\uptheta =a/b$, $\upeta=c/d\in\Q$, (\ref{gdprodintro}) reduces to the product map
 $\bast (b)\times \bast (d)\rightarrow \bast (bd)$ of the principal ideals generated by the denominators.

For $\upmu\geq\upnu$ we define the  composability relation 
\[\uptheta {}_{\upmu}\!\!\owedge_{\upnu}\upeta\]
whenever the groups appearing in the product (\ref{gdprodintro}) are nontrivial i.e.\ for $(\upmu,\upnu)\in {\rm Spec}(\uptheta)$, $(\upnu,\upmu)\in {\rm Spec}(\upeta)$.
The remainder of \S \ref{gdarith} is devoted to analyzing this relation with respect to the linear classification of real numbers.   Roughly speaking, composability increases as one progresses
from the badly approximable numbers to the Liouville numbers.

In this connection a new phenomenon emerges: the existence of {\bf {\em  antiprimes}} -- classes of numbers for which the relation ${}_{\upmu}\!\owedge_{\upnu}$ is empty
for all  possible growth-decay parameters.  The unique maximal antiprime set is the set $\mathfrak{B}$ of badly approximable numbers.  There is a ``splitting'' theory for antiprimality
not unlike that for primes when one passes to an algebraic extension, which is described further below.

Approximate ideal arithmetic in the case of the flat product (which amounts to the consideration of the flat relation ${}_{\upmu}\!\owedge_{\upmu}$)
does not parse along the linear classification 
and properties relating to the combinatorics of the partial quotient representation $\uptheta=[a_{1}a_{2}...]$
must be used to study  composability.  The classification is transverse to the linear classification e.g. there exist Liouville numbers which are not flat composable
with any other number, see
 \S \ref{flatharith}.

In the field of Diophantine Approximation, one frequently restricts attention to diophantine approximations
with error dominated by some function $\uppsi:\bast\Z\rightarrow \bast\R$ i.e. in our language this means studying the set
\[ \bast\Z (\uptheta|\uppsi )=\{ 0\not=\bast n\in\bast\Z (\uptheta )|\; |\upvarepsilon(\bast n)|<|\uppsi(\bast n)|\}\cup \{ 0\}. \] 
When $\uppsi (x)=x^{-1}$, $\bast\Z (\uptheta |x^{-1} )$ is the set of elements of bounded {\bf {\em $\boldsymbol\uptheta$-norm}} 
\[ |\bast n|_{\uptheta}:=(|\bast n|\cdot|\upvarepsilon (\bast n)|)^{1/2}\mod \bast\R_{\upvarepsilon}.\]  In \S \ref{metrical} we show that $\bast\Z (\uptheta|x^{-1} )$ has the structure of an {\bf {\em approximate group}}: with respect to the growth-decay grading
$\bast\Z (\uptheta|x^{-1} )= \{\bast\Z^{\upmu}_{\upnu} (\uptheta|x^{-1} )\}$ there is a sum
\begin{align}\label{growthdecaysumintro}\bast\Z^{\upmu}_{\upnu} (\uptheta|x^{-1} ) + \bast\Z^{\upnu}_{\upmu} (\uptheta|x^{-1} )\subset \bast\Z^{\upmu-\upnu} (\uptheta|x^{-1} )
\end{align}
where $\upmu-\upnu=\min (\upmu,\upnu)$, see Theorem \ref{symmgroupology} of \S \ref{metrical}.

There is an important further refinement of the above approximate group defined by the set
of {\bf {\em symmetric diophantine approximations}} 
\[ \bast\Z^{\rm sym} (\uptheta)= \{ \bast n\in\bast\Z (\uptheta )|\;  0<|\bast n|_{\uptheta}<\infty\}\cup \{ 0\}\subset \bast\Z (\uptheta|x^{-1} ) .  \]
We show that $\bast\Z^{\rm sym} (\uptheta )$ is non trivial, 
and has the structure of a uni-indexed approximate group, see Theorem \ref {symmgroupologystruct}
of \S \ref{metrical}. 
In the case of $\uptheta=\upvarphi=$ the golden mean, we give an explicit description of the elements of $\bast\Z^{\rm sym} (\upvarphi )$
using the Zeckendorf representations of natural numbers.  The latter may be useful in the consideration of the
 {\it Littlewood conjecture}: 
\[  \lim\inf_{n} n\|n\uptheta\|\|n\upeta\|=0, \quad \uptheta,\upeta\in\mathfrak{B},  \quad \|\cdot \|= \text{ distance to the nearest integer}\]
which 
is implied by the statement 
\[ \bast\Z^{\rm sym} (\uptheta)\cap\bast\Z (\upeta)\not=\emptyset \quad \text{or}\quad \bast\Z^{\rm sym} (\upeta)\cap\bast\Z (\uptheta)\not=\emptyset, 
\quad \uptheta,\upeta\in\mathfrak{B}.
\]

The restriction of $|\cdot |_{\uptheta}$ to $\bast\Z^{\rm sym} (\uptheta )$ is not subadditive: rather, it satisfies the reverse triangle inequality, due to the fact
that it most naturally arises from a Lorentzian bilinear pairing of signature $(1,1)$ on $\bast\Z^{\rm sym} (\uptheta )$.   
Thus, if we view diophantine approximations as ``material particles departing from $\uptheta$'' then $|\bast n|_{\uptheta}$ is nothing more than the initial speed; for badly approximable numbers, we have Heisenberg's uncertainty principle 
\[ |\bast n|_{\uptheta}>C_{\uptheta}\]
where $C_{\uptheta}$ is the corresponding element of the Lagrange spectrum.  See \S \ref{Lorentzian}.


The remaining sections concern the integration of the above theory with classical algebraic number theory.  
Before embarking on this road, we will need the analogue of diophantine approximation groups for matrices.
Given $\Uptheta$ a real $r\times s$ matrix (or in the classical language: a family of $r$ linear forms in $s$ variables), the matrix approximate ideal 
\[ \bast \Z^{s}(\Uptheta)=\{ (\bast \Z^{s})^{\upmu}_{\upnu}(\Uptheta)\}\] is the subject of \S \ref{matideoarith}.  The approximate ideal product derives from a fractional arithmetic on the set of all real matrices $\tilde{\mathcal{M}}(\R)$ based on the Kronecker product, as well as an arithmetic based on the Kronecker sum of matrices on the subset $\mathcal{M}(\R)\subset \tilde{\mathcal{M}}(\R)$ of square matrices.  
The classes of badly approximable, (very) well approximable and Liouville matrices are characterized (or rather defined) by the shape of their associated nonvanishing spectra.  
In the special case of a single form the dual groups give rise to anarithmetic of {\bf {\em nonprincipal approximate ideals}}.

Let
$K/\Q$ be a finite extension, $\mathcal{O}$ the ring of $K$-integers and $\K\cong \R^{d}$ the Minkowski space of $K$. In \S \ref{KIdeologicalArith} we consider
the diophantine approximation group of $\boldsymbol z\in\K$, which has the structure of an
approximate ideal 
\[ \bast\mathcal{O}(\boldsymbol z)=\{ \bast \mathcal{O}^{\boldsymbol\upmu}_{\boldsymbol\upnu}(\boldsymbol z)\}. \] These $K$-approximate ideals
may be multiplied according to an obvious 
analogue of (\ref{gdprodintro}).  If $K/\Q$ is Galois, then the action of ${\rm Gal}(K/\Q)$ on $\K$ extends to an action
on growth-decay indices so that the growth-decay product becomes Galois natural, c.f.\ Theorem \ref{galrespects}.
For $K/\Q$ finite degree and $\uptheta\in\R\subset\K$, the associated trace map ${\rm Tr}_{K}:\bast\mathcal{O}(\uptheta)\rightarrow\bast\Z(\uptheta)$ respects growth-decay structure, see Proposition \ref{TraceKDA}.  Not surprisingly,
the situation with norm maps is more complicated; however when $K/\Q$ is quadratic the norm map is defined and respects growth-decay structure, see Proposition \ref{normgrowthdecay}.
  The $K$-nonvanishing spectrum ${\rm Spec}_{K}(\boldsymbol z)$ may be used to define the nontrivial classes of
$K$-badly approximable, $K$-(very) well approximable and $K$-Liouville elements of $\K$.
One observes the phenomenon of {\bf {\em antiprime splitting}}, where a $\Q$-badly approximable number $\uptheta$ loses
its antiprime status upon diagonal inclusion in $\K$: this happens for quadratic Pisot-Vijayaraghavan numbers, see Theorem \ref{PVnumber}. 

The last section, \S \ref{ideologicalclasssection}, is devoted to the approximate ideal generalization of ideal class group.  The approximate ideal class of $\bast\mathcal{O}(\boldsymbol z)$ 
is defined
by the  {\bf {\em decoupled approximate ideal}}
\[ \bast[\mathcal{O}](\boldsymbol z ) :=  \bast \mathcal{O}(\boldsymbol z)+\bast \mathcal{O}(\boldsymbol z)^{\perp},\quad  \bast \mathcal{O}(\boldsymbol z)^{\perp} := \{ \bast \upalpha^{\perp}|\;
\bast\upalpha\in   \bast \mathcal{O}(\boldsymbol z)\}.\]
The set of decoupled approximate ideals $\mathcal{C}l (\K)$ extends the usual ideal class group $\mathcal{C}l (K)$ of $K/\Q$: if 
\[ \mathfrak{a} = (\upalpha,\upbeta ),\quad \mathfrak{a}' = (\upalpha',\upbeta' )\subset\mathcal{O}\]  are classical ideals and $\upgamma = \upalpha/\upbeta,\upgamma' = \upalpha'/\upbeta'$ then 
\[ \bast [\mathcal{O}](\upgamma)=\bast [\mathcal{O}](\upgamma') \Longleftrightarrow [\mathfrak{a}]=[\mathfrak{a}'] \;\; \text{(equality of ideal classes)}.\]
There is a canonical surjective map  \[ {\rm PGL}_{2}(\mathcal{O})\backslash\K \longrightarrow \mathcal{C}l (\K)\]
which extends the bijection ${\rm PGL}_{2}(\mathcal{O})\backslash K\leftrightarrow   \mathcal{C}l (K)$ and which is conjecturally a bijection as well.
 When $K=\Q$,  ${\rm PGL}_{2}(\Z)\backslash\R$ is the moduli space of quantum tori.
 
 While the product of decoupled approximate ideals extends the usual product of ideal classes,
 the result may not belong to $\mathcal{C}l (\K)$: indeed, there are nilpotent decoupled approximate ideals e.g. $\bast[\Z](\uptheta )^{2} =0$
 for $\uptheta$ badly approximable.  To retrieve these lost products, we introduce for each finite set $\{\boldsymbol z_{1},\dots ,\boldsymbol z_{k} \}\subset\K$ the {\bf {\em correlator
 decoupled approximate ideal}} 
 \[ \bast[\mathcal{O}](\boldsymbol z_{1}|\cdots |\boldsymbol z_{k} ),\]
 by definition the group generated by any approximate ideal admissible product of the $\bast[\mathcal{O}](\boldsymbol z_{i} )$, $i=1,\dots, k$.  The set of
 all such correlator  decoupled approximate ideals forms a monoid with nullity $\mathcal{C}l_{\infty} (\K)\supset \mathcal{C}l (\K)$.  
 In the case $K=\Q$ we conjecture that for $\uptheta$ well-approximable of exponent $\upkappa$, the decoupled approximate ideal
 $\bast[\Z](\uptheta )$ is $\lfloor\upkappa +2\rfloor$-step nilpotent, and that if $\uptheta$ is Liouville, we conjecture that $\bast[\Z](\uptheta )$ is neither
 nilpotent nor of finite order.


\vspace{3mm}

\noindent {\bf Acknowledgements.} 
This paper was supported in part by the CONACyT grant 058537 as well as the PAPIIT grant IN103708.

\section{Nonstandard Structures}

This brief section contains all the reader will need to know about nonstandard structures \cite{ChKei}, \cite{Go}.

Let $I$ be a set.  A {\bf filter} on $I$ is a subset $\mathfrak{f}\subset {\sf 2}^{I}$ satisfying
\begin{enumerate}
\item[-] If $X, Y\in \mathfrak{f}$ then $X\cap Y\in \mathfrak{f}$.
\item[-] If $X\in \mathfrak{f}$ and $X\subset Y$ then $Y\in \mathfrak{f}$.
\item[-] $\emptyset\not\in \mathfrak{f}$.
\end{enumerate}
Any set $\mathcal{F}\subset {\sf 2}^{I}$ satisfying the finite intersection property generates a filter, denoted $\langle \mathcal{F}\rangle$.
A maximal filter $\mathfrak{u}$ is called an {\bf ultrafilter}.  Equivalently, a filter $\mathfrak{u}$ is an ultrafilter $\Leftrightarrow$ for all $X\in {\sf 2}^{I}$, $X\in \mathfrak{u}$ or $I-X\in \mathfrak{u}$.
An ultrafilter $\mathfrak{u}$ is {\bf principal} if it contains a finite set $F$: equivalently $\mathfrak{u}=\langle F\rangle$.  Otherwise it is {\bf nonprincipal}.  By Zorn's lemma, every filter is contained in an ultrafilter.

Now let $\{G_{i}\}_{i\in I}$ be a family of algebraic structures of a fixed type: for our purposes, they will be groups, rings, fields. Let $\mathfrak{u}$ be an ultrafilter on $I$. The quotient
\[  \prod_{i\in I} G_{i}/\sim_{\mathfrak{u}},\quad (g_{i})\sim_{\mathfrak{u}} (g_{i}') \Longleftrightarrow \{ i|\; g_{i}=g_{i}'\}  \in\mathfrak{u} \]
is called the {\bf ultraproduct} of the $G_{i}$ w.r.t.\ $\mathfrak{u}$.  By the Fundamental Theorem of Ultraproducts (\L o\'{s}'s Theorem) \cite{ChKei}, the ultraproduct is also a group/ring/field according to the case.  If $G_{i}=G$ for all $i$ the ultraproduct is called an {\bf ultrapower} and is denoted 
\[ \bast G=\bast G_{\mathfrak{u}}.\]
Elements of $\bast G$ will be denoted 
\[ \bast g = \bast \{ g_{i}\}.\]
The canonical inclusion $G\hookrightarrow\bast G$ given by constants $g\mapsto(g=g_{i})$ is a monomorphism.  If $\mathfrak{u}$ is nonprincipal, this map is not onto and 
again by \L o\'{s}, exhibits $\bast G$ as a {\bf nonstandard model} of $G$: that is, the set of sentences in first order logic satisfied by $\bast G$ coincides with that of $G$.

If $I=\N$ and $\mathfrak{u}$ is a nonprincipal ultrafilter on $\N$ we denote by \[ \bast\Z\subset\bast\Q\subset\bast\R\subset\bast\C\] corresponding ultrapowers of $\Z\subset \Q\subset \R\subset \C$.  The field $\bast \R$ is totally ordered and the absolute value $|\cdot |$ extends to a map $|\cdot|:\bast \R\rightarrow \bast \R_{+}\cup \{0\}$.
We define the local subring of bounded elements
\[ \bast\R_{\rm fin} := \{\bast r\in\bast\R|\; \exists M\in\R_{+}\text { such that } |\bast r|< M\} \]
whose maximal ideal is the ideal of {\bf infinitesimals}
\[ \bast\R_{\upvarepsilon}  :=\{ \bast r\in\bast\R_{\rm fin}|\; \forall M\in \R_{+},\; |\bast r|< M \}.\]
Then $\bast\R$ is the field of fractions of $\bast\R_{\rm fin}$ and the residue class field is 
\[ \bast\R_{\rm fin}/ \bast\R_{\upvarepsilon}\cong \R.\]

\section{Tropical Growth-Decay Semi-Ring}\label{tropical}

 Let
\[ (\bast\R_{\rm fin})_{+}^{\times} = \text{ the group of positive units in the ring $\bast\R_{\rm fin}$}. \] Thus $(\bast\R_{\rm fin})_{+}^{\times}$ is
the multiplicative
subgroup of noninfinitesimal, noninfinite elements in $\bast\R_{+}$.
Consider the multiplicative quotient group 
\[  \bstar\PR\R := 
\bast\R_{+}/(\bast\R_{\rm fin})_{+}^{\times},\] 
whose elements will be written 
\[ \upmu = \bast x\cdot(\bast\R_{\rm fin})_{+}^{\times}.\]  We denote the product
in $\bstar\PR\R$ by ``$\cdot$''.

\begin{prop}\label{integerrepre} Every element $\upmu\in \bstar\PR\R$ may be written in the form 
\[ \bast n^{\upvarepsilon}\cdot(\bast\R_{\rm fin})_{+}^{\times}\]
where $\bast n\in\bast \Z_{+}-\Z_{+}$ or $\bast n =1$, and $\upvarepsilon=\pm 1$.
\end{prop}

 \begin{proof} Every element of $\bstar\PR\R$ is the class of 1, the class of an infinite element or the class of an infinitesimal element. 
If $\upmu$ is the class of $\bast r$ infinite, then there exists $\bast \bar{r}\in [0,1)=\{\bast x|\; 0\leq \bast x <1\}$ and $\bast n\in\bast\Z_{+}$
for which $\bast r= \bast n+\bast  \bar{r} = \bast n \cdot ((\bast n+\bast  \bar{r})/\bast n)$.  But  
$(\bast n+\bast  \bar{r})/\bast n = 1+\bast\bar{r}/\bast n\in(\bast\R_{\rm fin})_{+}^{\times}$, so  $\upmu = \bast n\cdot (\bast\R_{\rm fin})_{+}^{\times}$.  A similar argument may be used to show that when $\upmu$ represents an infinitesimal class, $\upmu =\bast n^{-1}\cdot(\bast\R_{\rm fin})_{+}^{\times}$ for some $\bast n\in\bast \Z_{+}-\Z_{+}$.
\end{proof}

\begin{prop}\label{denselinord}  $ \bstar\PR\R$ is a densely ordered group.
\end{prop}

\begin{proof}  The order is defined
by declaring that $\upmu<\upmu'$ in  $\bstar\PR\R$ if for any pair of representatives
$\bast x\in\upmu, \bast x'\in\upmu'$ we have $ \bast x<\bast x'$, evidently a dense order without endpoints. The left-multiplication action 
of $\bast \R_{+}$ on $ \bstar\PR\R$ preserves this order, therefore so does the product: if $\upmu <\upnu$ then for all $\xi\in \bstar\PR\R$, $\xi\cdot \upmu <\xi\cdot \upnu$.  
\end{proof}

We introduce the maximum of a pair of elements in $ \bstar\PR\R$ 
as a formal binary operation:
\[ \upmu+ \upnu:= \max (\upmu, \upnu).\]
The operation $+$
is clearly commutative and associative.
The following Proposition says that $+$ is the quotient of the operation $+$ of $\bast\R^{\times}_{+}$. 

\begin{prop} Let $\upmu = \bast x\cdot  (\bast\R_{\rm fin})_{+}^{\times}$, $\upmu'=\bast x'\cdot  (\bast\R_{\rm fin})_{+}^{\times}$.
Then \[ (\bast x+\bast x')\cdot (\bast\R_{\rm fin})_{+}^{\times}=
\upmu+ \upmu'.\]
\end{prop}

\begin{proof} Note that $\bast x+\bast x'\in\bast\R_{+}$ and
$ \bast x+\bast x'\in \max (\upmu, \upmu')$.   Indeed, suppose first that $\upmu\not=\upmu'$, say $\upmu<\upmu'$.  Then there exists $\bast \upvarepsilon$ infinitesimal
for which $\bast x = \bast \upvarepsilon\bast x'$, and we have   
$   \bast x' + \bast x = \bast x' (1+\bast \upvarepsilon)\in \upmu'$.       
If $\upmu=\upmu'$ then  $\bast x' = \bast r\bast x$ for $\bast r\in \bast\R_{+}^{\times}$ and 
$ (\bast x+ \bast x')\cdot  (\bast\R_{\rm fin})_{+}^{\times} = \bast x (1+\bast r)\cdot  (\bast\R_{\rm fin})_{+}^{\times} = \upmu = \upmu +\upmu$.
\end{proof}

\begin{prop}\label{troprop}  Let $\bast r,\bast s\in\bast\R_{+}$ and $\upmu,\upnu,\upnu'\in \bstar\PR\R$. Then
\begin{enumerate}
\item[1.]  $\upmu \cdot (\upnu+ \upnu') =  (\upmu \cdot \upnu )+ (\upmu \cdot \upnu' )$.
\item[2.]  $\bast r \cdot (\upnu+ \upnu')= (\bast r\cdot\upnu )+ (\bast r\cdot\upnu' )$.
\item[3.] $(\bast r+\bast s)\cdot\upmu = (\bast r\cdot\upmu)+ (\bast s\cdot\upmu)$.
\end{enumerate}
\end{prop}

\begin{proof}  1.  It is enough to check the equality in the case $\upnu'>\upnu$.  Then $\upmu \cdot (\upnu+ \upnu')=\upmu\cdot\upnu'$.
But the latter is equal to $ (\upmu \cdot \upnu )+ (\upmu \cdot \upnu' )$ since the product preserves
the order.  The proof of 2. is identical, where we use the fact that the multiplicative action by
$\bast\R_{+}$ preserves the order.  Item 3. is trivial.
\end{proof}

It will be convenient to add the class $-\infty$ of the element $0\in\bast\R$
to the space $\bstar\PR\R$: in other words, we will
reconsider $\bstar\PR\R$ as the quotient 
$ (\bast\R_{+}\cup\{0\})/(\bast\R_{\rm fin})_{+}^{\times}$. 
Note that
we have for all $\upmu\in\bstar\PR\R$
\[  -\infty+ \upmu = \upmu, \quad  -\infty\cdot \upmu =-\infty .\]
In particular, $-\infty$ is the neutral element for the operation $+$.
Thus, by Proposition \ref{troprop}:

\begin{theo} $\bstar\PR\R$
is an abstract (multiplicative)
 tropical semi-ring: that is, a max-times semi ring.
\end{theo}

We will refer to $\bstar\PR\R$ as the {\bf growth-decay semi-ring}.  Let $\bstar \PR\R_{\upvarepsilon} \subset\bstar\PR\R$ be the image of the $(\bast\R_{\rm fin})_{+}^{\times}$-invariant multiplicatively closed set $(\bast\R_{\upvarepsilon})_{+}$.  With the operations $\cdot, +$, $\bstar \PR\R_{\upvarepsilon}$ is a sub tropical semi-ring: the {\bf decay semi-ring}.

 If we forget the tropical addition, considering $ \bstar\PR\R$ as a linearly ordered multiplicative group, then
the map 
\[ \langle \cdot \rangle :\bast\R\rightarrow \bstar\PR\R, \quad \langle\bast x\rangle = |\bast x|\cdot(\bast\R_{\rm fin})_{+}^{\times},\] is the Krull valuation associated 
to the local ring $\bast\R_{\rm fin}$
(see for example \cite{ZS}).  The restriction of $\langle \cdot \rangle$ to $\R$ is just the trivial valuation, so
that $\langle \cdot \rangle $ cannot be equivalent to the usual valuation $|\cdot |$ on $\bast\R$ induced from the euclidean norm.
Note also that $\langle \cdot \rangle$
is nonarchimedean.  We refer to $\langle\cdot\rangle$ as the {\bf growth-decay valuation}.

\section{Growth-Decay Filtration}\label{gdfiltsection}

As in the previous section, $\bstar \PR\R_{\upvarepsilon} \subset\bstar\PR\R$ denotes the decay semi-ring.
We will measure growth of an infinite element of $\bast\Z$ in terms of the decay of its reciprocal: this has the advantage
of allowing us to make the vital comparison of denominator growth with error decay of a diophantine approximation in a single, unambiguous setting.
 
 For each $0\not=\bast n\in\bast\Z$ define its {\bf growth} by 
 \[ \upmu (\bast n):=\langle |\bast n^{-1}|\rangle\in\bstar \PR\R_{\upvarepsilon}\cup \{ 1\};\]
 note that $\upmu (\bast n)=1$ $\Leftrightarrow$ $\bast n=n\in\Z$.  For each
  $\upmu\in \bstar \PR\R_{\upvarepsilon} $  denote by
\[  \bast\Z^{\upmu }= \left\{0\not= \bast n \in \bast \Z\middle|\; \upmu (\bast n) >\upmu  \right\}\cup \{ 0\}=\left\{ \bast n \in \bast \Z\middle|\; |\bast n|\cdot\upmu \in \bstar\PR\R_{\upvarepsilon } \right\}
.  \]
Note that $ \bast\Z^{\upmu }$ is a well-defined 
subgroup of $\bast\Z$.  If $\upmu< \upmu'$ then
\begin{align}\label{growthincl}
  \bast\Z^{\upmu}\supset   \bast\Z^{\upmu '}.
  \end{align}
The collection $\{ \bast\Z^{\upmu}\}$ forms an order-reversing filtration of $\bast\Z$ by subgroups,
called the {\bf growth filtration}.  Notice that  
\[ \bast \Z^{\upmu}\cdot \Z^{\upmu' }\subset \Z^{\upmu\cdot\upmu'}\]
so that $\bast\Z$ has the structure of a filtered ring with respect to the growth filtration.

It will be useful to introduce the following subordinate filtration to the growth filtration.  Fix $\upmu\in \bstar\PR\R_{\upvarepsilon } $ and
for each $\upiota\in \bstar\PR\R_{\upvarepsilon } $ define
\[ \bast\Z^{\upmu[\upiota ]}:=\left\{ \bast n\middle|\; |\bast n|\cdot \upmu < \upiota  \right\} . \]
Then $\bast\Z^{\upmu[\upiota ]}$ is a group 
since by Proposition \ref{troprop}, item 3., 
\[|\bast m +\bast n|\cdot\upmu  \leq |\bast m| \cdot\upmu + |\bast n|\cdot \upmu <\upiota .
\]
Note that if $\upiota < \uplambda$ then $\bast\Z^{\upmu[\upiota ]}\subset \bast\Z^{\upmu[ \uplambda]}$.
We call this the {\bf fine growth bi-filtration}. The fine growth bi-filtration makes of $\bast\Z$ a bi-filtered ring:
\[ \bast \Z^{\upmu[\upiota ]}\cdot \Z^{\upmu' [\upiota']}\subset \Z^{\upmu\cdot\upmu'[\upiota\cdot\upiota']}.\]

For $\uptheta\in\R$, recall (see \S 2 of \cite{Ge4}) that by a {\bf diophantine approximation} 
we mean an element $\bast n\in \bast\Z$ such that
\[   \upvarepsilon (\bast n):= \bast n\uptheta - \bast n^{\perp} \in\bast\R_{\upvarepsilon}\]
for some 
\[ \bast n^{\perp} = \bast n^{\perp_{\uptheta}}\in\bast\Z, \]
called the {\small $\boldsymbol\uptheta${\bf -dual}} or simply the {\bf dual} of $\bast n$ if $\uptheta$ is understood. 
The {\bf diophantine approximation group} is then
\begin{align}\label{dagdefrecall}
\bast\Z (\uptheta ) & =\left\{\bast n\in\bast\Z\middle|\;\bast n\text{ is a diophantine approximation of }\uptheta \right\}\subset\bast\Z.
\end{align} 
Write
$  \bast\Z^{\upmu}(\uptheta ) =  \bast\Z^{\upmu} \cap  \bast\Z(\uptheta )$ and $\bast\Z^{\upmu[\upiota ]}(\uptheta )=\bast\Z^{\upmu[\upiota ]}\cap
\bast\Z(\uptheta )$.

We now introduce a second filtration which is only available for the groups $\bast\Z (\uptheta )$.  Let $\upnu\in \bstar\PR\R_{\upvarepsilon }$.  For
each $\bast n\in\bast\Z (\uptheta )$ write 
\[ \upnu (\bast n):=\langle |\upvarepsilon (\bast n)|\rangle \in\bstar\PR\R_{\upvarepsilon }, \] 
the {\bf decay} of $\bast n$. 
We define
\[  \bast\Z_{\upnu}(\uptheta) = \left\{ \bast n \in \bast\Z (\uptheta ) \middle| \; \upnu (\bast n)   \leq \upnu \right\} \]
which is a subgroup of $\bast\Z (\uptheta )$:
for $\bast n,\bast n'\in \bast\Z_{\upnu}(\uptheta)$,
$|\upvarepsilon (\bast n +\bast n')|\leq |\upvarepsilon (\bast n)|+|\upvarepsilon (\bast n)|$ and therefore
$\upnu (\bast n+\bast n')\leq\upnu $.
 Note that if $\upnu< \upnu'$,
\begin{align}\label{decayincl}
\bast\Z_{\upnu}(\uptheta)  \subset  \bast\Z_{\upnu'}(\uptheta)  
\end{align}
which produces an order-preserving filtration of $ \bast\Z (\uptheta )$ called
the {\bf decay filtration}.
Finally we denote the intersection
\[  \bast\Z^{\upmu}_{\upnu}(\uptheta ) = \bast\Z^{\upmu}(\uptheta )\cap \bast\Z_{\upnu}(\uptheta ) , \]
which we refer to as the {\bf growth-decay bi-filtration}
of $\bast\Z (\uptheta )$.  In addition we have the 
  {\bf fine growth-decay tri-filtration}
 $\bast\Z^{\upmu [\upiota]}_{\upnu}(\uptheta )$.
 
\begin{aside}  The reader may wonder why we have chosen to use a {\it strict} inequality to define the growth filtration and yet
a {\it non strict} inequality to define the decay filtration.   The strict inequality in the growth filtration is required in the formulation of the approximate ideal product (see Theorem \ref{productformula}).  The non strict inequality in the decay filtration is used in order to take into account the fact that the strict inequality present in Dirichlet's Theorem may become non strict upon passage to the growth-decay semi ring $\bstar\PR\R$.
\end{aside}

\begin{prop} For all $\upmu, \upnu,\upiota\in \bstar\PR\R_{\upvarepsilon }$, $\upnu\not=-\infty$, $\bast\Z^{\upmu[\upiota]}(\uptheta )$ and  $\bast\Z_{\upnu}(\uptheta )$ are nontrivial and uncountable.  For $\upnu=-\infty$,  $\Z_{-\infty}(\uptheta )$ is non trivial $\Leftrightarrow$
$\uptheta\in\Q$.
\end{prop}

\begin{proof}  Let $\upmu$, $\upnu$  and $\upiota$ be represented by sequences of positive real numbers $\{r_{k}\}$, $\{s_{k}\}$ and $\{ i_{k} \}$  converging to $0$.  Let $\bast n\in\bast\Z (\uptheta )$ be represented by the
sequence of integers $\{ n_{k} \}$.  We may choose $\bast n$ so that $n_{k}r_{k}\rightarrow 0$;
in fact, so that $|n_{k}r_{k}|<i_{k}$.  The sequence $\{n_{k}\}$ may then be used to construct uncountably many elements
of $\bast\Z^{\upmu[\upiota]}(\uptheta )$.  Similarly, we may find a class $\bast m\in\bast\Z (\uptheta )$
represented by $\{ m_{k} \}$ so that $|m_{k}\uptheta - m_{k}^{\perp}|\leq s_{k}$, and the sequence $\{m_{k}\}$ may then be used to construct uncountably many elements
of $\bast\Z_{\upnu}(\uptheta )$.  The last claim in the statement of the Proposition follows from the fact that $\uptheta\in\R$
admits $0\not=\bast n\in\bast\Z (\uptheta )$ with $\upvarepsilon (\bast n)=0$ $\Leftrightarrow$ $\uptheta\in\Q$.

\end{proof}

The duality map
$\bast n\mapsto \bast n^{\perp}$ defines an isomorphism 
\begin{align}\label{dualityiso} \perp : \bast\Z (\uptheta )\longrightarrow \bast \Z(\uptheta^{-1} )
\end{align}
for all $\uptheta\not=0$.

\begin{prop} Let $\uptheta\not=0$.  Then the duality isomorphism (\ref{dualityiso}) respects the fine growth-decay tri-filtration: 
\[   \bast\Z^{\upmu[\upiota]}_{\upnu} (\uptheta )^{\perp} = \bast\Z^{\upmu[\upiota]}_{\upnu} (\uptheta^{-1})   \]
\end{prop}

\begin{proof}  Note that 
\begin{equation}\label{dualordmag} \bast n\cdot\upmu <\upiota \Leftrightarrow \bast n^{\perp}\cdot\upmu <\upiota
\end{equation}
 which implies that duality respects the fine growth bi-filtration.  On the other hand, $\upvarepsilon (\bast n^{\perp}) = -\uptheta^{-1}\upvarepsilon (\bast n)$
so the decay filtration is preserved as well.
\end{proof}

Recall \cite{La1} that $\uptheta$ is {\bf projective linear equivalent} to $\upeta$ if there exists $A\in {\rm PGL}_{2}(\Z )$ such that $A(\uptheta )=\upeta$.   The relation of projective linear equivalence is denoted in this paper by:
\[  \uptheta\Bumpeq\upeta .\]

\begin{theo}\label{triisomorphism} If $\uptheta\Bumpeq\upeta$ by $A\in {\rm PGL}_{2}(\Z )$, then $A$ induces an isomorphism
of $\bast\Z (\uptheta )$ with $\bast\Z (\upeta )$ preserving the fine growth-decay tri-filtration.
\end{theo}

\begin{proof}  The isomorphism is induced by the matrix action of a linear representative $A=\left(\begin{array}{cc}
a & b \\
c& d
\end{array}      \right)$ on pairs $(\bast n^{\perp},\bast n)$ where $\bast n\in\bast\Z (\uptheta )$ and $\uptheta\cdot\bast n\simeq\bast n^{\perp}$.
That is, $A(\bast n) = c\bast n^{\perp}+d\bast n$ and $A(\bast n^{\perp}) = a\bast n^{\perp}+b\bast n$.
By (\ref{dualordmag}), $\bast n\in\bast\Z^{\upmu[\upiota]}$ $\Leftrightarrow$ $\bast n^{\perp}\in\bast\Z^{\upmu[\upiota]}$.  It follows then
that $\bast n\in\bast\Z^{\upmu[\upiota]}$ $\Leftrightarrow$ $A(\bast n) \in\bast\Z^{\upmu[\upiota]}$.  On the other hand, 
\begin{align*}
\upeta\cdot A(\bast n)-A(\bast n^{\perp}) & = 
\frac{1}{c\uptheta+d}\bigg\{(a\uptheta+b)\big(c\bast n^{\perp}+d\bast n\big)-(c\uptheta+d)\big(a\bast n^{\perp}+b\bast n \big)  \bigg\} \\
& = \frac{1}{c\uptheta+d} \big(\uptheta\bast n-\bast n^{\perp}\big) \\
& = \frac{\upvarepsilon (\bast n)}{c\uptheta+d} .
\end{align*}
Therefore: $\bast n\in\bast\Z_{\upnu} (\uptheta )$ $\Leftrightarrow$ $A(\bast n )\in\bast\Z_{\upnu} (\upeta )$.
\end{proof}

\section{Nonvanishing Spectra}\label{nonvanspec}

The nontriviality of the group $\bast\Z^{\upmu}_{\upnu}(\uptheta ) $ 
for specific indices $\upmu,\upnu\in\bstar\PR\R_{\upvarepsilon}$ depends intimately on the type of $\uptheta$. 
We define the {\bf nonvanishing spectrum}
to be the subset
\[ {\rm Spec}(\uptheta ) = \left\{ (\upmu ,\upnu )\middle|\;  \bast\Z^{\upmu}_{\upnu}(\uptheta )\not=0 \right\}\subset 
\bstar\PR\R_{\upvarepsilon}^{2}.   \]
In this section, we will characterize the spectra of a real
number according to its ``linear classification'' (rational, badly approximable, well approximable, etc.).  We begin with some very
general results.

\begin{prop}\label{isospectral}  If $\uptheta\Bumpeq\upeta$ then ${\rm Spec}(\uptheta )={\rm Spec}(\upeta )$.
\end{prop}

\begin{proof}  This follows immediately from Theorem \ref{triisomorphism}.
\end{proof}

\begin{theo}\label{gennonvan}  For all $\uptheta\in\R$ and $\upmu <\upnu$, $\bast\Z^{\upmu}_{\upnu}(\uptheta )\not=0$.
\end{theo}

\begin{proof}  By Proposition \ref{denselinord},
we may find $\uprho$ with $\upmu <\uprho <\upnu$; and by Proposition \ref{integerrepre},
$\uprho=\langle\bast N^{-1}\rangle$ for some
$\bast N\in\bast\Z_{+} - \Z_{+}$.  By the Uniform Dirichlet Theorem\footnote{For any real number $N>1$, there exist $p,q\in\Z$ with $1\leq q<N$
such that $|q\uptheta - p |<1/N$. See \cite{Wa}.} 
 there is $\bast n\in\bast\Z (\uptheta )$
such that 
$  | \upvarepsilon (\bast n )|<\bast N^{-1}  $
where $\bast n <\bast N$.  Therefore, 
$  | \upnu (\bast n)| \leq \upnu $.  
On the other hand
$  \bast n\cdot \upmu \leq\bast N\cdot \upmu \in \bstar\PR\R_{\upvarepsilon} $
since $\upmu <\uprho$, so $\bast n\in\bast\Z^{\upmu} (\uptheta )$
\end{proof}

The set $\{(\upmu,\upnu )|\;\upmu<\upnu\}\subset {\rm Spec}(\uptheta )$ is called the {\bf slow component}.

For $\uptheta\in\R$, denote by $\{ a_{i}=a_{i}(\uptheta )\}$, $i=0,1,\dots$, the sequence of its partial quotients \cite{La1}: an infinite sequence
 $\Leftrightarrow$ $\uptheta\not\in\Q$.  As is the custom, we write 
 \[ \uptheta= [a_{0}a_{1}\dots ].\]
The sequence $\{ q_{i}\}$ of {\bf best denominators} of $\uptheta$ is defined recursively by
the formula 
\[ q_{i+1}=a_{i+1}q_{i}+q_{i-1},\quad q_{0}=1, \; q_{1}=a_{1} .\]   Similarly, the sequence 
$\{ p_{i}\}$
of {\bf best numerators} is defined 
\[ p_{i+1}=a_{i+1}p_{i}+p_{i-1}, \quad p_{0}=a_{0},\; p_{1}=a_{1}a_{0}+1.\]  
We have (e.g. see Theorem 5 of \cite{La1})
\begin{align}\label{bestineq} q_{i}|q_{i}\uptheta -p_{i}|<q_{i}^{-1}.\end{align}
The sequence of quotients $\{p_{i}/q_{i}\}$ is called the sequence of
{\bf best approximations} (or {\bf principal convergents}) of $\uptheta$: by (\ref{bestineq}) they satisfy 
$p_{i}/q_{i}\rightarrow \uptheta$. See \cite{Ca}, \cite{La1}, \cite{Sch}.

Consider now a sequence 
 $\{ q_{n_{i}}\}$ in which $q_{n_{i}}$ is the $n_{i}$th best denominator of $\uptheta$, where $n_{i}\leq n_{i+1}$ for all $i$ and
 $n_{i}\rightarrow\infty$. 
 By (\ref{bestineq}) the associated sequence class defines an element
 \[ \bast \widehat{q}:=\bast \{ q_{n_{i}}\}\in \bast\Z (\uptheta )\] called
a {\bf best denominator class}, and 
the classes 
\[  \widehat{\upmu}:=\upmu (\bast  \widehat{q}),\quad \text{resp.}\quad
  \widehat{\upnu} := \upnu (\bast  \widehat{q})\] will be referred to
as the associated {\bf best growth} resp.\  {\bf best decay} of $\bast  \widehat{q}$.  We will denote by
$\bast \widehat{q}^{+}$ resp.\ $\bast \widehat{q}^{-}$ the classes of the successor and predecessor sequences $\bast\{q_{n_{i}+1}\}$ resp.\ $\bast\{q_{n_{i}-1}\}$,
with a similar notation employed for the associated best growth and best decay classes e.g.\ $\widehat{\upmu}^{+}$ = the growth
class of $\bast \widehat{q}^{+}$.  Note that $\widehat{\upmu}^{+}\leq \upmu$
and $\widehat{\upnu}^{+}\leq \upnu$.  The above terminology applies without change to the corresponding sequence
of best numerators $\{ p_{n_{i}}\}$, yielding the associated best numerator class $\bast \widehat{p}$ and its best growth.

\begin{prop}  Let $ \bast \widehat{q}$ be a best denominator class, $\bast \widehat{p}$ the corresponding
best numerator class. 
Then $ \bast \widehat{q}^{\perp}=\bast \widehat{p}\in\bast\Z (\uptheta^{-1})$.  In particular, the best growth $\widehat{\upmu}$
of $\bast \widehat{q}$ is also the best growth of $\bast \widehat{p}$.  
\end{prop}

\begin{proof}  That $ \bast \widehat{q}^{\perp}=\bast \widehat{p}$ follows from (\ref{bestineq}).  Since $ \bast \widehat{q}\uptheta-\bast \widehat{p}=\upvarepsilon (\bast \widehat{q})$, the best growth  class of $\bast \widehat{p}$ coincides with that of $ \bast \widehat{q}$.
\end{proof}

\begin{note} When $\uptheta\in\Q$, the sequence of best approximations is finite, so every best approximation class $\bast \widehat{q}$ is standard and equal to one
of the $q_{i}$.  In this case, every best growth is $  \widehat{\upmu}=1$ and every best decay is $  \widehat{\upnu}=-\infty$. 
\end{note}

For $\uptheta\in\R-\Q$, we denote by:
\begin{itemize}
\item $\bast\Z_{\rm b}(\uptheta )$ 
the set of best denominator 
classes.
\item $ \bstar\PR\R_{\upvarepsilon}^{\rm bg}(\uptheta) $ ($ \bstar\PR\R_{\upvarepsilon}^{\rm bd}(\uptheta) $) the set of best growths (best decays) of best denominator classes.
\end{itemize}   

\begin{prop} For $\uptheta\in\R-\Q$, $\bstar\PR\R_{\upvarepsilon}^{\rm bg}(\uptheta)$ is closed in the order topology.
\end{prop}

\begin{proof} If $\bstar\PR\R_{\upvarepsilon}^{\rm bg}(\uptheta)=\bstar\PR\R_{\upvarepsilon}$ we are done, so suppose
otherwise.  Given $\upmu\in \bstar\PR\R_{\upvarepsilon}-\bstar\PR\R_{\upvarepsilon}^{\rm bg}(\uptheta)$, we will construct an interval $(\upmu',\upmu'')\ni\upmu$ containing
no elements of $ \bstar\PR\R_{\upvarepsilon}^{\rm bg}(\uptheta)$.
Let $\bast x\in \upmu^{-1}$.  Then there exists a largest
$\bast \widehat{q}$ for which $\bast x>\bast \widehat{q}$: indeed,
if we choose $\{ x_{i}\}\in\bast x$ non-decreasing and let $q_{n_{i}}$
be the largest member of $\{ q_{i}\}$ which is less than $x_{i}$,
then $\bast \widehat{q} = \bast \{ q_{n_{i}}\}$ works.
Since $\upmu\not\in  \bstar\PR\R_{\upvarepsilon}^{\rm bg}(\uptheta)$, there exists $\bast r$ infinite with $\bast r\cdot \bast \widehat{q} = \bast x$.  Now let
$\bast s\in\bast\R_{+}$ be such that both $\bast s $ and $\bast r/\bast s$ are infinite, and let $\bast y = (\bast r/\bast s)\cdot\bast \widehat{q}$.
If we denote by $\upmu'$ the class of $\bast y^{-1}$ then $\widehat{\upmu}>\upmu'>\upmu$ and $[\upmu,\upmu']\cap  \bstar\PR\R_{\upvarepsilon}^{\rm bg}(\uptheta)=\emptyset$.  In the same way, we may produce 
$\upmu''<\upmu$ with $[\upmu'',\upmu]\cap \bstar\PR\R_{\upvarepsilon}^{\rm bg}(\uptheta)=\emptyset$.  Thus
$(\upmu'',\upmu')$ is the sought after interval.
\end{proof}

The following result is our first vanishing theorem: a straightforward reinterpretation of the quality of being a
best denominator class in terms of the growth-decay bi-filtration.

\begin{theo}\label{bestfilt}  Let $\uptheta\in\R-\Q$ and let $\bast \widehat{q}$ be
any best denominator class with associated growth and decay $\widehat{\upmu},\widehat{\upnu}$.  Then for all $\upmu \geq  \widehat{\upmu} $ and $\upnu < \widehat{\upnu}$, 
$ \bast\Z^{\upmu}_{\upnu}(\uptheta ) = 0 $.  
\end{theo}

\begin{proof}  For $\upmu \geq  \widehat{\upmu} $ and $\upnu < \widehat{\upnu}$, suppose there exists a non-zero $\bast n\in \bast\Z^{\upmu}_{\upnu}(\uptheta )$,
which we may assume is positive.
Then $\bast n\cdot \widehat{\upmu}\leq \bast n\cdot \upmu\in \bstar\PR\R_{\upvarepsilon}$
implies that $\bast n <\bast \widehat{q}$.  In turn, the latter implies, since $\bast \widehat{q}$ is the class
of a non decreasing sequence of best denominators of $\uptheta$, that
\[ |\upvarepsilon (\bast n)| =  |\uptheta \bast n -\bast n^{\perp}| \geq  |\uptheta \bast \widehat{q}-\bast \widehat{q}^{\perp}| = |\upvarepsilon (\bast \widehat{q})|. \]
From this we derive $\upnu (\bast n )\geq \widehat{\upnu} > \upnu$, contradiction.
\end{proof}

In the $(\upmu ,\upnu )$-plane the coordinates belonging to the right-infinite horizontal strip
\[ \widehat{R} =\{(\upmu ,\upnu)|\; \upmu \geq  \widehat{\upmu}, \;\upnu < \widehat{\upnu}\}\]
 give parameters where the groups
$ \bast\Z^{\upmu}_{\upnu}(\uptheta )$ vanish.  We call $\widehat{R}$ a {\bf vanishing strip}.
See the graph labeled  ``generic irrational'' in Figure \ref{portraits}.

\begin{figure}[htbp]\label{Spectralportraits}
\centering
\includegraphics[width=5in]{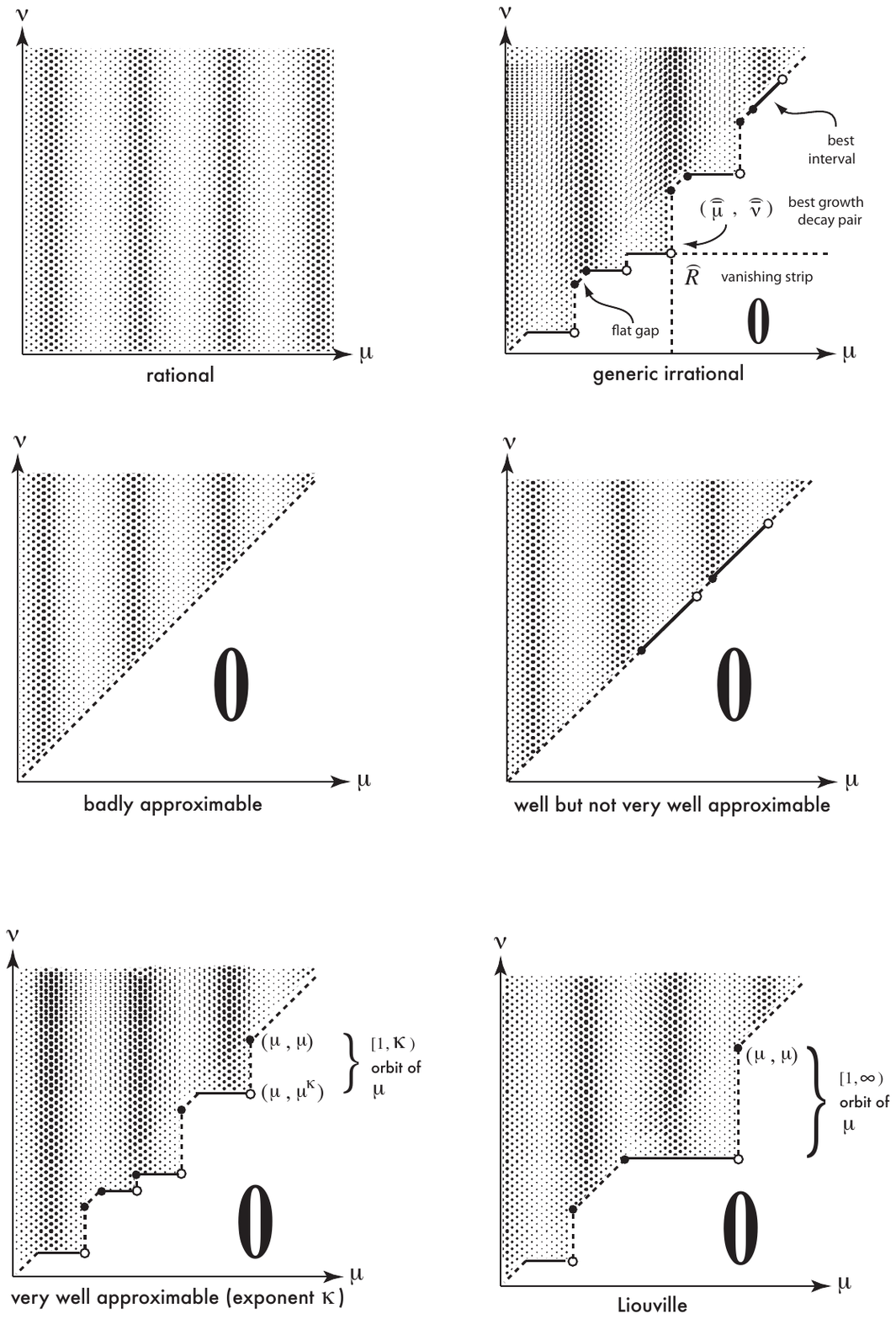}
\caption{Portraits of Spectra.  Shaded regions and heavy lines represent nonvanishing.}\label{portraits}
\end{figure}

We now give a spectral characterization of the linear classification of real numbers.

\begin{prop}\label{rationalspec}  $\uptheta\in\Q$ $\Leftrightarrow$ ${\rm Spec}(\uptheta ) = \bstar\PR\R_{\upvarepsilon}^{2}$.
\end{prop}

\begin{proof}  If $\uptheta\in\Q$ then for all $\upnu$, $\bast\Z_{\upnu}(\uptheta )=\bast\Z_{-\infty}(\uptheta )
=\bast\Z(\uptheta )$,
so $\bast\Z^{\upmu}_{\upnu}(\uptheta ) =\bast\Z^{\upmu}(\uptheta ) \not=0$
for all $\upmu, \upnu$.  On the other hand, if $\uptheta\in\R-\Q$ then by Theorem \ref{bestfilt}, 
${\rm Spec}(\uptheta ) \subsetneq \bstar\PR\R_{\upvarepsilon}^{2}$.
\end{proof}

Recall that $\uptheta\in\R-\Q$ is {\bf badly approximable} if 
\[   \lim_{n\rightarrow\infty}\inf \;n\| n\uptheta\| >0 ,\] 
where $\|\cdot\|$ is the distance-to-the-nearest-integer function.  Or equivalently, if there exists a real number $C>0$ such that for all $0\not=\bast n\in \bast\Z(\uptheta)$, 
\[   |\bast n|\cdot |\upvarepsilon (\bast n)| \geq C .\] 
The set $\mathfrak{B}$ of badly approximable numbers has cardinality the continuum \cite{Sch}.

\begin{theo}\label{badapproxchar} The following statements are equivalent:
\begin{enumerate}
\item[i.] $\uptheta\in\mathfrak{B}$.
\item[ii.] $\bast\Z^{\upmu}_{\upnu}(\uptheta )=0$
for all $\upmu\geq \upnu$. 
\item[iii.] $\bstar\PR\R_{\upvarepsilon}^{\rm bg}(\uptheta)=\bstar\PR\R_{\upvarepsilon}$.  In particular, for all $\widehat{\upmu}\in 
\bstar\PR\R_{\upvarepsilon}^{\rm bg}(\uptheta)$, $\widehat{\upmu}^{+}=\widehat{\upmu}$.
\item[iv.] $\widehat{\upmu} = \widehat{\upnu}$ for every best growth decay pair: that is $\bstar\PR\R_{\upvarepsilon}^{\rm bg}(\uptheta)=\bstar\PR\R_{\upvarepsilon}^{\rm bd}(\uptheta)$.
\end{enumerate}
\end{theo}

\begin{proof} i.\ $\Rightarrow$ ii. If $\uptheta\in\mathfrak{B} $ then for all non-zero $\bast n\in \bast\Z (\uptheta )$ we have
$|\bast n|\cdot |\upvarepsilon (\bast n )|\geq C$.  If there exists $\upmu\geq\upnu$ with 
$0\not=\bast n\in \bast\Z^{\upmu}_{\upnu}(\uptheta )$ then in $\bstar\PR\R_{\upvarepsilon}$,
\[ |\bast n|\cdot \upmu\geq | \bast n|\cdot \upnu \geq |\bast n|\cdot \upnu (\bast n )\geq 1 = \text{ the $\bstar\PR\R$-class of $C$},\]  
implying that $|\bast n|\cdot \upmu\not\in \bstar\PR\R_{\upvarepsilon}$ and $\bast n\not\in \bast\Z^{\upmu}$.  ii.\ $\Rightarrow$ i.  If $\bast\Z^{\upmu}_{\upnu}(\uptheta )=0$
for all $\upmu\geq \upnu,$ then for each $\bast n\in \bast\Z (\uptheta )$,
$|\bast n|\cdot|\upvarepsilon (\bast n)| \geq \updelta >0$ where $\updelta\in\R$. 
We can choose delta uniformly: if not, then by a diagonal sequence argument
we could produce an element $\bast\Z (\uptheta )$ for which $|\bast n|\cdot\upvarepsilon (\bast n)$ is infinitesimal (i.e.\ we could produce a non trivial element of
$\bast\Z^{\upnu(\bast n)}_{\upnu(\bast n)}(\uptheta )$)
violating the hypothesis.     i.\ $\Rightarrow$ iii.  $\uptheta\in\mathfrak{B}$  $\Leftrightarrow$ the partial quotients $a_{i}$
are uniformly bounded $\Leftrightarrow$ the successive
ratios of best denominators $q_{i+1}/q_{i}$ are uniformly bounded.  
Now given $\upmu\in \bstar\PR\R_{\upvarepsilon}$ let $\{n_{i}\}\subset\N_{+}$ represent
$\upmu^{-1}$.  For each $i$ let $q_{k_{i}}$ be the largest best denominator with $q_{k_{i}}\leq n_{i}$ so that $n_{i}<q_{k_{i}+1}$.
By hypothesis there exists a constant $B>1$ so that $q_{k_{i}+1}<Bq_{k_{i}}$.  It follows that $n_{i}= b_{i}q_{k_{i}}$ with $1\leq b_{i}<B$.
Then the growth of the class $\bast \widehat{q}$ is equal to $\upmu$.
iii.\ $\Rightarrow$ i. If 
$\uptheta\not\in\mathfrak{B}$, choose $\widehat{\upmu}$ so that if $\{ q_{n_{i}}\}$ represents $\widehat{\upmu}^{-1}$ then $q_{n_{i}+1}/q_{n_{i}}$
is monotone and unbounded.  Then the successor sequence
$\{q_{n_{i}}^{+}=q_{n_{i}+1}\}$ defines a distinct element $\widehat{\upmu}^{+}\in\bstar\PR\R_{\upvarepsilon}^{\rm bg}(\uptheta)$ with $\widehat{\upmu}^{+}<\widehat{\upmu}$.  It follows that 
$\bstar\PR\R_{\upvarepsilon}^{\rm bg}(\uptheta)\not= \bstar\PR\R_{\upvarepsilon}$.  i.\ $\Leftrightarrow$ iv.
From Dirichlet's Theorem and the definition of $\mathfrak{B}$, $\widehat{\upmu}=\widehat{\upnu}$  for every best growth decay pair $\Leftrightarrow$ 
$1>|\bast \widehat{q}|\cdot \upvarepsilon (\bast \widehat{q}) \geq C $  for every best denominator class
for some uniform $C>0$ $\Leftrightarrow$ $\uptheta\in\mathfrak{B}$.
\end{proof}

 Recall that $\uptheta\in\R-\Q$ which is not badly approximable is called {\bf well approximable}: that is,
$ \lim_{n\rightarrow\infty}\inf \;n\| n\uptheta\| =0$.  We denote the set
of well approximable numbers by 
\[ \mathfrak{W}=\R-(\Q\cup\mathfrak{B}).\]

\begin{theo}\label{wellapprox} 
Let $\uptheta\in\R-\Q$.  The following statements are equivalent:
\begin{enumerate}
\item[i.] $\uptheta\in\mathfrak{W}$.
\item[ii.] There exists
$\upmu\in\bstar\PR\R_{\upvarepsilon}$ such that $\bast\Z^{\upmu}_{\upmu}(\uptheta  )\not =0$.
\item[iii.] There exists $\widehat{\upmu}\in 
\bstar\PR\R_{\upvarepsilon}^{\rm bg}(\uptheta)$ such that $\widehat{\upmu}^{+}<\widehat{\upmu}$.
In particular, $\bstar\PR\R_{\upvarepsilon}^{\rm bg}(\uptheta)$ is not a dense order.   
\item[iv.] $\widehat{\upmu}>\widehat{\upnu}$ for some best growth best decay pair.
\end{enumerate}
\end{theo}

\begin{proof} 
i.\ $\Leftrightarrow$ ii. By Theorem \ref{badapproxchar},
$\uptheta\in\mathfrak{B}$ $\Rightarrow $ $\bast\Z^{\upmu}_{\upmu}(\uptheta )=0$ for all $\upmu$.
On the other hand if $\uptheta\in\mathfrak{W}$ $\Rightarrow $ there exists $\upmu\geq\upnu$ with 
$\bast\Z^{\upmu}_{\upnu}(\uptheta )\not=0$.  By the order-reversing property of the growth filtration, the latter
implies that $\bast\Z^{\upnu}_{\upnu}(\uptheta )\not=0$. i.\ $\Leftrightarrow$ iii.  If $\uptheta\in\mathfrak{W}$
then there exists a best denominator sequence $\{q_{n_{i}}\}$ for which the sucessor $\{q^{+}_{n_{i}}=q_{n_{i}+1}\}$
satisfies $q_{n_{i}}^{+}/q_{n_{i}}\rightarrow\infty$.  The other direction follows from Theorem \ref{badapproxchar}.
iii.\ $\Leftrightarrow$ iv.  Immediate from Theorem \ref{badapproxchar}.
\end{proof}

Let $\upkappa\geq1$.  Recall \cite{BD} that $\uptheta$ is {\bf  $\boldsymbol\upkappa$-approximable} 
if the set of $n\in\N$ for which
$\|n\uptheta \| <n^{-\upkappa} $
has infinite cardinality i.e.
\begin{align}\label{verywellaproxexpk} \lim_{n\rightarrow\infty}\inf \;n^{\upkappa}\| n\uptheta\| <1.
\end{align}
The set of $\upkappa$-approximable numbers is denoted $\mathfrak{W}_{\upkappa}$.
By Dirichlet's Theorem $\mathfrak{W}_{1}=\R-\Q$.
 If $\uptheta$ is $\upkappa$-approximable for $\upkappa >1$ then
we say that $\uptheta$ is {\bf very well approximable}; the set of such numbers is denoted
$\mathfrak{W}_{>1}=\bigcup_{\upkappa>1}\mathfrak{W}_{\upkappa}$.  The inclusion $\mathfrak{W}_{>1}\subset\mathfrak{W}$ is proper
and we write $\mathfrak{W}_{1^{+}} = \mathfrak{W}-\mathfrak{W}_{>1}=$ the set of {\bf well but not very well approximable numbers}.
For $\uptheta\in\mathfrak{W}_{>1}$
its {\bf exponent}\footnote{Equal to the {\it asymptotic irrationality exponent} defined in \cite{Wa}.} is 
\[ \upkappa (\uptheta ) := \sup_{\uptheta\in \mathfrak{W}_{\upkappa}} \upkappa \in (1,\infty ] .\]
It is not necessarily the case that $\uptheta$ is $\upkappa (\uptheta )$-approximable.
We say that the exponent $\upkappa = \upkappa (\uptheta )$ of 
$\uptheta\in\mathfrak{W}_{>1}$ is 
 {\bf excellent} if  $ \lim_{n\rightarrow\infty}\inf \;n^{\upkappa}\| n\uptheta\| =0$ 
 In particular, if the exponent $\upkappa$ of $\uptheta$ is excellent then $\uptheta\in\mathfrak{W}_{\upkappa}$. 
 If the exponent is not excellent, we will say that it is {\bf  bad} and say that $\uptheta$ is {\bf  $\boldsymbol\upkappa$-bad}.

\begin{theo}\label{verywellapprox}  Let $\uptheta\in\mathfrak{W}_{>1}$.  The following statements are equivalent:
\begin{enumerate}
\item[i.] $\uptheta$ has (excellent) exponent $\upkappa>1$. 
\item[ii.] 
The following conditions hold: 
\begin{itemize}
\item[a.]  For all $\upkappa'>\upkappa$, $\bast\Z^{\upmu}_{\upmu^{\upkappa'}}(\uptheta  ) =0$.
\item[b.] 
$ \bigcap_{\uplambda\in [1,\upkappa )}\bast\Z^{\upmu}_{\upmu^{\uplambda}}(\uptheta  )\not =0$ ($\bast\Z^{\upmu}_{\upmu^{\upkappa}}(\uptheta )\not=0$).  In particular, $\bast\Z^{\upnu}_{\upnu}(\uptheta  )\supset\bast\Z^{\upmu}_{\upnu}(\uptheta  )\not =0$ for all $\upmu\geq \upnu>\upmu^{\upkappa}$.
\end{itemize} 
\item[iii.] The following conditions hold: 
\begin{itemize}
\item[a.] For all $\upkappa'>\upkappa$ and for every growth decay pair $(\widehat{\upmu},\widehat{\upnu})$, $\widehat{\upmu}^{\upkappa' }\leq\widehat{\upnu}$.
\item[b.]  There exists a growth decay pair $(\widehat{\upmu},\widehat{\upnu})$ such that $\widehat{\upmu}^{\uplambda }>\widehat{\upnu}$
for all $1\leq\uplambda<\upkappa$ (such that $\widehat{\upmu}^{\upkappa }>\widehat{\upnu}$).
\end{itemize}
\end{enumerate}
\end{theo}

\begin{proof} i.\ $\Rightarrow$ ii. $\uptheta$ has exponent $\upkappa>1$ $\Leftrightarrow$ 
\begin{itemize} 
\item[a'.]  For all $\bast m\in\bast\Z (\uptheta )$ and all $\upkappa'>\upkappa$,
$\upnu (\bast m )> \upmu(\bast m)^{\upkappa'}$.
\item[b'.] there exists
$\bast n\in\bast\Z (\uptheta )$ such that $\upnu (\bast n )< \upmu(\bast n)^{\uplambda}$ for all $1\leq\uplambda<\upkappa$.
\end{itemize}
Properties a'., b'. together are equivalent to properties a., b. of ii.  Indeed, the vanishing $\bast\Z^{\upmu}_{\upmu^{\upkappa'}}(\uptheta  ) =0$ for all $\upkappa'>\upkappa$
is equivalent to a'. .  Assuming b'.:
since $\bstar\PR\R_{\upvarepsilon}$ is a dense linear order, we may find $\upmu\in \bstar\PR\R_{\upvarepsilon}$ such that 
$\upnu (\bast n )< \upmu^{ \uplambda} < \upmu(\bast n)^{\uplambda}$
for all $1\leq\uplambda<\upkappa$.
Then $\upmu < \upmu(\bast n)$ and therefore 
\begin{align}\label{verywellintersection} \bast n\in \bast\Z^{\upmu}_{\upnu (\bast n )}(\uptheta )
\subset  \bigcap_{\uplambda\in [1,\upkappa )} \bast\Z^{\upmu}_{\upmu^{\uplambda}}(\uptheta ).\end{align}
On the other hand, the existence of $\bast n$ satisfying (\ref{verywellintersection}) implies b'..

i.\ $\Rightarrow$ iii. We show that a'., b'. $\Leftrightarrow$ a., b. of iii.  Assuming a'., b'.  Let $\bast m$ be as in a'. and let $\bast \widehat{q}$
be the largest best denominator class $\leq\bast m$.  If  $\bast m$ is itself a best denominator class we are done, so assume otherwise.
Then 
\[ \widehat{\upmu }^{\uplambda} \geq \upmu(\bast m)^{\uplambda} >\upnu(\bast m)\geq \widehat{\upnu}. \]
We leave the excellent versions to the reader.
\end{proof}

The {\bf Liouville numbers} are those which are very well approximable for every exponent $\upkappa >1$; they are
denoted $\mathfrak{W}_{\infty}$.  For any $\upmu\in\bstar\PR\R_{\upvarepsilon}$, write $(\upmu^{\infty},\upmu]=\bigcup_{\upkappa>1}
(\upmu^{\upkappa},\upmu]$.
The next result follows immediately from Theorem \ref{verywellapprox}: we recall from \S \ref{tropical} that $\bar{\upmu}$ is the orbit
of $\upmu$ with respect to the Frobenius action of $(\bast\R_{\rm fin})_{+}^{\times}$.  For $\widehat{\upmu},\widehat{\upnu}$ a best
growth-decay pair, the corresponding Frobenius orbits are denoted $\widehat{\bar{\upmu}},\widehat{\bar{\upnu}}$.

\begin{coro}\label{LiouvilleSpec}
Let $\uptheta\in\R-\Q$.  The following statements are equivalent: 
\begin{enumerate}
\item[i.] $\uptheta\in\mathfrak{W}_{\infty}$. 
\item[ii.] There exists $\upmu\in\bstar\PR\R_{\upvarepsilon}$ 
such that 
\[ \bigcap_{\uplambda\in [1,\infty )}\bast\Z^{\upmu}_{\upmu^{\uplambda}}(\uptheta  )\not =0.\] 
In particular, $\bast\Z^{\upnu}_{\upnu}(\uptheta  )\not =0$ for all $\upnu\in (\upmu^{\infty}, \upmu]$.
\item[iii.] $\widehat{\bar{\upmu}}>\widehat{\bar{\upnu}}$ for some best growth decay pair.
\end{enumerate}
\end{coro}

\section{Flat Spectra}\label{flatsection}

The line $\upmu=\upnu$ represents a critical divide whose intersection with ${\rm Spec}(\uptheta )$ gives a new invariant
of $\uptheta$
which is strongly influenced by patterns found in the sequence of partial quotients; as opposed to the full spectrum
which is essentially determined by the exponent.

We define the {\bf flat spectrum} of $\uptheta$ to be the set
\[ {\rm Spec}_{\rm flat}(\uptheta ) =\left\{\upmu\in\bstar\PR\R_{\upvarepsilon}\, |\;  \bast\Z^{\upmu}_{\upmu}(\uptheta )\not=0 \right\}.\]
By Proposition \ref{rationalspec}, ${\rm Spec}_{\rm flat}(\uptheta )=\bstar\PR\R_{\upvarepsilon}$ for all $\uptheta\in\Q$, and by Theorem \ref{badapproxchar}, ${\rm Spec}_{\rm flat}(\uptheta )=\emptyset$ for all $\uptheta\in\mathfrak{B}$.  Therefore we will restrict attention in this section to $\uptheta\in\mathfrak{W}$ = the set of well approximable numbers.  

Suppose that $\bast m\in\bast\Z (\uptheta )$
can be factored 
\[ \bast m=\bast x\cdot\bast n\]  for $\bast x\in\bast\Z$ and $\bast n \in \bast\Z (\uptheta )$.  If $\upnu=\upnu (\bast n )$ and 
$\bast x\in \bast\Z^{\upnu}$ then it follows that 
\[ \upvarepsilon(\bast m) = \bast x\cdot\upvarepsilon (\bast n),\quad \text{i.e.} \quad \upnu (\bast m ) =|\bast x|\cdot\upnu.\]  
In this case we refer to $\bast m=\bast x\cdot\bast n$ as an {\bf approximate ideal factorization} and speak of $\bast m$ as being an {\bf approximate multiple}
of $\bast n$.  Note that in this case $\bast m^{\perp}=\bast x\cdot\bast n^{\perp}$. Conversely, the action 
\[ \bast\Z^{\upnu}\times\bast\Z_{\upnu}(\uptheta )\rightarrow \bast\Z (\uptheta ), \quad (\bast x,\bast n)\mapsto \bast x\cdot\bast n, \]
 has image consisting of approximate multiples.  There may be
factorizations of $\bast m$ which do not consist of approximate factors.

Call $\bast m\in\bast\Z (\uptheta )$ a {\bf multibest denominator} if there is an approximate factorization
\[\bast m = \bast x\cdot\bast\widehat{q}\] for some
best denominator $\bast\widehat{q}$.  
We have 
\begin{align}\label{multibestineq}
\widehat{\upnu}\leq\upnu (\bast m ) = \bast x\cdot\widehat{\upnu} ,\quad \upmu(\bast m) =\bast x^{-1}\cdot\widehat{\upmu}\leq\widehat{\upmu}.
 \end{align}
Recall that if $\bast\widehat{q}$ is the class of $\{ q_{n_{i}}\}$, we denote by $\bast\widehat{q}^{+}$ the successor class 
  $\bast\{ q_{n_{i}+1}\}$ and by $\bast\widehat{q}^{-}$ the predecessor class 
  $\bast\{ q_{n_{i}-1}\}$.  Then if $\bast a$ denotes the class of $\{a_{n_{i}+1}\}$ (the corresponding sequence of partial quotients), we have
 \[\bast\widehat{q}^{+}=\bast a\cdot \bast \widehat{q}+\bast \widehat{q}^{-}.\]  We say that $\bast\widehat{q}$
 has {\bf infinite partial quotient} if $\bast a$ is infinite.  Note that $\bast\widehat{q}$ has infinite partial quotient $\Leftrightarrow$
 $\widehat{\upmu}^{+}<\widehat{\upmu}$.
 We apply the same terminology to
 a multibest class $\bast m=\bast x\cdot\bast\widehat{q}$
 if $\bast\widehat{q}$ has infinite partial quotient.  An element $\uptheta\in\mathfrak{W}$ will
possess best classes $\bast\widehat{q}$ having {\em finite} partial quotient precisely
 when the sequence $\{ a_{i}\}$ of partial quotients has an infinite bounded subsequence e.g.\ when  
 $\uptheta =e=[2,1,2,1,1,4,1,1,6,\dots ]$.

\begin{theo}\label{flatmultibest} Let $\uptheta\in\mathfrak{W}$.  Then 
\begin{align}\label{flatcharacterization} \bast\Z^{\upmu}_{\upmu}(\uptheta ) & =\bigg\{ \bast m\text{ {\rm multibest with infinite partial quotient} }\bigg|\; 
\upnu (\bast m ) \leq \upmu<
\upmu (\bast m )\bigg\}.
\end{align}
\end{theo}

\begin{proof}  We may assume without loss of generality that $\uptheta>0$.  Clearly the set on the right hand side of (\ref{flatcharacterization}) is contained in $\bast\Z^{\upmu}_{\upmu}(\uptheta )$.  
Suppose now that $0\not=\bast m\in \bast\Z^{\upmu}_{\upmu}(\uptheta )$.
Then we may find monotone representatives $\{ s_{i}\} $ of $\upmu$ and $\{ m_{i}\}$ of $\bast m$ for which
$m_{i}s_{i}\rightarrow 0$ and $\| m_{i}\uptheta\|\leq s_{i}$.  From the first fact we get $s_{i}<m_{i}^{-1}$ so that we may write
\begin{align}\label{upbound} 
\| m_{i}\uptheta\|\leq s_{i}<m_{i}^{-1}.
\end{align}  In particular, if $\{m_{i}^{\perp}\}$ is a representative of the dual $\bast m^{\perp}$ then $|\uptheta-(m_{i}^{\perp}/m_{i})|<m_{i}^{-2}$.  By Grace's Theorem (Theorem 10 of \cite{La1}), the $m_{i}^{\perp}/m_{i}$ are (intermediate) convergents of $\uptheta$: that is, for each $i$ there exists $n=n_{i}$, $r=r_{i}$ nonnegative integers with
$m_{i}/m_{i}^{\perp}= p_{n,r}/q_{n,r}$, 
where
\[   p_{n,r} := rp_{n+1} + p_{n}, \quad q_{n,r} := rq_{n+1} + q_{n},\]
$\{ p_{i}\}$, $\{ q_{i}\}$ are the sequences of numerators and denominators of the principal convergents (best approximations) of $\uptheta$ and $0\leq r\leq a_{n+2}-1$.   Moreover, Grace's Theorem further affirms that the possible values of $r$ in this case are $r=0,1$ or $a_{n+2}-1$.   
If we denote by $\bast\widehat{q}$ the class of $\{q_{n_{i}}\}$, $\bast\widehat{q}^{+}$ the class of $\{q_{n_{i}+1}\}$ and by
$\bast r$ the class of $\{ r_{i}\}$ then we may write
\[ \bast m=\bast x\cdot\bast\widehat{q}_{\bast r}:=\bast x\cdot (\bast r\bast\widehat{q}^{+}+\bast\widehat{q}),\quad
\bast m^{\perp}=
\bast x\cdot\bast\widehat{p}_{\bast r}:= \bast x\cdot (\bast r\bast\widehat{p}^{+}+\bast\widehat{p}).
\] 
 Note that the factorization $\bast m= \bast x\cdot\bast\widehat{q}_{\bast r}$
is approximatel: for
\[ \upvarepsilon(\bast m)=\bast m\uptheta-\bast m^{\perp}=\bast x\cdot(\bast\widehat{q}_{\bast r}\uptheta -\bast\widehat{p}_{\bast r}) =\bast x\cdot
\upvarepsilon(\bast\widehat{q}_{\bast r}) \]
implying (since $\upvarepsilon(\bast m)\in\bast\R_{\upvarepsilon}$)  that $\bast x\in\bast\Z^{\upnu}$ for $\upnu=\upnu (\bast\widehat{q}_{\bast r})$.

We now show that $\bast m$ is multibest having infinite partial quotient i.e.\ that $\bast r=0$ and $\bast\widehat{q}_{\bast r}=\bast\widehat{q}$ has
infinite partial quotient.  To do this we will
make use of a closed expression for the error term $\upvarepsilon(\bast\widehat{q}_{\bast r})$ of the convergent $\bast\widehat{q}_{\bast r}$.  First, let $\bast a^{+} $ be the sequence class of $\{ a_{n_{i}+2}\}$; thus the possibilities afforded by Grace's Theorem are $\bast r=0,1,\bast a^{+}-1$.  Now for any $n$ we define $\uptheta_{n}$ by the formula $\uptheta =[a_{1}\dots a_{n-1}\uptheta_{n}]$.  In particular, $\uptheta_{n}=[a_{n}a_{n+1}...]=a_{n}+\uptheta_{n+1}^{-1}$.  Let $\bast\uptheta$ be the class of $\{ \uptheta_{n_{i}+2}\}$ and note that $\bast\uptheta=\bast a^{+}+(\bast\uptheta^{+})^{-1}$
where $\bast\uptheta^{+}$ is the class of $\{ \uptheta_{n_{i}+3}\}$.
Then the Lemma of \S I.4 of \cite{La1} yields
\[  |\upvarepsilon(\bast\widehat{q}_{\bast r})|= \frac{\bast\uptheta-\bast r}{\bast\uptheta\bast\widehat{q}^{+}+\bast\widehat{q}}.\]
and therefore $|\upvarepsilon(\bast m)|=\bast x \cdot (\bast\uptheta-\bast r)/(\bast\uptheta\bast\widehat{q}^{+}+\bast\widehat{q})$.

\vspace{3mm}

\noindent \fbox{\small {\em Case} 1} $\bast r=0$.  Since $\upnu(\bast m)\leq \upmu<\upmu (\bast m)$, the multibest inequalities (\ref{multibestineq}) imply that 
$\widehat{\upnu}\leq \upmu<\widehat{\upmu }$ and hence $\widehat{\upnu}<\widehat{\upmu }$ .  Since $\bast\uptheta \geq 1$, we have 
\[ \widehat{\upnu} =   \left\langle \bast \widehat{q}^{+}+ (\bast\widehat{q}/\bast\uptheta)\right\rangle^{-1}= \left\langle \bast \widehat{q}^{+}\right\rangle^{-1} = \widehat{\upmu}^{+}. \]
In particular, $\widehat{\upmu}^{+}<\widehat{\upmu}$ and therefore $\bast\widehat{q}$ has infinite partial quotient.
In other words, this case
comprises precisely the multibest denominators $\bast m$ having infinite partial quotient, for which $\upnu (\bast m ) \leq \upmu<
\upmu (\bast m ).$

\vspace{3mm}

\noindent \fbox{\small {\em Case} 2} $\bast r=1$.  Then $\upnu (\bast\widehat{q}_{\bast r})=\upmu (\bast\widehat{q}_{\bast r})=\widehat{\upmu}^{+}$, so it is impossible that $\upnu (\bast m)<\upmu(\bast m)$, in view of the approximate factorization
$\bast m= \bast x\cdot\bast\widehat{q}_{\bast r}$.  Indeed, $\upnu (\bast m)<\upmu(\bast m)$ implies that $\bast x\cdot \widehat{\upmu}^{+}<\bast x^{-1}\cdot \widehat{\upmu}^{+}$ or $\langle \bast x^{2}\rangle <1$, impossible since $\bast x $, being a sequence class of integers, is not infinitesimal.

\vspace{3mm}

\noindent \fbox{\small {\em Case} 3} $\bast r=\bast a^{+}-1$. Then $\bast\uptheta-\bast r=1+(\bast\uptheta^{+})^{-1}>1$ and is bounded
from above.  Thus
$\upnu (\bast\widehat{q}_{\bast r})=\upmu (\bast\widehat{q}_{\bast r})=\langle\bast a^{+}\widehat{q}^{+}\rangle^{-1}=(\bast a^{+})^{-1}\cdot
\widehat{\upmu}^{+}$,
so as in \fbox{\small {\em Case} 2} we cannot have $\upnu (\bast m)<\upmu(\bast m)$.
\end{proof}



Let $\bast\Z^{\infty}_{\rm b} (\uptheta )$ be the set of best denominators having infinite partial quotient and let  $\bast\Z^{\infty}_{\rm mb} (\uptheta )$ $\supset \bast\Z^{\infty}_{\rm b} (\uptheta )$ be the set
of multibest denominators having infinite partial quotient.

\begin{coro}\label{quasibestunion} Let $\uptheta\in\mathfrak{W}$.  Then 
\begin{align}\label{bestintdecomp} {\rm Spec}_{\rm flat}(\uptheta) = \bigcup_{\bast m\in\bast\Z^{\infty}_{\rm mb} (\uptheta )}  \;\big[\upnu (\bast m),\upmu (\bast m) \big) = 
 \bigcup_{\bast \widehat{q}\in\bast\Z^{\infty}_{\rm b} (\uptheta )} \;\big[\widehat{\upnu},\widehat{\upmu} \big).
 \end{align}
 In particular, ${\rm Spec}_{\rm flat}(\uptheta)$ has interior.\end{coro}

\begin{proof}  The first equality follows immediately from Theorem \ref{flatmultibest}.  If $\bast m = \bast x\cdot\bast\widehat{q}$
then \[ \big[\upnu (\bast m),\upmu (\bast m) \big)\subset \big[\widehat{\upnu},\widehat{\upmu} \big)\] giving the second equality.
\end{proof}





An interval $\big[\widehat{\upnu},\widehat{\upmu} \big)$ appearing in (\ref{bestintdecomp}) is called a {\bf best interval}.
These have been indicated in the portrait of ``generic irrational'' found in Figure \ref{portraits}.

\begin{note} As specified in Theorem \ref{verywellapprox} and Corollary \ref{LiouvilleSpec}, for $\uptheta\in \mathfrak{W}_{>1}$, ${\rm Spec}_{\rm flat}(\uptheta )$ contains power intervals of the form $\big[\widehat{\upmu}^{\uplambda},\widehat{\upmu} \big)$, $\uplambda>1$, and for $\uptheta\in \mathfrak{W}_{\infty}$,
${\rm Spec}_{\rm flat}(\uptheta )$
contains Frobenius rays $\Upphi_{(1,\infty)}(\widehat{\upmu})$.  If $\uptheta\in \mathfrak{W}_{1^{+}}$, then Corollary \ref{quasibestunion}
says that although ${\rm Spec}_{\rm flat}(\uptheta )$ contains no intervals of the form $\big[\widehat{\upmu}^{\uplambda},\widehat{\upmu} \big)$,
it nevertheless has interior.  We can summarize this trichotomy by saying that the flat spectrum of elements of $\uptheta$ in $ \mathfrak{W}_{1^{+}}$, 
$\mathfrak{W}_{>1}$ or $ \mathfrak{W}_{\infty}$ contains components having {\it subpolynomial}, {\it polynomial} or {\it exponential connectivity}, respectively.
\end{note}

We will show that
for $\uptheta\in\R-\Q$, ${\rm Spec}_{\rm flat}(\uptheta )$ is a proper subset of $\bstar\PR\R_{\upvarepsilon}$.
Let $\upmu\in\bstar\PR\R_{\upvarepsilon}$, which we assume can be represented by a sequence $\{ s_{i}\}$ of positive reals
monotonically tending to $0$.  We say that $\upmu$ is {\bf  shift invariant}
if we may choose $\{ s_{i}\}$ so that there exists $M>0$ with $s_{i}/s_{i+1}<M$
for all $i$ i.e.\ $s_{i+1}<s_{i}<Ms_{i+1}$ for all $i$.  If $\{ s_{i}'\}$ is another monotone representative sequence representing $\upmu$ then there exists
a constant $C>0$, $C\in\R$, for which $\bast s-C\bast s'\simeq 0$.  Thus any two monotone representative sequences will have this
property, and therefore being shift invariant is independent of the selected representative sequence.  If $\upmu$ is shift invariant and $\{ s_{i}\}$ is a monotone representative
then $ \langle \bast \{ s_{i+1}\}\rangle = \upmu$, hence the terminology.
Let $\bstar\PR\R_{\upvarepsilon}^\mathrm{sh}$
be the set of shift invariant elements of $\bstar\PR\R_{\upvarepsilon}$.  In what follows $X^{\text{\tt C}}$ denotes the complement of a set $X\subset Y$.

\begin{prop} $\bstar\PR\R_{\upvarepsilon}^\mathrm{sh}$ is clopen in $\bstar\PR\R_{\upvarepsilon}$.
\end{prop}

\begin{proof}  Suppose that $\upmu\not\in \bstar\PR\R_{\upvarepsilon}^\mathrm{sh}$, represent it by $\bast s=\bast\{ s_{i}\}$ with $s_{i}/s_{i+1}\rightarrow\infty$.  Let $0<\bast\upvarepsilon\in\bast\R_{\upvarepsilon}$ be such that $\upvarepsilon_{i}s_{i}/s_{i+1}
\rightarrow\infty$.  Define $\bast r=\bast \upvarepsilon\bast s$ and $\upmu'=\langle\bast r\rangle$.  Notice that $\upmu'\not\in  \bstar\PR\R_{\upvarepsilon}^\mathrm{sh}$ and $\upmu'<\upmu$.  Let $\upnu\in (\upmu',\upmu)$ be represented by $\{ x_{i}\}$ monotone and tending to $0$ in which
 $r_{i}<x_{i}<s_{i}$ for all $i$.  Then $x_{i}/x_{i+1}>r_{i}/s_{i+1}=\upvarepsilon_{i}s_{i}/s_{i+1}
\rightarrow\infty$.  Thus $ (\upmu',\upmu)\subset(\bstar\PR\R_{\upvarepsilon}^\mathrm{sh})^{\text{\tt C}}$.  A similar argument shows that there exists
$\upmu''>\upmu$ with $ (\upmu,\upmu'')\subset(\bstar\PR\R_{\upvarepsilon}^\mathrm{sh})^{\text{\tt C}}$, thus $\bstar\PR\R_{\upvarepsilon}^\mathrm{sh}$
is closed.  On the other hand, let $\upmu\in\bstar\PR\R_{\upvarepsilon}^\mathrm{sh}$ be represented by $\{ s_{i}\}$ with $s_{i}/s_{i+1}<M$
for all $i$.
Choose $0<\bast\upvarepsilon\in\bast\R_{\upvarepsilon}$ such that $\upvarepsilon_{i}/\upvarepsilon_{i+1}<N$ for all $i$,
let $\bast r=\bast\upvarepsilon\bast s$ and $\upmu'=\langle\bast r\rangle$.  Let $\upnu\in (\upmu',\upmu)$ be represented by $\{ x_{i}\}$ monotone and tending to $0$.  Then we may write $x_{i}=M_{i}r_{i}$ with $M_{i}\rightarrow\infty$; note then that (trivially) $M_{i}/M_{i+1}<M'$ for some $M'$.
Then $x_{i}/x_{i+1}<M'MN$ for all $i$ and thus $(\upmu',\upmu)\subset \bstar\PR\R_{\upvarepsilon}^\mathrm{sh}$.  By the symmetry of the argument, we have $\upmu\in (\upmu',\upmu'')\subset  \bstar\PR\R_{\upvarepsilon}^\mathrm{sh}$ for some $\upmu''>\upmu$, and so $\bstar\PR\R_{\upvarepsilon}^\mathrm{sh}$ is open as well.
\end{proof}

\begin{theo}\label{flatspectra} Let $\uptheta\in\R-\Q$.  Then $\bstar\PR\R_{\upvarepsilon}^\mathrm{sh}\subset {\rm Spec}_{\rm flat}(\uptheta )^{\text{\tt C}}$.
In particular, \[ {\rm Spec}_{\rm flat}(\uptheta )\subsetneqq\bstar\PR\R_{\upvarepsilon}.\]
\end{theo}

\begin{proof}  If $\uptheta$ is badly approximable the result follows trivially from Theorem \ref{badapproxchar}.  So we assume
$\uptheta$ is well approximable.  
Now suppose that $0\not=\bast m\in \bast\Z^{\upmu}_{\upmu}(\uptheta )$ where $\upmu\in \bstar\PR\R_{\upvarepsilon}^\mathrm{sh}$.
Then we may find monotone representatives $\{ s_{i}\} $ of $\upmu$ and $\{ m_{i}\}$ of $\bast m$ for which
$m_{i}s_{i}\rightarrow 0$ and for which there exists $M$ such that $s_{i}/s_{i+1}<M$ for all $i$.  
In particular, the first fact says that there exists an infinite sequence $\{R_{i}\}$ such that $s_{i} = (R_{i}m_{i})^{-1}$.
On the other hand, by Theorem \ref{flatmultibest}, $\bast m\in\bast\Z^{\infty}_{\rm mb} (\uptheta )$.  Thus if $\bast m = \bast x\cdot\bast\widehat{q}$
and we represent $\widehat{q}= \{ q_{n_{i}}\}$ then $q_{n_{i+1}}/q_{n_{i}}\rightarrow\infty$.  It follows that for any
representative $\{ x_{i}\}$ of $\bast x$
\[ \frac{s_{i}}{s_{i+1}}   =  \frac{R_{i+1}x_{i+1}q_{n_{i+1}}}{R_{i}x_{i}q_{n_{i}} }\longrightarrow \infty
\]
(since $|x_{i}|\geq 1$ as $\bast x\in\bast\Z$) contradicting the shift invariance of $\upmu$.







\end{proof}


\begin{note}  The shift invariant set $\bstar\PR\R_{\upvarepsilon}^\mathrm{sh}$ is greater (in the order $<$) than its complement
in $\bstar\PR\R_{\upvarepsilon}$; its elements may be characterized as the classes of ``slow'' infinitesimals.   Theorem \ref{flatspectra} 
says that the flat spectrum of $\uptheta$ irrational contains no slow indices.
\end{note}

\section{The Arithmetic of Approximate Ideals}\label{gdarith}

In this section we use the growth-decay filtration to provide the diophantine approximation groups with a partially defined multiplicative structure subject to matching conditions along the growth-decay indices.  We begin with a result which describes the sense in which diophantine approximation groups generalize ideals.

\begin{prop}\label{ideological} Let $\uptheta\in\R$.  For all $\upmu,\upnu,\upiota,\uplambda\in\bstar\PR\R_{\upvarepsilon}$ there exists an action
\[ \bast\Z^{\upnu[\upiota]}\times \bast\Z^{\upmu[\uplambda]}_{\upnu} (\uptheta )\longrightarrow\bast\Z^{\upmu\cdot\upnu[\upiota\cdot\uplambda]}_{\upiota} (\uptheta ),\quad (\bast a,\bast n )\mapsto \bast a\bast n\]
in which $(\bast a\bast n)^{\perp} =\bast a\bast n^{\perp}$.
\end{prop}

 \begin{proof}  For $\bast a\in\Z^{\upnu}$, $\bast a\bast n\uptheta= \bast a\bast n^{\perp}+\bast a\upvarepsilon(\bast n )$.
 Since $\bast a\cdot\upnu<\upiota\in \bstar\PR\R_{\upvarepsilon}$, $\bast a\upvarepsilon(\bast n )\in\bast\R_{\upvarepsilon}$.
 \end{proof}
 
 A 
{\bf (fine) approximate ideal} of $\bast\Z$ is a subgroup $\mathfrak{a}\subset\bast\Z$ equipped with a filtration by subgroups $\mathfrak{a}=\{\mathfrak{a}_{\upnu}\}$ 
indexed by $\upnu\in\bstar\PR\R_{\upvarepsilon}$
for which 
\[ \bast \Z^{\upnu}\cdot \mathfrak{a}^{\upmu}_{\upnu}\subset \mathfrak{a}^{\upmu\cdot\upnu}\quad
(\bast \Z^{\upnu[\upiota]}\cdot \mathfrak{a}^{\upmu[\uplambda]}_{\upnu} \subset \mathfrak{a}^{\upmu\cdot\upnu[\upiota\cdot\uplambda]}_{\upiota})\] for all $\upmu,\upnu\in\bstar\PR\R_{\upvarepsilon}$ (for all $\upmu,\upnu,\upiota,\uplambda\in\bstar\PR\R_{\upvarepsilon}$), where 
$\mathfrak{a}^{\upmu}_{\upnu}=\mathfrak{a}_{\upnu}\cap\bast \Z^{\upmu}$ (where $\mathfrak{a}^{\upmu[\uplambda]}_{\upnu}=\mathfrak{a}_{\upnu}\cap\bast \Z^{\upmu[\uplambda]}$).
If one forgets the filtrations,
what is left is the usual notion of ideal.  By Proposition \ref{ideological}, diophantine approximation groups are fine approximate ideals. 

More generally, one can define a {\bf (fine) approximate module} (over $\bast \Z$) as a bi-filtered abelian group $M=\{M_{\upnu}^{\upmu}\}$ 
(a tri-filtered abelian group $M=\{M_{\upnu}^{\upmu[\uplambda]}\}$) in which 
there is an action $\bast \Z^{\upnu}\cdot M^{\upmu}_{\upnu}\subset M^{\upmu\cdot\upnu}$ ($\bast \Z^{\upnu[\upiota]}\cdot M^{\upmu[\uplambda]}_{\upnu}\subset M^{\upmu\cdot\upnu[\upiota\cdot\uplambda]}_{\upiota}$).  A  homomorphism $f:M\rightarrow N$ between
approximate modules  is called a {\bf (fine) approximate module homomorphism} if 
\begin{itemize}
\item[MH1] it is filtered: $f(M_{\upnu}^{\upmu})\subset N_{\upnu}^{\upmu}$ 
for all $\upmu,\upnu\in \bstar\PR\R_{\upvarepsilon}$ ($f(M_{\upnu}^{\upmu[\uplambda]})\subset N_{\upnu}^{\upmu[\uplambda]}$ 
for all $\upmu,\upnu,\uplambda \in \bstar\PR\R_{\upvarepsilon}$) . 
\item[MH2] it respects the $\bast\Z$ action: for all $\upmu,\upnu\in \bstar\PR\R_{\upvarepsilon}$, $\bast m\in\bast\Z^{\upnu}$ and $x\in M^{\upmu}_{\upnu}$
\[ f(\bast m\cdot x)=\bast m\cdot f(x). \]
\end{itemize}
If we set $\bast\Z_{\upnu}:=\bast\Z^{\upnu}$ then $\bast \Z$ is an approximate module over itself, and any approximate ideal is an approximate module.


\begin{prop} Let $\uptheta,\upeta\in \R$.  If $\uptheta\Bumpeq \upeta$ by $A\in{\rm PGL}_{2}(\Z)$ then the induced isomorphism $A:\bast\Z (\uptheta )\cong \bast\Z (\upeta )$ is a fine approximate module isomorphism.
\end{prop}

\begin{proof}  By Theorem \ref{triisomorphism} we already know that $\uptheta\Bumpeq \upeta$ by $A\in{\rm PGL}_{2}(\Z)$ induces a tri-filtered
isomorphism.  If $\bast m \in\bast\Z^{\upnu}$, $\bast n\in\bast\Z^{\upmu}_{\upnu}$ then MH2 follows from:
\[  A(\bast m\cdot \bast n) = c(\bast m\cdot \bast n)^{\perp} + d(\bast m\cdot \bast n) = c\bast m\cdot \bast n^{\perp} + d\bast m\cdot \bast n = \bast m
\cdot A(\bast n).
\]
\end{proof}

The following result forms the basis of approximate ideal arithmetic.

\begin{theo}[Approximate Ideal Arithmetic]\label{productformula} Let $\uptheta ,\upeta\in\R$.  Then
\begin{align}\label{grdecprod1}
  \bast\Z^{\upmu[\upiota ]}_{\upnu}(\uptheta )\cdot  \bast\Z^{\upnu[ \uplambda ]}_{\upmu}(\upeta ) & \;\;\subset  \;\;
\bast\Z^{\upmu\cdot\upnu [\upiota\cdot \uplambda]}_{\upiota+ \uplambda}(\uptheta\upeta )\;\cap\; \bast\Z^{\upmu\cdot\upnu [\upiota\cdot \uplambda]}_{\upiota+ \uplambda}(\uptheta+\upeta )\;\cap\;\bast\Z^{\upmu\cdot\upnu [\upiota\cdot \uplambda]}_{\upiota+ \uplambda}(\uptheta-\upeta ).
\end{align}
In particular,
 $ \bast\Z^{\upmu}_{\upnu}(\uptheta )\cdot  \bast\Z^{\upnu}_{\upmu}(\upeta ) \subset  
\bast\Z^{\upmu\cdot\upnu}(\uptheta\upeta ) \cap \bast\Z^{\upmu\cdot\upnu}(\uptheta+\upeta )\cap \bast\Z^{\upmu\cdot\upnu}(\uptheta-\upeta )$.  Moreover, for
$\bast m\in  \bast\Z^{\upmu}_{\upnu}(\uptheta )$ and $\bast n\in  \bast\Z^{\upnu}_{\upmu}(\upeta )$
\begin{align}\label{addmultdual}
 (\bast m\cdot \bast n)^{\perp_{\uptheta\upeta}} = \bast m^{\perp_{\uptheta}} \cdot \bast n^{\perp_{\upeta}}
 \quad
\text{and}\quad
(\bast m\cdot \bast n)^{\perp_{\uptheta\pm\upeta}} = \bast m^{\perp_{\uptheta}} \cdot \bast n\pm\bast m \cdot \bast n^{\perp_{\upeta}}  \end{align}
\end{theo}

\begin{proof}  We prove first that the left hand side of (\ref{grdecprod1}) is contained in $\bast\Z^{\upmu\cdot\upnu [\upiota\cdot \uplambda]}_{\upiota+ \uplambda}(\uptheta\upeta )$.  Given $\bast m\in  \bast\Z^{\upmu[\upiota ]}_{\upnu}(\uptheta )$ and $\bast n\in   \bast\Z^{\upnu[ \uplambda ]}_{\upmu}(\upeta ) $,
\begin{align}\label{productofdas}  \uptheta\upeta( \bast m\bast n ) = \bast m^{\perp}\bast n^{\perp} +  \upvarepsilon (\bast m) \upvarepsilon (\bast n) +  \upvarepsilon (\bast m)\bast n^{\perp}
+  \upvarepsilon (\bast n)\bast m^{\perp} .  
\end{align}
By hypothesis, the cross terms on the right hand side of (\ref{productofdas}), $ \upvarepsilon (\bast m)\bast n^{\perp}$
and $ \upvarepsilon (\bast n)\bast m^{\perp}$, are infinitesimals: here we are using
(\ref{dualordmag}).  Thus $\uptheta\upeta(\bast m\bast n)$
 is infinitesimal to $ \bast m^{\perp}\bast n^{\perp} $
and  $\bast m\bast n\in \bast\Z( \uptheta\upeta)$.  In particular, $(\bast m\bast n)^{\perp} = \bast m^{\perp}\bast n^{\perp} $,
giving the product duality in (\ref{addmultdual}).
Moreover, 
$ (\bast m\bast n)\cdot (\upmu\cdot\upnu ) =
(\bast m\cdot \upmu)\cdot ( \bast n\cdot \upnu)<\upiota\cdot \uplambda $ so that $\bast m\bast n\in 
\bast\Z^{\upmu\cdot\upnu [\upiota\cdot \uplambda ]}(\uptheta\upeta )$.  By (\ref{dualordmag}) again,
$   \bast m^{\perp}\cdot\upnu (\bast n ) \leq\bast m^{\perp}\cdot \upmu <\upiota$, $ \bast n^{\perp}\cdot\upnu (\bast m ) \leq\bast n^{\perp}\cdot \upnu < \uplambda $.
Thus $\upnu (\bast m\bast n)$ satisfies the bound
$  \upnu (\bast m\bast n) \leq (\upmu\cdot\upnu)+ \upiota +  \uplambda$ . 
Since $\bast n$ is infinite, $\bast n\cdot\upnu < \uplambda$ implies that $\upnu< \uplambda$ and therefore 
$\upmu\cdot\upnu<\upmu\cdot \uplambda <  \uplambda$.  Hence $(\upmu\cdot\upnu)+ \upiota +  \uplambda =\upiota +  \uplambda$, and $\bast m\bast n\in \bast\Z_{\upiota +  \uplambda}(\uptheta\upeta)$ as claimed.  As for the inclusion into $\bast\Z^{\upmu\cdot\upnu [\upiota\cdot \uplambda]}_{\upiota+ \uplambda}(\uptheta+\upeta )$, this follows from the additive analog of (\ref{productofdas}),
\[  
 (\uptheta+\upeta)( \bast m\bast n ) = \bast m^{\perp}\bast n +  \bast m\bast n^{\perp} +  \upvarepsilon (\bast m)\bast n^{\perp}
+  \upvarepsilon (\bast n)\bast m^{\perp} .
\]
The proof of the inclusion into $\bast\Z^{\upmu\cdot\upnu [\upiota\cdot \uplambda]}_{\upiota+ \uplambda}(\uptheta-\upeta )$ is identical.
\end{proof}

The product defined by (\ref{grdecprod1}) will be referred to as the {\bf growth-decay } or {\bf approximate ideal product}.
While the approximate ideal product has image in $\bast\Z(\uptheta\upeta )$, it is not the case that it 
has image in $\bast\Z(\uptheta\upeta^{-1} )$: in contrast with the additive situation, in which the image of (\ref{grdecprod1}) is contained in both
$\bast\Z(\uptheta+\upeta )$ and $\bast\Z(\uptheta-\upeta )$.  
That the approximate ideal product yields diophantine approximations of the product, sum and difference has to do with the fact
that it is essentially the precursor of the product of associated ``two generator'' approximate ideals (generated by ``decoupled'' numerators and denominators).
This will be taken up in \S \ref{ideologicalclasssection}.  

\begin{note}\label{slowimage} By definition of the growth-decay trifiltration, $\upmu<\upiota$, $\upnu< \uplambda$ so that $\upmu\cdot\upnu<\upiota+ \uplambda$.  Thus the image of the product (\ref{grdecprod1}) is contained in the groups indexed by the slow components of ${\rm Spec}(\uptheta\upeta)$,
${\rm Spec}(\uptheta\pm\upeta)$.
\end{note}

\begin{note}\label{genlprodnote}  As the argument in the proof of Theorem \ref{productformula} clearly shows, for any pair of growth-decay indices
$(\upmu_{1}[\upiota_{1}], \upnu_{1})$, $(\upmu_{2}[\upiota_{2}], \upnu_{2})$ for which
\[  \upnu_{1}\leq \upmu_{2},\quad \upnu_{2}\leq \upmu_{1} \]
we have the product 
\begin{align}\label{grdecprodgenl}
  \bast\Z^{\upmu_{1}[\upiota_{1} ]}_{\upnu_{1}}(\uptheta )\cdot  \bast\Z^{\upmu_{2}[ \upiota_{2} ]}_{\upnu_{2}}(\upeta ) 
  & \;\;\subset  \;\;
\bast\Z^{\upmu_{1}\cdot\upmu_{2} [\upiota_{1}\cdot \upiota_{2}]}_{\upiota_{1}+ \upiota_{2}}(\uptheta\upeta )\;\cap\; 
\bast\Z^{\upmu_{1}\cdot\upmu_{2} [\upiota_{1}\cdot \upiota_{2}]}_{\upiota_{1}+ \upiota_{2}}(\uptheta+\upeta )\;\cap\;
\bast\Z^{\upmu_{1}\cdot\upmu_{2} [\upiota_{1}\cdot \upiota_{2}]}_{\upiota_{1}+ \upiota_{2}}(\uptheta-\upeta ).
\end{align}
However for such indices with say $\upmu_{1}\geq \upnu_{1}$ we have
\[  \bast\Z^{\upmu_{2}[ \upiota_{2} ]}_{\upnu_{2}}(\upeta ) \subset \bast\Z^{\upnu_{1}[ \upiota_{2} ]}_{\upmu_{1}}(\upeta )  \]
so the product (\ref{grdecprodgenl}) is subsumed by that of (\ref{grdecprod1}).
\end{note}

There is no harm in stating the obvious: the formulas (\ref{addmultdual}) for the additive and multiplicative duals 
are just the formulas for the numerators of fractional sum and product:
\[ \frac{\bast m^{\perp_{\uptheta}}}{\bast m}\cdot \frac{\bast n^{\perp_{\upeta}}}{\bast n} = \frac{(\bast m\bast n)^{\perp_{\uptheta\upeta}}}{\bast m\bast n},\quad
 \frac{\bast m^{\perp_{\uptheta}}}{\bast m}\pm \frac{\bast n^{\perp_{\upeta}}}{\bast n} = \frac{\bast m^{\perp_{\uptheta}}\bast n\pm\bast m\bast n^{\perp_{\upeta}}}{\bast m\bast n}
=\frac{(\bast m\bast n)^{\perp_{\uptheta\pm\upeta}}}{\bast m\bast n}
 ,\]
formulas which are compatible with their standard parts: the product and sum/difference of $\uptheta$ and $\upeta$.   When $\uptheta,\upeta\in\Q$ the growth decay product is just
the product on the fine growth filtration:
\[  \bast\Z^{\upmu[\upiota ]}(\uptheta )\cdot  \bast\Z^{\upnu[ \uplambda ]}(\upeta)  \subset  
\bast\Z^{\upmu\cdot\upnu [\upiota\cdot \uplambda]}(\uptheta\upeta )\cap \bast\Z^{\upmu\cdot\upnu [\upiota\cdot \uplambda]}(\uptheta\pm\upeta )  \]
which reduces upon restriction to standard parts to $ \Z(\uptheta )\cdot  \Z(\upeta) = 
\Z (\uptheta\upeta)\cap \Z (\uptheta\pm\upeta )$.  This is in keeping with the fact that the diophantine approximation group of an element of $a/b\in\Q$
 is the ideal $\bast (b)$ generated by its denominator.
 
Let us compare approximate ideal arithmetic with ideal arithmetic in $\Z$.  For $a\in\Z$, let $\bast (a )$ =  the ideal generated by $a$ in $\bast\Z$.
Then given
 $a,b\in\Z$, clearly $\bast\Z(a)^{\perp} =\bast (a )$,
$\bast\Z(b)^{\perp}=\bast (b )$, $ \bast\Z(a+b)^{\perp}=\bast (a+b)$, $ \bast\Z(ab)^{\perp}=\bast (ab)$ .  The  map (\ref{grdecprod1})
induces on the level of dual groups a pair of bilinear maps corresponding to the sum and the product.  The map corresponding to the product is
\[ \bast\Z(a)^{\perp}\times  \bast\Z(b)^{\perp}\longrightarrow \bast\Z(ab)^{\perp}, \quad 
(\bast m\cdot a,\bast n\cdot b)\mapsto (\bast m\cdot a,\bast n\cdot b)^{\perp_{ab}}=\bast m\bast n\cdot ab,
\]
which corresponds exactly to the usual product of ideals.  

On the other hand, the map of dual groups corresponding to the sum is 
\[ \bast\Z(a)^{\perp}\times  \bast\Z(b)^{\perp}\longrightarrow \bast\Z(a+b)^{\perp}\subsetneqq  \bast (a)+\bast (b)\]
since it is given by
\[ (\bast m\cdot a,\bast n\cdot b)\mapsto (\bast m\cdot a,\bast n\cdot b)^{\perp_{a+b}}=\bast m\bast n\cdot (a+b) . \]
In particular, the image of the duality map is not all of $\bast (a)+\bast (b)$, since the elements belonging to the latter
are {\it independent} linear combinations of $a,b$.  This situation will be addressed in \S \ref{ideologicalclasssection} by ``decoupling''
numerators and denominators.


Let $\upmu\geq \upnu$.  We will write 
\begin{align}\label{prodnotation}  \uptheta {}_{\upmu}\!\!\owedge_{\upnu}\upeta,\quad \left( \uptheta {}_{\upmu[\upiota]}\!\!\owedge_{\upnu[ \uplambda]}\upeta \right) \end{align} to indicate that
both of $ \bast\Z^{\upmu}_{\upnu}(\uptheta )$, $ \bast\Z^{\upnu}_{\upmu}(\upeta )$ (both of $ \bast\Z^{\upmu[\upiota]}_{\upnu}(\uptheta )$, $ \bast\Z^{\upnu[ \uplambda]}_{\upmu}(\upeta )$) are non-zero, so that the growth-decay product  is non-trivial.   When $\upmu> \upnu$ strictly, then $\uptheta {}_{\upmu}\!\!\owedge_{\upnu}\upeta $ exactly when $\bast\Z^{\upmu}_{\upnu}(\uptheta )$
is non-trivial, by Theorem \ref{gennonvan}. In particular, for $\upmu> \upnu$,  $\uptheta {}_{\upmu}\!\!\owedge_{\upnu}\upeta$ implies that $\uptheta {}_{\upmu'}\!\!\owedge_{\upnu'}\upeta$
whenever $\upmu\geq\upmu'> \upnu'\geq\upnu$.

The
relation ${}_{\upmu}\!\owedge_{\upnu}$ is not symmetric (commutative) i.e.\ (\ref{prodnotation}) does not imply that $\upeta {}_{\upmu}\!\!\owedge_{\upnu}\uptheta$; in fact, as we will see below, (\ref{prodnotation}) does not even imply that
$\upeta{}_{\upmu'}\!\owedge_{\upnu'}\uptheta$ for {\it some} pair $(\upmu', \upnu')\in\bstar\PR\R_{\upvarepsilon}$ with $\upmu'\geq \upnu'$.

 In view of Theorem \ref{productformula}, 
 only the ``tri-filtered'' relation $ {}_{\upmu[\upiota]}\owedge_{\upnu[ \uplambda]}$ can be iterated, and 
 in view of Note \ref{slowimage}, 
 iterated compositions move {\it to the left}.  More precisely the iterated relation 
\[\upxi {}_{\upmu'[\upiota']}\!\!\owedge_{\upnu'[ \uplambda']}\big( \uptheta {}_{\upmu[\upiota]}\!\!\owedge_{\upnu[ \uplambda]}\upeta\big)\]
is {\it defined}
provided $\upiota+ \uplambda\leq \upmu'$, $\upnu'\leq \upmu\cdot\upnu$ and $\lambda'\geq \upiota\cdot \uplambda$.  In particular, the issue
of associativity is moot since only the right-to-left composition has the possibility of being defined.
Intuitively, as one continues to iterate, the nonvanishing spectrum of the real number appearing on the left should become progressively larger.

 The approximate ideal product is natural with respect to multiplication and the duality/reciprocal maps:  
$ ( \bast m\cdot\bast n )^{\perp_{\uptheta\upeta}} = \bast m^{\perp_{\uptheta}}\cdot\bast n^{\perp_{\upeta}} $ for
$\bast m \in \bast\Z^{\upmu}_{\upnu}(\uptheta )$ and $\bast n\in \bast\Z^{\upnu}_{\upmu}(\upeta )$.
That is, in terms of the growth-decay bi-filtration,
\[ 
\begin{CD}
\bast\Z^{\upmu}_{\upnu}(\uptheta )\times \bast\Z^{\upnu}_{\upmu}(\upeta )     @>\cdot >>   \bast\Z^{\upmu\cdot \upnu}(\uptheta\upeta )\\
@V\perp\times\perp V\cong V                @V\cong V\perp V\\
\bast\Z^{\upmu}_{\upnu}(\uptheta^{-1} )\times  \bast\Z^{\upnu}_{\upmu}(\upeta^{-1} )    @>\cdot >>   \bast\Z^{\upmu\cdot\upnu}((\uptheta\upeta)^{-1} )
\end{CD}
\]
with a similar diagram for the fine growth-decay tri-filtration.
 If we fix $\bast m\in \bast\Z^{\upmu}_{\upnu}(\uptheta )$ then multiplication by $\bast m$,
$\bast n\mapsto \bast m\cdot\bast n$,
  defines a linear map from
$ \bast\Z^{\upnu}_{\upmu}(\upeta )$ onto its image in $\bast\Z^{\upmu\cdot \upnu}(\uptheta\upeta )$.

For the remainder of this section we will regard the growth-decay product as defining a bilinear map to the product approximation group
 \[ \bast\Z^{\upmu[\upiota ]}_{\upnu}(\uptheta )\cdot  \bast\Z^{\upnu[ \uplambda ]}_{\upmu}(\upeta ) \longrightarrow
 \bast\Z^{\upmu\cdot\upnu [\upiota\cdot \uplambda]}_{\upiota+ \uplambda}(\uptheta\upeta ).\]
 All statements which follow will have a corresponding additive counterpart, obtained
 by replacing the word ``divisor'' by ``summand''.
 
Let $\upomega=\uptheta\upeta$ and suppose that $\uptheta {}_{\upmu}\!\!\owedge_{\upnu}\upeta$ for $\upmu >\upnu$.
We say that $\uptheta$ is a   {\bf $\upmu/\upnu $-fast divisor} of $\upomega$ ($\upeta$ is a   {\bf $\upnu/\upmu $-slow divisor} of $\upomega$) and we write
\[   \uptheta\,{}_{\upmu}\!\!\Uparrow_{\upnu} \upomega \quad (\upeta\,{}_{\upmu}\!\!\Downarrow_{\upnu} \upomega).  \]
In addition, we write $ \uptheta \Uparrow\upomega$ ($ \upeta \Downarrow\upomega $) to mean that $\uptheta\,{}_{\upmu}\!\!\Uparrow_{\upnu}\upomega$ ($\upeta\,{}_{\upmu}\!\!\Downarrow_{\upnu}\, \upomega$) for some $\upmu>\upnu$.
Fast divisors (slow divisors) have error terms
which tend to zero more rapidly (slowly) than their denominators  tend to infinity.  These designations are not symmetric, and we will see below that the badly approximable numbers are never fast divisors.  

\begin{prop}  $ \uptheta\,{}_{\upmu}\!\!\Uparrow_{\upnu}\, \upomega$ for all $\upmu>\upnu$ $\Leftrightarrow$ $\uptheta\in\Q$.
\end{prop}

\begin{proof}  $\bast\Z^{\upmu}_{\upnu}(\uptheta )$ is nontrivial for all $\upmu>\upnu$ $\Leftrightarrow$ $\uptheta\in\Q$.
\end{proof}

If $\uptheta {}_{\upmu}\!\!\owedge_{\upmu}\upeta$,
we say that both $\uptheta$ and $\upeta$ are {\bf  $\upmu$-flat divisors}
of $\upomega$ and write 
\[ \uptheta \talloblong\upomega= \uptheta\talloblong_{\upmu}\upomega\ ,\quad 
\upeta\talloblong\upomega=\upeta \talloblong_{\upmu}\upomega.\]  
Thus $\uptheta \talloblong\upomega$ $\Leftrightarrow$ ${\rm Spec}_{\rm flat}(\uptheta )\cap {\rm Spec}_{\rm flat}(\upeta )\not=\emptyset$.
If $\uptheta \Downarrow\upomega$, $\uptheta \Uparrow\upomega$ we will
write $ \uptheta \Updownarrow\upomega$ and say that $\uptheta$ is an {\bf elastic divisor}; if $\uptheta$ is elastic and $\uptheta \talloblong\upomega$
as well then we will
write $ \uptheta \bar{\Updownarrow}\upomega$
say that $\uptheta $ is a  {\bf strong elastic divisor} of $\upomega$.  Note that if $ \uptheta \bar{\Updownarrow}\upomega$
and $\upeta=\upomega/\uptheta$ then $\upeta \bar{\Updownarrow}\upomega$ as well.

\begin{prop} $\uptheta {}_{\,\upmu}\!\Uparrow_{\upnu}\upomega$, $\uptheta {}_{\,\upmu}\Downarrow_{\upnu}\upomega$ and $\uptheta \Updownarrow_{\upmu}\upomega$
for all $\upmu\geq \upnu$  $\Leftrightarrow$ $\uptheta,\upomega\in\Q$.
\end{prop}

\begin{proof} Trivial.
\end{proof}

If $\upomega=\uptheta\upeta$ but 
$\uptheta\not\Uparrow\upomega$,
$\uptheta \not\Downarrow \upomega$ and $\uptheta \not\talloblong\upomega$ we say that $\uptheta$ is an {\bf  antidivisor}
and write 
\[ \uptheta \lightning\upomega.\]  Note that if $\uptheta \lightning\upomega$ and $\upeta=\upomega/\uptheta$ it is not necessarily the
case that $\upeta\lightning\upomega$ (c.f.\ Theorem \ref{classesdivision}, c. below).
If on the other hand for any $\upomega$, $\uptheta$ is a divisor (of some speed: fast, slow or flat) 
we say that $\uptheta$ is an {\bf  omnidivisor}.

We now examine the notions of divisibility described above with regard to the classes $\mathfrak{B}$,  $\mathfrak{W}_{1^{+}}$,
$\mathfrak{W}_{>1}$ and $\mathfrak{W}_{\infty}$.

\begin{theo}\label{classesdivision}  Let $\upomega\in\R$. 
\begin{enumerate}
\item[a.] For all $\uptheta\in\mathfrak{B}$, $\uptheta\not\Uparrow\upomega$, $\uptheta\not\talloblong\upomega$.
\item[b.] For all $\uptheta\in\mathfrak{W}_{1^{+}}$, $\uptheta\not\Uparrow\upomega$. 
\item[c.] If $\upomega=\uptheta\upeta$, $\uptheta\in\mathfrak{B}$ and $\upeta\in \mathfrak{B}\cup \mathfrak{W}_{1^{+}}$ then
$\uptheta\lightning\upomega$.  
\item[d.] If $\upomega=\uptheta\upeta$, $\uptheta,\upeta\in\mathfrak{W}_{>1}$ then
$\uptheta,\upeta\Updownarrow\upomega$.  If moreover $\uptheta$ and $\upeta$ are equivalent then 
$\uptheta,\upeta\bar{\Updownarrow} \upomega$.
\end{enumerate}
\end{theo}

\begin{proof} a.\ \& b. The spectrum of any element of $\mathfrak{B}$ (of $\mathfrak{W}_{1^{+}}$) 
consists exactly of pairs $\upmu<\upnu$ (consists of pairs which satisfy $\upmu\leq \upnu$), so
flat or fast composition (fast composition) with $\uptheta$ is not possible.  c. This follows trivially from the definitions. d.
For any $\uptheta\in\mathfrak{W}_{>1}$ there exists pair $\upmu >\upnu$ in ${\rm Spec}(\uptheta )$ so
$ \uptheta {}_{\upmu}\!\!\owedge_{\upnu}\upeta$.  Switching the roles of $\uptheta$ and $\upeta$, for appropriate $\upmu >\upnu$,
$ \upeta {}_{\upmu}\!\!\owedge_{\upnu}\uptheta$ as well.  Thus $\uptheta,\upeta\Updownarrow\upomega$.
If $\uptheta$ and $\upeta$ are equivalent then ${\rm Spec}_{\rm flat}(\uptheta )={\rm Spec}_{\rm flat}(\upeta )$ implying $\uptheta,\upeta\bar{\Updownarrow} \upomega$.
\end{proof}

\begin{coro} The set of omnidivisors is
precisely $\mathfrak{W}_{>1}$. 
\end{coro}

\begin{proof}  By Theorem \ref{classesdivision}, parts a., b., the set of omnidivisors is contained in $\mathfrak{W}_{>1}$;
by part d., every element of $\uptheta\in\mathfrak{W}_{>1}$ is an omnidivisor.
\end{proof}

A subset $X\subset \R$ is called an  {\bf  antiprime set} if for all $\uptheta_{1},\uptheta_{2}\in X$, 
$\uptheta_{1},\uptheta_{2} \lightning\uptheta_{1}\uptheta_{2}$.  

\begin{prop}  $\mathfrak{B}$ is the unique maximal antiprime set in $\R$.
\end{prop} 

\begin{proof}  Note that $\mathfrak{B}$ is an anti-prime set: $\uptheta_{1},\uptheta_{2} \lightning\uptheta_{1}\uptheta_{2}$
for all $\uptheta_{1},\uptheta_{2} \in\mathfrak{B}$.  Moreover, if we add another element $\upeta\not\in\mathfrak{B}$ 
we lose the defining property since for such $\upeta$, $\upeta\talloblong \upeta^{2}$.  On the other hand, any $\uptheta\in\R-\mathfrak{B}$
is composable with itself, so there are no antiprime sets containing elements not in $\mathfrak{B}$.
\end{proof}

Thus we shall refer to $\mathfrak{B}$ as the {\bf antiprimes} of approximate ideal arithmetic.  We say that
$\upomega$ has an {\bf antiprime decomposition} if $\upomega=\uptheta\upeta$ for $\uptheta, \upeta\in\mathfrak{B}$.
For example, every 
$q\in \Q$ which is not a square has the antiprime decomposition $q=\sqrt{q}\sqrt{q}$.  
Antiprime decompositions are outside the realm of the version of approximate ideal arithmetic presented here, to analyze them 
requires the finer arithmetic of {\it symmetric diophantine approximations}, the subject of \S \ref{metrical}.  


\begin{theo}  Every non zero real number admits an antiprime decomposition.
\end{theo}

\begin{proof}  Let $F(n)=\{\uptheta\in\mathfrak{B}|\; 
\forall i\; a_{i}(\uptheta )\leq n
\}$.  By \cite{Ha}, every real number $r>1$ is a product $r=\uptheta\upeta$
of elements of $F(4)$.  Since $\mathfrak{B}$ is closed under inversion, the claim follows.
\end{proof}


A set $X$ is said to be {\bf (strongly) approximately generated} by
$X_{0}\subset X$ if for all $\upomega\in X$, there exist $\uptheta_{1},\uptheta_{2}\in X_{0}$ such that 
$\uptheta_{1}\uptheta_{2}=\upomega$ with $\uptheta_{1},\uptheta_{2}\Updownarrow\upomega$ 
($\uptheta_{1},\uptheta_{2}\bar{\Updownarrow}\upomega$).  

\begin{coro}  The set of Liouville numbers $\mathfrak{W}_{\infty}$ approximately generates $\R$.
\end{coro}

\begin{proof}  By \cite{Er} every real number $\upomega$ 
may be written as a product of $\uptheta_{1},\uptheta_{2}\in\mathfrak{W}_{\infty}$
and clearly $\uptheta_{1},\uptheta_{2}\Updownarrow\upomega$.
\end{proof}

\section{Flat Arithmetic}\label{flatharith}

In this section we consider the commutative relation $ {}_{\upmu}\!\owedge_{\upmu}$, which is best understood using
the sequence of partial quotients $\uptheta =[a_{0}a_{1}\dots ]$.  As alluded to at the beginning of \S \ref{flatsection}, this makes the determination of
the flat product
somewhat transverse to the linear classification: $\Q$, $\mathfrak{B}$,
 $\mathfrak{W}_{\upkappa}$, $\mathfrak{W}_{\infty}$.  Since elements of $\mathfrak{B}$ cannot be involved in flat products,
 all reals considered in this section are assumed not to belong to $\mathfrak{B}$.

Given $\uptheta\in\mathfrak{W}$, the basic problem is to determine the set of $\upeta\in\mathfrak{W}$
which have a flat product with $\uptheta$:
\[ \Upomega (\uptheta )=\left\{\upeta\in\mathfrak{W} |\; \uptheta {}_{\upmu}\!\owedge_{\upmu}\upeta\text{ for some }\upmu
\in  {\rm Spec}_{\rm flat}(\uptheta )\right\}.\]
It is clear that $\Upomega(\uptheta )$ is a projective linear invariant:
\begin{prop}  If $\uptheta \Bumpeq\upeta$ then 
$ \Upomega (\uptheta )= \Upomega (\upeta )$ and $ \Upomega (\uptheta )\supset \{  A(\uptheta ) |\; A\in {\rm PGL}_{2}(\Z ) \}$.
\end{prop}
 
 \begin{proof}  Immediate from Proposition \ref{isospectral}.
\end{proof}
 
The following Proposition gives a simple criterion in terms of best classes for when flat products are defined.

\begin{lemm}\label{firstflatcond} Let $\uptheta,\upeta\in\mathfrak{W}$.  Then there exists $\upmu\in\bstar\PR\R_{\upvarepsilon}$ such that
$\uptheta {}_{\upmu}\!\owedge_{\upmu}\upeta$ $\Leftrightarrow$ there exist
best classes $\bast \widehat{q}\in\bast \Z (\uptheta )$, $\bast \widehat{q}'\in\bast \Z (\upeta )$ having infinite partial quotient such that 
either 
\begin{align}\label{flatcondition}
\widehat{\upmu}^{+}\leq ( \widehat{\upmu}')^{+}<\widehat{\upmu} 
&\quad \text{ or  }\quad ( \widehat{\upmu}')^{+}\leq \widehat{\upmu}^{+}< \widehat{\upmu}'.
\end{align}
\end{lemm}

\begin{proof} 
By Corollary \ref{quasibestunion},
a flat composition is defined $\Leftrightarrow$ there exists best classes having infinite partial quotient $\bast \widehat{q} \in\bast\Z^{\infty}_{\rm b} (\uptheta )$, 
$\bast \widehat{q}'\in\bast\Z^{\infty}_{\rm b} (\upeta)$ for which  
\begin{align}\label{flatintersectionnotriv} [\widehat{\upnu},\widehat{\upmu})\cap
 [\widehat{\upnu}',\widehat{\upmu}')\not=\emptyset.\end{align}
Because $\bast \widehat{q}\in\bast \Z (\uptheta )$ has infinite partial quotient, 
$ \widehat{\upnu}= \widehat{\upmu}^{+}=\langle\bast \widehat{q}^{+}\rangle^{-1} <\widehat{\upmu}= \langle\bast \widehat{q}\rangle^{-1}$,
since 
\[ \widehat{\upnu}= \left\langle\frac{\bast\uptheta}{\bast\uptheta\bast \widehat{q}^{+} +\bast \widehat{q}}\right\rangle =
\langle\bast \widehat{q}^{+}\rangle^{-1} =  \widehat{\upmu}^{+}.\]  Thus (\ref{flatintersectionnotriv}) is equivalent to $ [\widehat{\upmu}^{+},\widehat{\upmu})\cap
 [(\widehat{\upmu}')^{+},\widehat{\upmu}')\not=\emptyset$ which in turn is equivalent to (\ref{flatcondition}). 
\end{proof}

We introduce now a class of numbers for which flat products are manifestly always defined: $\uptheta=[a_{0}a_{1}\dots ]\in\mathfrak{W}$ is called {\bf resolute} if $\lim a_{i}=\infty$ (i.e.\ there are no bounded subsequences). 
For example, $\uptheta = [1,2,3,\dots]$ is resolute whereas $e=[2;1,2,1,1,4,1,1,\dots ]$ is not.  It is fairly easy to produce
Liouville numbers which are resolute and Liouville numbers which are not:

\begin{exam}\label{resoluteexam} Let $\uptheta = [a_{0}a_{1}\dots ]$ be defined inductively by taking $a_{0}=a_{1}=1$ and $a_{n+1} = q_{n}^{n-1}$
for $n\geq 1$.  Then 
$\| q_{n}\uptheta\|<q_{n+1}^{-1}<q_{n}^{-n}$ so that $\uptheta$ is a resolute Liouville number.
If instead we define $a_{n} = q_{n-1}^{n}$ for $n$ even and $a_{n}=1$ for $n$ odd,
then for $n$ odd,  we have $\| q_{n}\uptheta\|<q_{n+1}^{-1}<q_{n}^{n+1}$ so that $\uptheta$ is Liouville but
 irresolute.  The Liouville number $L(m)=\sum_{j=0}^{\infty} m^{-(j+1)!}$ is irresolute for all $m>1$ since $1$ occurs infinitely
often in its sequence of partial quotients \cite{Sh}.  
\end{exam}

Define the relation \[ \uptheta\rightslice_{\upmu}\upeta\] if $\uptheta {}_{\upmu}\!\owedge_{\upmu}\upeta$
and there exists $\bast \widehat{q}\in\bast \Z^{\upmu}_{\upmu} (\uptheta )$, $\bast \widehat{q}'\in\bast \Z^{\upmu}_{\upmu} (\upeta )$
with $ \widehat{\upmu}^{+}\leq (\widehat{\upmu}')^{+}< \widehat{\upmu} $.

\begin{theo}  If $\uptheta\in\mathfrak{W}$ is resolute then for all $\upeta\in\mathfrak{W}$ there exists $\upmu\in {\rm Spec}_{\rm flat}(\uptheta )$
with $\uptheta\rightslice_{\upmu}\upeta$.
\end{theo}

\begin{proof}   For $\uptheta\in\mathfrak{W}$ resolute every best class $\bast \widehat{q}$ has infinite partial quotient.  
Now let $\bast \widehat{q}'$
be a best class for $\upeta$ with infinite partial quotient, represented by the sequence $\{ q'_{n_{i}}\}$.  For each $i$
choose $q_{n_{i}}$ so that $q_{n_{i}}\leq q_{n_{i}+1}'\leq q_{n_{i}+1}$.  Then denoting $\bast \widehat{q}=\bast\{q_{n_{i}}\}$,
we have $\bast \widehat{q}\leq (\bast \widehat{q}')^{+}\leq \widehat{q}^{+}$.  Since
$\bast \widehat{q}$ has infinite partial quotient we must have either $\langle \bast \widehat{q}\rangle
<\langle (\bast \widehat{q}')^{+}\rangle$ or $\langle (\bast \widehat{q}')^{+}\rangle<\langle \bast \widehat{q}^{+}\rangle$.
If the former is true, this implies the first inequality in (\ref{flatcondition}).  If not, $\langle \bast \widehat{q}\rangle=\langle (\bast \widehat{q}')^{+}\rangle$;
then replace $\bast \widehat{q}$ by $\bast \widehat{q}^{-}$ = $\bast\{ q_{n_{i}-1}\}$ to get the desired
inequality.  
\end{proof}

With Lemma \ref{firstflatcond} this gives:

\begin{coro}\label{rescor}  If $\uptheta\in\mathfrak{W}$ is resolute, $\Upomega (\uptheta )=\R-\mathfrak{B}$.  
\end{coro}

For any class $\mathfrak{C}$ of real numbers, we denote by $\mathfrak{C}^{\rm res}$ the resolute members.  For example,
$\mathfrak{B}^{\rm res}=\emptyset$ and $\mathfrak{W}_{\infty}^{\rm res}\subsetneqq\mathfrak{W}_{\infty}$.

\begin{coro}  If $\uptheta,\upeta\in\mathfrak{W}_{>1}^{\rm res}$ and $\upomega=\uptheta\upeta$ then
 $\uptheta,\upeta\bar{\Updownarrow}\upomega$.   In particular, $\uptheta,\upeta\in\mathfrak{W}_{>1}^{\rm res}$
 strongly approximately generates $\mathfrak{W}_{>1}^{\rm res}\cdot\mathfrak{W}_{>1}^{\rm res}$.
\end{coro}

We will now consider the other extreme:  pairs of real numbers which may not be flat composed at all.  Constructing such
pairs depends on finding real numbers which have large gaps in their flat spectra: that is, having increasingly long
blocks of partial quotients all of whose members are uniformly bounded.  

More precisely we say that $\uptheta$ has an
{\bf abyss} if there exist a sequence of blocks of consecutive partial quotients 
\[ G=\{ B_{n}\}, \quad B_{n}=\{ a_{i_{n}},a_{i_{n}+1}\dots ,
a_{i_{n}+k_{n}}\}\] which are pairwise disjoint, each element of which is bounded by a constant $M$ independent of $n$, such that the block length $k_{n}\rightarrow\infty$.  In this event we
say that $\uptheta$ is {\bf abyssal}; if $\uptheta$ is not abyssal, it is called {\bf abyssless}. For example:
any resolute number is abyssless;  $e$ is abyssless yet irresolute, as is
the Liouville number $L(m)$ \cite{Sh}.  Thus the set
of abyssal real numbers is a strict subset of the irresolute real numbers.  On the other hand, one can easily create Liouville numbers
which are abyssal following a procedure similar to that used in {\it Note} \ref{resoluteexam}.

\begin{prop}  $\uptheta\in\mathfrak{W}$ is abyssal $\Leftrightarrow$ ${\rm Spec}_{\rm flat}(\uptheta )$ is disconnected.
\end{prop}

\begin{proof}  $\Rightarrow$ Suppose $\uptheta$ is abyssal.  Then we may find a best class $\bast\widehat{q}$ for which each element of
the bi-infinite sequence of predecessors and successors 
\[ \cdots \bast\widehat{q}^{--}<\bast\widehat{q}^{-}<\bast\widehat{q}<\bast\widehat{q}^{+}
<\bast\widehat{q}^{++}<\cdots\] defines the same element of $\bstar\PR\R$.  From this it follows that $\widehat{\upmu} =\widehat{\upnu}=\langle\bast\widehat{q}\rangle^{-1}\not\in {\rm Spec}_{\rm flat}(\uptheta )$.  
On the other hand, since $\uptheta\in\mathfrak{W}$, there exist best classes $\bast \widehat{q}_{1},
\bast \widehat{q}_{2}$ with infinite partial quotient for which $\widehat{\upnu}_{1}<\widehat{\upnu}<\widehat{\upnu}_{2}$.
$\Leftarrow$ If $\uptheta$ is abyssless then any best class $\bast \widehat{q}$ which has bounded partial quotient is finitely many successors
as well as finitely many predecessors away from a best class with infinite bounded quotient.  This implies that the common growth and decay
class of $\bast \widehat{q}$,
$\widehat{\upmu}=\widehat{\upnu}$, belongs to ${\rm Spec}_{\rm flat}(\uptheta )$ since it occurs as the right and left endpoint of a pair of adjacent
 best intervals
$[\widehat{\upnu}',\widehat{\upmu}')$, $[\widehat{\upnu}'',\widehat{\upmu}'')$ corresponding
to neighboring best classes with infinite bounded quotient.  Thus ${\rm Spec}_{\rm flat}(\uptheta )$ is connected.
\end{proof}




We have the following strengthening of Corollary \ref{rescor}.

\begin{theo}\label{gapless} If $\uptheta$ is abyssless then $\Omega(\uptheta )=\R-\mathfrak{B}$.  For any $\upeta\in \R-\mathfrak{B}$, 
${\rm Spec}_{\rm flat}(\upeta)\subset {\rm Spec}_{\rm flat}(\uptheta )$.  
\end{theo}

\begin{proof}  Let $\upeta\in \R-\mathfrak{B}$ and let $\bast\widehat{q}'\in\bast\Z (\upeta )$ be a best class with infinite partial quotient
for $\upeta$.  There exist best classes $\bast\widehat{q}_{1}, \bast\widehat{q}_{2}$ with infinite partial quotient
for $\uptheta$ such that $\bast\widehat{q}_{1} <\bast\widehat{q}' <(\bast\widehat{q}')^{+}<\bast\widehat{q}_{2}$.  By the connectivity
of ${\rm Spec}_{\rm flat}(\uptheta )$, $[\widehat{\upnu}',\widehat{\upmu}')\subset {\rm Spec}_{\rm flat}(\uptheta )$.
\end{proof}

\begin{coro} If $\uptheta\in\R-\mathfrak{B}$ ($\uptheta\in \mathfrak{W}_{>1}\cup\Q$) is abyssless, $\uptheta\talloblong\upomega$
($\uptheta\bar{\Updownarrow}\upomega$)
for all
$\upomega\in\R$ with $\uptheta\upomega^{-1}\not\in\mathfrak{B}$. \end{coro}

\begin{coro}\label{flatintcor}  If $\uptheta, \upeta$ are abyssless, then 
${\rm Spec}_{\rm flat}(\upeta)={\rm Spec}_{\rm flat}(\uptheta )$.
\end{coro}

We call this common connected set of Corollary \ref{flatintcor} the 
{\bf flat interval} $\bstar\PR\R_{\rm flat}$.  In particular, $\bstar\PR\R_{\upvarepsilon}-\bstar\PR\R_{\rm flat}$ 
is disjoint from all flat spectra and includes the shift invariant elements. Note also that it is clear that there exist abyssless $\uptheta$, $\upeta$ which are not equivalent
yet their flat spectra coincide: so ${\rm Spec}_{\rm flat}(\uptheta )$ is not a complete invariant of $\uptheta$.  An abyssless number $\uptheta$ is an {\bf  omniflatdivisor}: an element $\uptheta\in\mathfrak{W}$ in which $\uptheta\talloblong\upomega$ for all
$\upomega\in\R$ with $\uptheta\upomega^{-1}\not\in\mathfrak{B}$.


The next result shows that the set $\Upomega (\uptheta)$ need not be equal to $\R-\mathfrak{B}$.

\begin{theo}\label{flatundefined}  There exist abyssal $\uptheta,\upeta\in\mathfrak{W}$ such that $\uptheta {}_{\upmu}\!\owedge_{\upmu}\upeta $
is undefined for all $\upmu$.  In fact, one can find such a pair in which $\uptheta\in\mathfrak{W}_{\upkappa}$, $\upeta\in\mathfrak{W}_{\upkappa'}$ for any 
 $\upkappa,\upkappa'\in[1,\infty]$.
\end{theo}

\begin{proof}  We construct $\uptheta$, $\upeta$ by way of partial fractions. We begin by specifying the initial partial fractional segment 
$ [\text{\bf 1}_{1} a_{m_{1}}]$ of $\uptheta$,  where $\text{\bf 1}_{1}$ is a large block of ones of size
$m_{1}-1$.
Let $q_{m_{1}}=a_{m_{1}}q_{m_{1}-1}+q_{m_{1}-2}$ be the best denominator corresponding to $a_{m_{1}}$.
 Now specify the initial segment for $\upeta$, $[\text{\bf 1}'_{1} b_{n_{1}}]$ so that for $N_{1}$ a fixed integers with $2N_{1}<n_{1}$, there are 
\begin{itemize}
\item[i.] $N_{1}$ best denominators associated to the $1$'s of $\text{\bf 1}'_{1}$ that are less than $q_{m_{1}}$ i.e.\ 
\[ q'_{1},\dots , q_{N_{1}}'<q_{m_{1}}.\]
If necessary we go back and choose $m_{1}$ larger so that this can be done.
\item[ii.]  $N_{1}$ best denominators associated to the 1's of $\text{\bf 1}'_{1}$ which are greater than $q_{m_{1}}$: 
\[   q'_{n_{1}-1},\dots , q_{n_{1}-1-N_{1}}'>q_{m_{1}}.\]
\end{itemize}
The next step is to select $M_{2}>N_{1}$ so that if $\text{\bf 1}_{2}$ is a block of $1's$ having at least $2M_{2}$ elements then at least $M_{2}$ of
the new best denominators associated to the augmented sequence $ [\text{\bf 1}_{1} a_{m_{1}+1} \text{\bf 1}_{2}]$ are less than $q_{n_{1}}'$ i.e.
\[q_{m_{1}+1},\dots , q_{m_{1}+M_{2}}<q'_{n_{1}}.\]
Then choose $m_{2}>2M_{2}$ so that if $\text{\bf 1}_{2}$ has $m_{2}$ elements than at least $M_{2}$ of the  
new best denominators associated to the augmented sequence $ [\text{\bf 1}_{1} a_{m_{1}} \text{\bf 1}_{2}]$ are greater than $q_{n_{1}}'$:
\[q_{m_{1}+m_{2}},\dots , q_{m_{1}+m_{2}-M_{2}}>q'_{n_{1}}.\]
Let $a_{m_{1}+m_{2}}>a_{m_{1}}$ and consider $ [\text{\bf 1}_{1} a_{m_{1}} \text{\bf 1}_{2}a_{m_{1}+m_{2}}]$, applying
the above procedure to specify $ [\text{\bf 1}'_{1} b_{n_{1}} \text{\bf 1}'_{2}b_{n_{1}+n_{2}}]$, where $ b_{n_{1}}<b_{n_{1}+n_{2}}$.
Inductively we
build the sequence of partial quotients of $\uptheta$ and $\upeta$ in this way, arranging that the non $1$ partial quotients
 \[ a_{m_{1}+\cdots +m_{k}}, b_{n_{1}+\cdots +n_{k}},\] as well as the $M_{k},N_{k}$, tend to $\infty$.  See Figure \ref{partialfracconst}.
 \begin{figure}[htbp]
\centering
\includegraphics[width=5in]{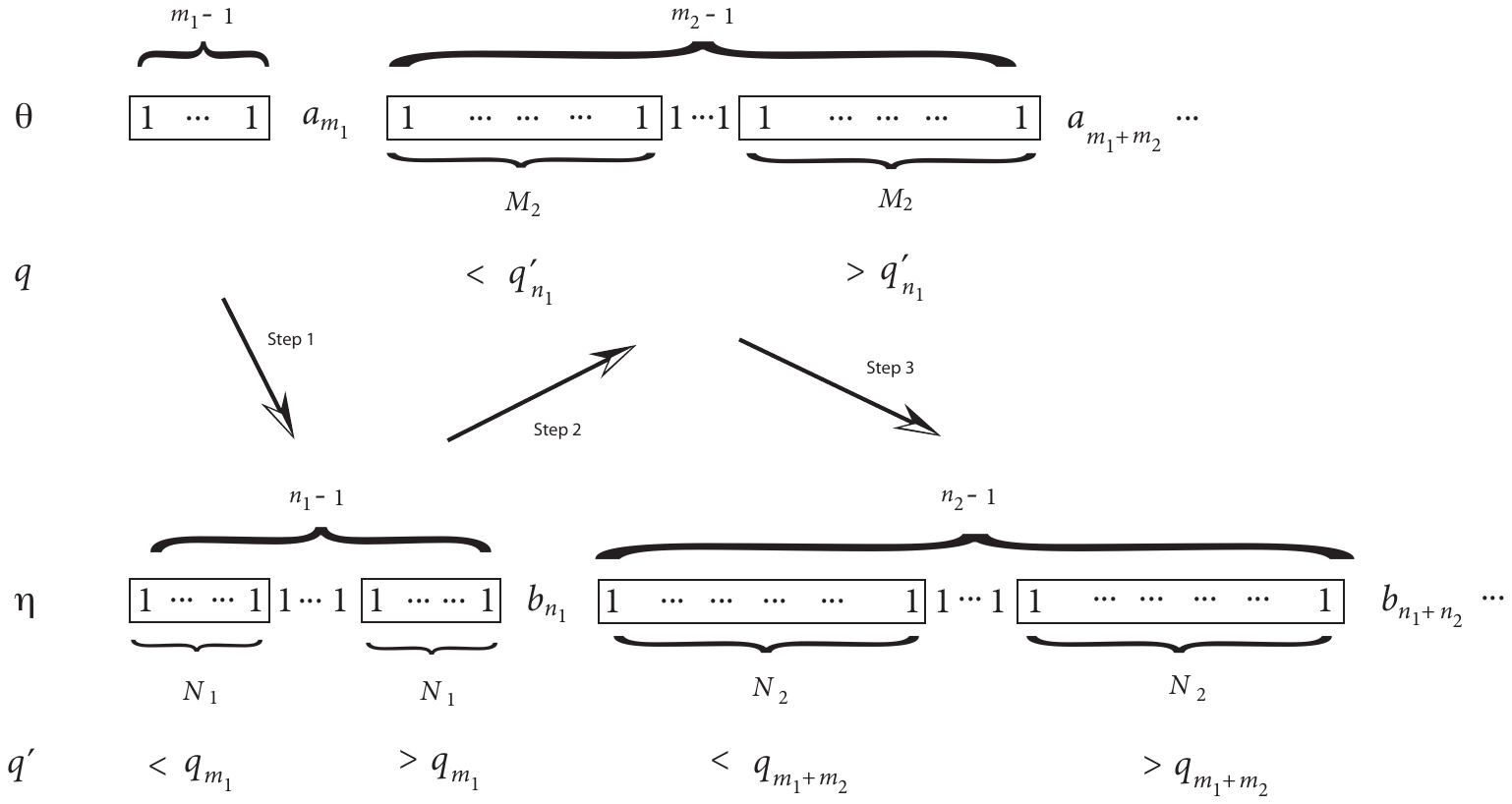}
\caption{Construction of $\uptheta$, $\upeta$.}\label{partialfracconst}
\end{figure} 
 
Consider the sequences
\begin{align}\label{isolatedseq} \{q_{m_{1}+\cdots +m_{k}-1}\}_{k=1}^{\infty},\;\; \{q_{n_{1}+\cdots +n_{k}-1}'\}_{k=1}^{\infty}
\end{align}
and let $\bast \widehat{q}$, $\bast \widehat{q}'$ be best classes for $\uptheta$, $\upeta$ having infinite partial quotient.
By construction of the partial fraction sequences, $\bast \widehat{q}$, $\bast \widehat{q}'$ are classes of sequences formed from the
elements of
the corresponding sequences of (\ref{isolatedseq}).  Moreover, our choices of blocks of 1's in the partial quotients of $\uptheta$ and $\upeta$ ensure that 
\begin{itemize}
\item[-] $(\bast \widehat{q}')^{+}\not\in 
[\bast \widehat{q},\bast \widehat{q}^{+}]$ and that $(\bast \widehat{q}')^{+}, \bast \widehat{q},\bast \widehat{q}^{+}$ define
distinct classes in $\bstar\PR\R$.
\item[-]  $\bast \widehat{q}^{+}\not\in 
[\bast \widehat{q}',(\bast \widehat{q}')^{+}]$ and that $\bast \widehat{q}^{+}, \bast \widehat{q}',(\bast \widehat{q}')^{+}$ define
distinct classes in $\bstar\PR\R$.
\end{itemize}
  This implies
that
the growth decay interval 
$[\widehat{\upmu}^{+},\widehat{\upmu})$ corresponding to $\bast \widehat{q}$ cannot contain the decay $(\widehat{\upmu}')^{+}$ 
of $\bast \widehat{q}'$, and {\it vice verca}.
By Lemma \ref{firstflatcond}, $\uptheta {}_{\upmu}\!\owedge_{\upmu}\upeta $
is undefined for all $\upmu$.  By choosing the non 1 partial fractions appropriately we can ensure that  $\uptheta\in\mathfrak{W}_{\upkappa}$, $\upeta\in\mathfrak{W}_{\upkappa'}$ for any fixed
 $\upkappa,\upkappa'\in[1,\infty]$.
\end{proof}

It follows that the problems of determining 
$ \Upomega (\uptheta )$ and flat composability are non trivial.  
Nevertheless, it seems plausible that the techniques in \cite{Sh} can be extended to show that the Liouville numbers occurring in the sum and product representations \cite{Er} are abyssless.  

\begin{conj}  $\mathfrak{W}_{\infty}$ strongly approximately generates $\R$.
\end{conj}

\section{Symmetric Diophantine Approximations}\label{metrical}

The study of special approximations in which the the error term is dominated
by a function
$ \uppsi:\Z-\{0\}\rightarrow\R  $
has held, from the very beginning, a distinguished position in the subject of Diophantine Approximation.
Classically, for a fixed $\uppsi$, one looks for conditions on $\uptheta$ which guarantee the existence of infinitely many
solutions to the inequality $|n\uptheta-m|<|\uppsi (n)|^{-1}$, or equivalently, in the language of this paper, a single solution to
\begin{align}\label{psiinequality} 
| \upvarepsilon (\bast n )| =  | \bast n\uptheta -\bast n^{\perp}  | & < |\uppsi( \bast n )| ,\quad  \bast n\in\bast\Z (\uptheta ). 
\end{align}
The Theorems of Dirichlet, Liouville and Roth all fall under this heading.  More generally, if one only specifies convergence
properties of the sum $\sum \uppsi (n)$, one seeks (e.g. Khintchine's Theorem) to measure the size of the set of real numbers
having solutions to (\ref{psiinequality}).

In this section, we will shift the emphasis from one of existence
to a qualitative study of the solution set:
\begin{align*}     \bast\Z (\uptheta |\uppsi) & := \left\{ 0\not= \bast n\in \bast\Z (\uptheta )| \; \text{there exists }C>0\text{ such that }   | \upvarepsilon (\bast n )|< 
C |\uppsi( \bast n )| \right\}\cup\{ 0\},
\end{align*} 
focusing on the extent to which the set $\bast\Z (\uptheta |\uppsi ) $ has interesting arithmetic structure.  

We begin by fixing the choice $\uppsi(x) =x^{-1}$: note then that 
\begin{align}\label{decayltoregrowth} 
 \bast\Z\big(\uptheta | x^{-1} \big)=\left\{  \bast n\in \bast\Z (\uptheta )| \upnu (\bast n)\leq \upmu (\bast n )\right\},
 \end{align}
a set which is closed under the operation of taking additive inverses.

It turns out that the same hypothesis used to describe the approximate ideal product
structure of the $\bast\Z (\uptheta )$ can be used to deduce an {\it additive}
structure in 
$ \bast\Z\big(\uptheta | x^{-1} \big)$.
To this end, denote
\[    \bast\Z^{\upmu}_{\upnu}\big(\uptheta |  x^{-1} \big) :=\bast\Z^{\upmu}_{\upnu}\big(\uptheta )\cap  \bast\Z (\uptheta | x^{-1} \big)  .\]
Note that the defining condition $\upnu (\bast n)\leq \upmu (\bast n )$ in (\ref{decayltoregrowth}) {\it does not} imply that 
$\bast\Z^{\upmu}_{\upnu}\big(\uptheta |  x^{-1} \big)=0$ for $\upmu<\upnu$.
Recall that tropical subtraction in $\bstar\PR\R$ is defined 
\[  \upmu- \upnu :=\min (\upmu, \upnu ).  \]

\begin{theo}\label{symmgroupology}  Let $\upmu,\upnu,\upiota,  \uplambda\in \bstar\PR\R_{\upvarepsilon}$.  
Then
\[  \bast\Z^{\upmu [\upiota]}_{\upnu}\big(\uptheta \big| x^{-1}\big)+ \bast\Z^{\upnu [ \uplambda ]}_{\upmu}\big(\uptheta \big| x^{-1} \big)
\subset   \bast\Z^{\upmu-\upnu [\upiota+ \uplambda]}\big(\uptheta \big|x^{-1} \big).  \]
In particular, 
\[   \bast\Z^{\upmu}_{\upnu}\big(\uptheta \big| x^{-1} \big) + \bast\Z^{\upnu}_{\upmu}\big(\uptheta \big| x^{-1} \big)\subset
  \bast\Z^{\upmu-\upnu} \big(\uptheta \big|x^{-1} \big).
  \]
\end{theo}

\begin{proof}  Let $\bast m\in  \bast\Z^{\upmu [\upiota]}_{\upnu}\big(\uptheta | x^{-1} \big)$, $\bast n\in  \bast\Z^{\upnu[  \uplambda]}_{\upmu}\big(\uptheta | x^{-1} \big)$.  Then
\[ \big| (\bast m + \bast n)^{2} \uptheta -(\bast m + \bast n)(\bast m^{\perp} + \bast n^{\perp})  \big|
\leq C + \big|2\bast m\bast n\uptheta -(\bast m \bast n^{\perp} +\bast m^{\perp} \bast n)   \big|\]
for some constant $C$.
But
\begin{align*}  \left|2\bast m\bast n\uptheta -(\bast m \bast n^{\perp} +\bast m^{\perp} \bast n)   \right| & = 
\uptheta^{-1} \left|( \bast m\uptheta -\bast m^{\perp})( \bast n\uptheta -\bast n^{\perp}) +(\bast m\bast n\uptheta^{2} - \bast m^{\perp}\bast n^{\perp})    \right|  \\
 & = \uptheta^{-1}\left|\upvarepsilon (\bast m )\cdot\upvarepsilon(\bast n) + \upvarepsilon (\bast m\bast n)\right|.
\end{align*}
which is infinitesimal (here we are using Theorem \ref{productformula} to conclude that
$\upvarepsilon (\bast m\bast n)=\bast m\bast n\uptheta^{2} - \bast m^{\perp}\bast n^{\perp}$ is 
the error term of a diophantine approximation, hence is infinitesimal).  Thus
$\bast m +\bast n\in \bast\Z\big(\uptheta \big| x^{-1} \big)$.  Finally we note that
\[   (\bast m+\bast n)\cdot (\upmu-\upnu) \leq \bast m\cdot (\upmu-\upnu) +
\bast n\cdot (\upmu-\upnu) <\upiota +  \uplambda. \]
Thus $\bast m+\bast n\in  \bast\Z^{\upmu-\upnu[ \upiota+ \uplambda]}\big(\uptheta \big|x^{-1} \big)$.
\end{proof}


With its sum partially defined along the growth-decay filtration, we refer to 
$\bast\Z\big(\uptheta \big| x^{-1} \big)$ as an  {\bf  approximate group};
the sum is referred to as the {\bf  growth-decay sum} or {\bf  approximate group sum}.
We say that the approximate group $\bast\Z\big(\uptheta \big| x^{-1} \big)$ is {\bf  barren} if there are no
non-trivial approximate group sums.

\begin{coro}\label{BAantiprime} Let $\uptheta\in\R-\Q$.  Then 
$\bast\Z\big(\uptheta \big| x^{-1} \big)$ is barren $\Leftrightarrow$
$\uptheta\in\mathfrak{B}$.
\end{coro}

\begin{proof} $\Leftarrow$ Immediate from Theorem \ref{badapproxchar}.  $\Rightarrow$ If $\uptheta\in\mathfrak{W}$
there exists $\upmu$
for which $\bast\Z^{\upmu}_{\upmu}\big(\uptheta |x^{-1} \big)\not =0$.  In this case we have the non-trivial
growth-decay sum
\[\bast\Z^{\upmu }_{\upmu}\big(\uptheta | x^{-1} \big)+ \bast\Z^{\upmu}_{\upmu}\big(\uptheta | x^{-1} \big)
\subset   \bast\Z^{\upmu}\big(\uptheta |x^{-1} \big).\]

\end{proof}



Recall \cite{Ge4} the real vector space $\bbull\R := \bast\R/\bast\R_{\upvarepsilon}\supset\R$ of {\bf  extended reals} (not considered here as a topological vector space).
We define a function $|\cdot |_{\uptheta}:\bast\Z\rightarrow \bbull\R_{+} $ 
by
\[  |\bast n|_{\uptheta} := \big(|\bast n|\cdot \|\bast n\uptheta\|\big)^{1/2}\mod \bast\R_{\upvarepsilon}.  \]
Note that when $\bast n\in \bast\Z (\uptheta)$, $\|\bast n\uptheta\|= |\upvarepsilon( \bast n)|$.
By definition, $ |\bast n|_{\uptheta} \in\R_{+}$ for all $\bast n\in \bast\Z\big(\uptheta | x^{-1} \big)$.  
Somewhat abusively, we refer to $|\cdot|_{\uptheta}$ as the {\small $\boldsymbol\uptheta${\bf -norm}}; while it is technically not a norm,
it may be viewed as a generalized norm in a sense which will soon be made clear.
We have immediately:

\begin{prop}\label{trivialnorm} $|\cdot |_{\uptheta}\equiv 0$ on $\bast\Z^{\upmu}_{\upnu}\big(\uptheta | x^{-1} \big)$ for $\upmu\geq\upnu$.
\end{prop} 

\begin{note} For each $\uptheta\in\mathfrak{B}$, let $C_{\uptheta}>0$ be the supremum of constants
$C$ for which $\|n\uptheta\|<Cx^{-1}$ has only finitely many solutions.  The set of such
$C_{\uptheta}$ as one ranges over $\uptheta\in\mathfrak{B}$ is called the Lagrange spectrum \cite{Sch}.
Note that if $\uptheta\in\mathfrak{B}$ and $C_{\uptheta}>0$ is
the associated element of the Lagrange spectrum then for all $0\not= \bast n\in \bast\Z \big(\uptheta |x^{-1}\big)$
\[  |\bast n|_{\uptheta} \geq C^{1/2}_{\uptheta}. \]
Thus for badly approximable numbers, the $\uptheta$-norm is always positive on non $0$ elements.  
\end{note}

For arbitrary $\uptheta\in\R-\Q$, do there exist any $\bast n\in \bast\Z\big(\uptheta | x^{-1} \big)$
for which  $|\bast n|_{\uptheta}\not=0$?  
If there exists such an $\bast n$ then
there are positive real constants $c<C$ such that some representative sequence $\{n_{i}\}$ satisfies the inequality
\begin{align}\label{symmdefinition}    \frac{c}{n_{i}}\leq |n_{i}\uptheta - n^{\perp}_{i}| \leq  \frac{C}{n_{i}} \end{align}
i.e. $\langle|\bast n|\cdot |\upvarepsilon(\bast n ) |\rangle = \langle|\bast n|\rangle\cdot \upnu(\bast n )=1$ , or equivalently,
\[\upmu(\bast n)=\upnu(\bast n).
\]
We call such a class $\bast n$ a {\bf  symmetric diophantine approximation}, the set of which union $0$ is denoted 
\[ \bast\Z^{\rm sym}(\uptheta ):=|\cdot|_{\uptheta}^{-1}(0,\infty )\cup\{ 0\}\subset \bast\Z \big(\uptheta | x^{-1} \big).\]
We have trivially that $\bast n\in \bast\Z^{\rm sym}(\uptheta )$ $\Leftrightarrow$ $N\cdot \bast n\in  \bast\Z^{\rm sym}(\uptheta )$ for $N\in\Z -\{ 0\}$.

We will show that $\bast\Z^{\rm sym}(\uptheta )\not=0$ for all $\uptheta\in\R-\Q$.  For each $\upnu\in\bstar\PR\R_{\upvarepsilon}$, write 
\[ \bast\Z^{\rm sym}_{\upnu}(\uptheta )= \{\bast n\in\bast\Z^{\rm sym} (\uptheta )|\;  \upnu(\bast n) = \upnu\},\]
so that $\bast\Z^{\rm sym}(\uptheta )=\bigcup_{\upnu} \bast\Z^{\rm sym}_{\upnu}(\uptheta)$.  The following
Proposition identifies $\bast\Z^{\rm sym}_{\upnu}(\uptheta )$ as the part of $\bast\Z_{\upnu}(\uptheta )$
inhabiting the narrow space between the intersection of the slow diophantine approximations of decay $\upnu$
and the flat diophantine approximations of decay $\upnu$.

\begin{prop}\label{diffandsym}  Let $\uptheta\in\R$.  Then
$ \bast\Z^{\rm sym}_{\upnu}(\uptheta) =
 \bigg( \bigcap_{\upmu<\upnu} \bast\Z^{\upmu}_{\upnu} (\uptheta )\bigg) - 
 \bast\Z^{\upnu}_{\upnu} (\uptheta) $.    
 \end{prop}
 
 \begin{proof}  $\bast n \in \bast\Z^{\rm sym}_{\upnu}(\uptheta)$ $\Leftrightarrow$ 
 $\bast n\cdot \upnu= \bast n\cdot \upnu (\bast n ) =1$ in $\bstar\PR\R$ $\Leftrightarrow$ 
 $\bast n\not\in  \bast\Z^{\upnu}_{\upnu} (\uptheta) $ and $\bast n\in  \bast\Z^{\upmu}_{\upnu} (\uptheta)$
 for all $\upmu<\upnu$. 
  \end{proof}
  
  \begin{coro}   If $\uptheta\in\mathfrak{B}$ then \[  \bast\Z^{\rm sym}_{\upnu}(\uptheta) =
  \bigcap_{\upmu<\upnu} \bast\Z^{\upmu}_{\upnu} (\uptheta ).\]
  In particular, $\bast\Z^{\rm sym}_{\upnu}(\uptheta)$ is a subgroup of $\bast\Z_{\upnu}(\uptheta)$.
  \end{coro}
  
  
  Recall (see \S \ref{nonvanspec}) the set of best growths $\bstar\PR\R_{\upvarepsilon}^{\rm bg}\subset\bstar\PR\R_{\upvarepsilon}$.
 
 \begin{theo}\label{symnotempty}  For all $\uptheta\in \R-\Q$ and all $\upnu\in\bstar\PR\R_{\upvarepsilon}^{\rm bg}$, $\bast\Z_{\upnu}^{\rm sym}(\uptheta)\not=0$.  In particular, if $\uptheta\in\mathfrak{B}$ then $\bast\Z^{\rm sym}_{\upnu}(\uptheta)\not=0$ for all $\upnu\in\bstar\PR\R_{\upvarepsilon}$.
 \end{theo}
 
 \begin{proof}  For $\uptheta\in\mathfrak{B}$ the result is
 obviously true: any best class $\bast\widehat{q}\in \bast\Z^{\rm sym}(\uptheta)$, and moreover,
 we may realize any growth index $\upnu$ as a best growth by Theorem \ref{badapproxchar}. 
 Now assume that $\uptheta\not\in\mathfrak{B}\cup\Q$.  Let $\uptheta = [a_{1}a_{2}\dots ]$.  By Grace's Theorem \cite{La1},
 the intermediate best denominator (= denominator of the intermediate
 convergent) $q_{n,r}= rq_{n+1}+q_{n}$, where $0\leq r< a_{n+2}$, satisfies
 \[ 1< |q_{n,r}|\cdot\| q_{n,r}\uptheta\|\]
 for all $r\not= 0,1,a_{n+2}-1$.  For any infinite sequence of such intermediate convergents we take $c=1$ in (\ref{symmdefinition}).  On the other hand (Lemma
 I.4 of \cite{La1}), we have
\[    |q_{n,r}|\cdot\| q_{n,r}\uptheta\| = \frac{q_{n,r}(\uptheta_{n+2}-r)}{\uptheta_{n+2}q_{n+1}+q_{n}}= 
 \frac{(rq_{n+1}+q_{n})(\uptheta_{n+2}-r)}{\uptheta_{n+2}q_{n+1}+q_{n}}\] 
where (as in Theorem \ref{flatmultibest}) $\uptheta_{i}$ is defined by the equation $\uptheta = [a_{1}\dots a_{i-1}\uptheta_{i}]$.  We note that $r<a_{n+2}<\uptheta_{n+2}$.
Dividing out numerator and denominator by the dominant term $\uptheta_{n+2}q_{n+1}$ gives
\begin{align*}
 |q_{n,r}|\cdot\| q_{n,r}\uptheta\|  & = 
\frac{\bigg( \big(r/\uptheta_{n+2}\big)+ \big(q_{n}\big/(\uptheta_{n+2}q_{n+1})\big)\bigg)\big(\uptheta_{n+2}-r\big)}{1+ \big(q_{n}\big/(\uptheta_{n+2}q_{n+1})\big)} \\
& \\
& < r+ \frac{q_{n}}{q_{n+1}} -r\left(\frac{r}{\uptheta_{n+2}}+\frac{q_{n}}{\uptheta_{n+2}q_{n+1}}\right) \\
& \\
& < r+1.
\end{align*}
Then if we choose $r_{i}$ bounded and $\not= 0,1,a_{n_{i}+2}-1$, the class $\bast n =\bast \{q_{n_{i},r_{i}}\}$ will be symmetric.  It follows that
for any best class $\bast\widehat{q}$ there is a symmetric class $\bast n$ such that $\upmu (\bast\widehat{q}) = \upmu (\bast n )$.  
 \end{proof}
 
 \begin{coro}  If $\uptheta\in\mathfrak{B}$ then $\bast\Z^{\rm sym}(\uptheta)$ contains both the best denominator classes
 as well as the intermediate best denominator classes.
 \end{coro}
 
 \begin{proof}  From the proof of Theorem \ref{symnotempty}, we know that the intermediate best denominators of the form
 $\bast n =\bast \{q_{n_{i},r_{i}}\}$, where $r_{i}\not= 0,1,a_{n_{i}+2}-1$ and is uniformly bounded, belong to $\bast\Z^{\rm sym}(\uptheta)$.  On the
 other hand, since $\uptheta\in\mathfrak{B}$, the best denominators (which occur for $r_{i}=0,1$) as well as the consecutive
 difference $\bast \widehat{q}^{+}-\bast \widehat{q}$ (which occurs for $r_{i}=a_{n_{i}+2}-1$) belong to $\bast\Z^{\rm sym}(\uptheta)$.
 \end{proof}
 
 \begin{note}  As the above paragraphs show, the function $|\cdot |_{\uptheta}$ is nontrivial for all $\uptheta\in\R-\Q$.  We take a moment to contrast it with its rational and $p$-adic analogs.    
  \begin{itemize}
\item[i.]   If $\uptheta=q=a/b\in\Q$ then $|\cdot |_{q} \equiv 0$ on $\bast\Z (\uptheta ) =  \bast\Z(\uptheta | x^{-1} )$.
For $\bast n\in \bast\Z$ arbitrary, $|\bast n |_{q}=c\cdot |\bast n|$ where $c=a'/b$ for some $a'<b$.  In fact, $|\cdot|_{q}$
induces a function \[ |\cdot|_{q}:\bast\Z/\bast\Z (\uptheta ) \cong \Z/b\Z\longrightarrow\bast\Q/\bast\Z\cong\Q/\Z.\]
\item[ii.] If $\upxi\in\Q_{p}$ = $p$-adic numbers and we use the $p$-adic absolute value
 to define the distance-to-the-nearest-integer function $\|\cdot \|$, then $\|\cdot \|\equiv 0$ on $\bast\Z (\upxi ) =  \bast\Z(\upxi | x^{-1} )$.  Therefore,
for $\bast n\in \bast\Z$ arbitrary, $|\bast n |_{\upxi}\leq |\upxi|_{p}$ = the $p$-adic absolute value; if $\upxi\in\hat{\Z}_{p}$ then $|\cdot |_{\upxi}\equiv 0$.
\end{itemize}
\end{note}

\vspace{3mm}
\begin{center}
$\maltese$
\end{center}
\vspace{3mm}
 
Recall \cite{Qu} that the {\em Littlewood conjecture} asserts that for any pair $\uptheta,\upeta\in\mathfrak{B}$, 
\[  \lim\inf |n|\|n\uptheta \|\|n\upeta\| =0 . \]

\begin{obse}\label{LCCoro}  Given $\uptheta,\upeta\in\mathfrak{B}$, suppose that $\exists \upnu\in\bstar\PR\R_{\upvarepsilon}$
such that
\[  \{ 0\}\subsetneqq\bast\Z^{\rm sym}_{\upnu}(\uptheta)\cap \bast\Z^{\upnu}(\upeta) \quad \text{ or }\quad
 \{ 0\}\subsetneqq \bast\Z^{\upnu}(\uptheta)\cap\bast\Z^{\rm sym}_{\upnu}(\upeta)
. \]
Then the Littlewood conjecture holds for the pair $\uptheta,\upeta$.
\end{obse}

 In view of Observation \ref{LCCoro} it would be of interest to find an explicit description of $\bast\Z^{\rm sym}(\uptheta)$.
 When $\uptheta=\upvarphi = (\sqrt{5}+1)/2$ is the golden mean there are many symmetric diophantine approximations which are neither intermediate nor principle convergents; we characterize them now. Recall \cite{Ze} that every natural number has a unique {\em  Zeckendorf
 representation} 
 \begin{align}\label{ZeckRep}  N = F_{i_{1}} + \cdots + F_{i_{k}}
 \end{align}
 where $F_{k}=$ the $k$th Fibonacci number, and the sequence $i_{1}<\dots <i_{k}$ consists of non consecutive integers $\geq 2$.
Using Binet's formula \cite{NZM}
\[ F_{k}= \frac{\upvarphi^{k}-(-1)^{k}\upvarphi^{-k}}{\sqrt{5}} \]
one has the following 

\begin{lemm}\label{GMErrorLemma}  Let $N\in\N$ have the Zeckendorf representation (\ref{ZeckRep}).  Then $\| N\upvarphi\|<\upvarphi^{-n}$ $\Leftrightarrow$
\begin{itemize}
\item $i_{1}\geq n+1$.
\item $i_{1}=n$ and $i_{2}-i_{1}$ is odd and $\geq 3$.
\end{itemize}

\end{lemm}

\begin{proof}  See \cite{CaGe}.  
\end{proof}

Define the {\bf  Zeckendorf degree} of the representation (\ref{ZeckRep}) to be 
\[ \text{\rm Z-deg}(N):=i_{k}-i_{1}.\]  For $\bast n\in\bast\Z$ the Zeckendorf
degree $\text{\rm Z-deg}(\bast n)$  is an element of $\bast\N$.

\begin{theo}\label{goldensymmetric}  Let $\upvarphi$ be the golden mean.  Then  
\[  \bast\Z^{\rm sym}(\upvarphi)=\{ \bast n\in\bast\Z (\upvarphi)|\; \text{\rm Z-deg}(\bast n)<\infty \}.\] 
\end{theo}

\begin{proof}  Let $\bast n = \bast \{ n_{i}\}$ and let $M=\text{\rm Z-deg}(\bast n )$.   Then by Lemma \ref{GMErrorLemma} and Binet's formula
\[  |\bast n|\cdot |\upvarepsilon (\bast n )| \leq C\upvarphi^{M} \]
for $C>0$ a constant that depends only on $\upvarphi$.  On the other hand, if $\bast M=\text{\rm Z-deg}(\bast n )\in\bast \N-\N$
then  
\[ |\bast n|\cdot |\upvarepsilon (\bast n )|\geq  C\upvarphi^{\bast M}  . \]
\end{proof}

If we consider the Littlewood conjecture in the case of $\upvarphi$, we see that the following is true.  Let 
\[ \Z_{D}[\upvarphi ]\subset\Z [\upvarphi]\] be the set of 
$\Q(\sqrt{5})$-integers of the form
\[  \upvarphi^{I}:= \upvarphi^{i_{1}}+\cdots + \upvarphi^{i_{k}}  \]
with $I=(i_{1},\cdots ,i_{k})$ a sequence of increasing, non consecutive integers, $i_{1}\geq 2$ and $i_{k}-i_{1}< D$.

\begin{coro}\label{genlchowla} Let $\uptheta\in\mathfrak{B}$, $\uptheta\not\in \Q (\sqrt{5} )$.  If for some $D$ 
\[ \big\{ \| \upvarphi^{I}\uptheta \|\, \big|\;\upvarphi ^{I}\in \Z_{D}[\upvarphi ]\big\}  \]
is dense in $[0,1/2)$ then the Littlewood conjecture holds for the pair $(\upvarphi,\uptheta )$.
\end{coro} 

\begin{note}  The hypothesis of Corollary \ref{genlchowla} in the case of $D=1$ reduces to Chowla's
conjecture in the case of the golden mean \cite{Qu}.
\end{note}

\vspace{3mm}
\begin{center}
$\maltese$
\end{center}
\vspace{3mm}
 
 We turn to the matter of the general arithmetic structure of $\bast\Z^{\rm sym} (\uptheta )$, describing $\bast\Z^{\rm sym} (\uptheta )=\{ \bast\Z_{\upnu}^{\rm sym} (\uptheta )\}$
 as a subapproximate group of $ \bast\Z (\uptheta|x^{-1} )$.  For every sign pair $\upsigma\in \{ \pm\}^{2}$, let $\bast\Z (\uptheta )_{\upsigma}$ be the monoid consisting of $0$ and
 those diophantine approximations $\bast n\in \bast\Z (\uptheta )$ for which $({\rm sign}(\bast n), {\rm sign}(\upvarepsilon(\bast n)))=\upsigma$.
 Let $\bast\Z_{\upnu}^{\rm sym} (\uptheta )_{\upsigma} = \bast\Z_{\upnu}^{\rm sym} (\uptheta )\cap \bast\Z (\uptheta )_{\upsigma}$.
 
 \begin{theo}\label{symmgroupologystruct}  Let $\uptheta\in\R$.
 \begin{itemize}
 \item[1.] $\bast\Z^{\rm sym} (\uptheta )=\{ \bast\Z_{\upnu}^{\rm sym} (\uptheta )\} $ satisfies
 \[ \bast\Z_{\upnu}^{\rm sym} (\uptheta )+\bast\Z_{\upnu}^{\rm sym} (\uptheta )\subset 
 \bast\Z \big(\uptheta | x^{-1} \big).\]  
 \item[2.] For $\upsigma=(+,+)$ or $(-,-)$ and $\upnu\in\bstar\PR\R_{\upvarepsilon}$, $\bast\Z_{\upnu}^{\rm sym} (\uptheta )_{\upsigma}$ is a monoid:
 \[ \bast\Z_{\upnu}^{\rm sym} (\uptheta )_{\upsigma}+\bast\Z_{\upnu}^{\rm sym} (\uptheta )_{\upsigma}\subset 
 \bast\Z_{\upnu}^{\rm sym} (\uptheta )_{\upsigma}.\] 
 \end{itemize}
 \end{theo}
 
 \begin{proof} The first assertion is a consequence of the inequality
 \begin{align*}
 |\bast m+\bast n|\cdot| \upvarepsilon(\bast m+\bast n)| & \leq
 M + |\bast n||\upvarepsilon(\bast m)| + |\bast m||\upvarepsilon(\bast n)|: 
 \end{align*}
 as the terms $|\bast n||\upvarepsilon(\bast m)|$, $|\bast m||\upvarepsilon(\bast n)|$ are bounded.  The second assertion follows
 by noting that for $\bast m,\bast n\in \bast\Z_{\upnu}^{\rm sym} (\uptheta )_{\upsigma}$, 
\[ \langle|\bast m+\bast n|\rangle=\langle|\bast m|\rangle+ \langle|\bast n|\rangle=\upnu^{-1},\quad \langle|\upvarepsilon (\bast m +\bast n)|\rangle =\langle|\upvarepsilon (\bast m )|\rangle+ \langle|\upvarepsilon (\bast n)|\rangle=\upnu.\]
 \end{proof}
 
 \begin{coro}  If $\uptheta\in\mathfrak{B}$ then $ \bast\Z^{\rm sym} (\uptheta ) =\{ \bast\Z^{\rm sym}_{\upnu} (\uptheta )\}$ is a family of groups
 satisfying 
 \begin{align}\label{quasiideo} \bast\Z^{\upnu[\upiota]}\cdot \bast\Z^{\rm sym}_{\upnu} (\uptheta )\subset \bast\Z^{\upnu^{2}}_{\upiota}(\uptheta ).
 \end{align}
 \end{coro}
 
 \begin{proof} If $\uptheta\in\mathfrak{B}$ then $ \bast\Z \big(\uptheta | x^{-1} \big)= \bast\Z^{\rm sym} (\uptheta )$.  By 1. of
 Theorem \ref{symmgroupologystruct}, $\bast\Z^{\rm sym}_{\upnu} (\uptheta )$ is a group for all $\upnu\in\bstar\PR\R_{\upvarepsilon}$.
 The property (\ref{quasiideo}) follows immediately from the definitions.
 \end{proof}
 
 The symmetric set $\bast\Z_{\upnu}^{\rm sym} (\uptheta )$ has, in addition, a fractional addition/multiplication law
 which generalizes the approximate ideal product of Theorem \ref{productformula}, and which is nontrivial even for $\uptheta\in\mathfrak{B}$.  To formulate it
it is necessary to work with numerator denominator pairs rather than just denominators. 
 Thus let 
 \[ \bast\Z^{1,1} (\uptheta ) = \{ (\bast n^{\perp},\bast n)|\; \bast n\in\bast\Z (\uptheta )\}=
 \{ (\bast m,\bast n)\in\bast\Z^{2}|\; \langle |\bast n\uptheta -\bast m|\rangle< 1\}\subset\bast\Z^{2}\]
 be the associated group of numerator denominator pairs of diophantine approximations.  The canonical isomorphism
 $ \bast\Z^{1,1} (\uptheta )\cong\bast\Z (\uptheta)$ induces the growth-decay filtration
 $ \bast\Z^{1,1} (\uptheta ) =\{ (\bast\Z^{1,1})^{\upmu}_{\upnu} (\uptheta )\}$.  The set of symmetric numerator denominator pairs is denoted
 \[ (\bast\Z^{1,1})_{\upnu}^{\rm sym} (\uptheta )\subset (\bast\Z^{1,1})_{\upnu} (\uptheta ).\]  
 
 Now define
 \[   \bast\mathring{\Z}^{1,1} (\uptheta ) =\{ (\bast m,\bast n)\in\bast\Z^{2}|\; \langle |\bast n\uptheta -\bast m|\rangle\leq 1\}\supsetneqq  \bast\Z^{1,1} (\uptheta ) . \]
 Note that $\Z^{2}\subset \bast\mathring{\Z}^{1,1} (\uptheta )$, and if $(\bast m,\bast n)\in\bast\mathring{\Z}^{1,1} (\uptheta )-\Z^{2}$
 then both $\bast m$ and $\bast n$ are infinite.  In addition, for all infinite $(\bast m,\bast n)\in \bast\mathring{\Z}^{1,1} (\uptheta )$,
 $\bast m/\bast n\simeq \uptheta$ in $\bast\R$.  If we form the quotient group
 \[   \bast\overline{\Z}^{1,1} (\uptheta ) :=\bast\mathring{\Z}^{1,1} (\uptheta )/\Z^{2} \]
 then every class $ \bast \bar{n}\in\bast \overline{\Z}:=\bast\Z/\Z$ (= the group of universes in $\bast\Z$, a densely ordered group) determines a unique numerator $\bast  \bar{m}\in\bast \overline{\Z}$ for which $(\bast  \bar{m},\bast  \bar{n})\in \bast\overline{\Z}^{1,1} (\uptheta )$
 and we write $ \bast \bar{n}^{\perp}=\bast  \bar{m}$.  
 Since elements of  $\bast\Z^{1,1} (\uptheta )$ are already uniquely determined by their denominator, there is an induced
 inclusion 
 \[  \bast\Z^{1,1} (\uptheta )\hookrightarrow  \bast\overline{\Z}^{1,1} (\uptheta ). \]

  For 
  $( \bast \bar{n}^{\perp}, \bast \bar{n})\in  \bast\overline{\Z}^{1,1} (\uptheta )$ write
  $\upnu ( \bast \bar{n})= \upnu (\bast n)$ if $\bast n\in\bast \bar{n}$ belongs to $\bast\Z (\uptheta )$; otherwise write
  $\upnu ( \bast \bar{n})=1$.
In addition, write $\upmu (\bast\bar{n})=\langle\bast n^{-1}\rangle$ for any $\bast n\in\bast\bar{n}$, which is evidently independent
of the choice of representative. 

 Let $\bstar\PR\R_{\leq1}= \bstar\PR\R_{\upvarepsilon}\cup\{1\}$.
  Now for $\upmu\in\bstar\PR\R_{\upvarepsilon},\upnu\in \bstar\PR\R_{\leq1}$ define
 \[   (\bast\overline{\Z}^{1,1})^{\upmu}_{\upnu} (\uptheta )=\{ (\bast \bar{m},\bast \bar{n}) |\; \bast \bar{n}\cdot \upmu 
 \in \bstar\PR\R_{\upvarepsilon}, \upnu ( \bast \bar{n})\leq\upnu
 \} .\]
 Note that for all $\upnu<1$,  $(\bast\overline{\Z}^{1,1})^{\upmu}_{\upnu} (\uptheta )= (\bast\Z^{1,1})^{\upmu}_{\upnu} (\uptheta )\cong
 \bast\Z^{\upmu}_{\upnu} (\uptheta )$.  The proof of the following Theorem is left to the reader, who will note that it is available for elements of $\mathfrak{B}$, providing the latter with a weak form of approximate ideal arithmetic.
 
 \begin{theo}[Symmetric Approximate Ideal Arithmetic]\label{symgda}  For any $\uptheta,\upeta\in\R$ and all $\upnu\in\bstar\PR\R_{\upvarepsilon}$,
 there are maps
 \[\cdot, \pm\;: (\bast\Z^{1,1})_{\upnu}^{\rm sym} (\uptheta )\times(\bast\Z^{1,1})_{\upnu}^{\rm sym} (\upeta )\longrightarrow
(\bast\overline{\Z}^{1,1})_{1}^{\upnu^{2}} (\uptheta \upeta),\;
(\bast\overline{\Z}^{1,1})_{1}^{\upnu^{2}} (\uptheta \pm\upeta) 
 \]
 given by fractional product, sum and difference of pairs.
 \end{theo}

 \section{Lorentzian Structure}\label{Lorentzian}
 
We now make precise the extent to which we may regard $|\cdot |_{\uptheta}$ as a generalized norm. 
The following result suggests that we may view $|\cdot |_{\uptheta}$ as a pseudo-norm on the approximate group
$\bast\Z\big(\uptheta | x^{-1} \big)$.

\begin{theo}\label{nonarchtriangle}  The restriction of the function $|\cdot |_{\uptheta}$ to $ \bast\Z(\uptheta | x^{-1} )$ 
obeys the non-archimedean triangle inequality for all defined (i.e.\ approximate group) sums.\end{theo}

\begin{proof}  Let $\bast n_{1}\in  \bast\Z^{\upmu}_{\upnu}\big(\uptheta | x^{-1} \big)$, $\bast n_{2}\in  \bast\Z^{\upnu}_{\upmu}\big(\uptheta | x^{-1} \big)$ so that the sum 
$\bast n_{1} +\bast n_{2}$ is defined.  Then
\[  |\bast n_{1} +\bast n_{2} |^{2}_{\uptheta} \leq ( |\bast n_{1}| + |\bast n_{2}|  )(|\upvarepsilon ( \bast n_{1})|  +  | \upvarepsilon (\bast n_{2})|    ) \mod \bast\R_{\upvarepsilon} . \]
By hypothesis we have that the elements 
$   |\bast n_{1}|\cdot | \upvarepsilon ( \bast n_{2})| $,   $|\bast n_{2}|\cdot |\upvarepsilon ( \bast n_{1})|$
are infinitesimal, so that
\[   |\bast n_{1} +\bast n_{2} |^{2}_{\uptheta} \leq   
|\bast n_{1} |^{2}_{\uptheta} + |\bast n_{2} |^{2}_{\uptheta}.\]
However by Proposition \ref{trivialnorm}, for $\upmu\geq\upnu$,  $|\bast n_{1} |^{2}_{\uptheta}=0$, implying
$ |\bast n_{1} +\bast n_{2} |_{\uptheta} \leq \max (|\bast n_{1}  |_{\uptheta},|\bast n_{2} |_{\uptheta}   )$.
  \end{proof}

Since at least one of the elements of any defined approximate group sum already has
$\uptheta$-norm $0$, Theorem \ref{nonarchtriangle} is somewhat unsatisfying.  A more interesting norm-theoretic interpretation of $|\cdot |_{\uptheta}$  may be obtained by restricting to $\bast\Z_{\upnu}^{\rm sym} (\uptheta )$.
As it turns out, it is more natural to contemplate a {\it Minkowskian} formulation of $|\cdot|_{\uptheta}$.  

Given
$\bast m,\bast n\in \bast\Z_{\upnu}^{\rm sym} (\uptheta )$, consider the following symmetric function in two-variables:
\[  \boldsymbol[ \bast m,  \bast n\boldsymbol]_{\uptheta} := \frac{1}{2}\left(\bast m\upvarepsilon (\bast n) + \bast n\upvarepsilon (\bast m)\right)\mod\bast\R_{\upvarepsilon}\in\R\]
and write $\boldsymbol[ \bast m\boldsymbol]_{\uptheta} := \boldsymbol[ \bast m,  \bast m\boldsymbol]_{\uptheta}^{ \frac{1}{2}}$ so that $|\bast m|_{\uptheta}=
|\boldsymbol[ \bast m\boldsymbol]_{\uptheta} |$.  In particular, $|\cdot|_{\uptheta}$
is the ``Minkowski norm'' associated to $ \boldsymbol[ \cdot,  \cdot\boldsymbol]_{\uptheta}$.   We say that $\bast m$ is {\bf  time-like} if $\boldsymbol[ \bast m\boldsymbol]^{2}_{\uptheta}>0$ and {\bf  space-like}
if $\boldsymbol[ \bast m\boldsymbol]_{\uptheta}^{2}<0$.  
The time-like elements correspond to the signs $\upsigma= (+,+), (-,-)$ and the space-like elements correspond to the signs $\upsigma= (+,-), (-,+)$.
We say that time-like elements point in the same direction if their sign $\upsigma$ is the same: the elements with sign $(+,+)$ are
interpreted as future pointing.  

The function 
$  \boldsymbol[ \cdot,  \cdot\boldsymbol]_{\uptheta}$ extends by the same formula to all of
$\bast\Z (\uptheta|x^{-1} )=\{ \bast\Z^{\upmu}_{\upnu} (\uptheta|x^{-1} )\}$.  We call an element $\bast m\in\bast\Z (\uptheta|x^{-1} )$
{\bf  light-like} if $\boldsymbol[ \bast m\boldsymbol]_{\uptheta} =0$ e.g.\ if $\bast m\in  \bast\Z^{\upmu}_{\upnu} (\uptheta|x^{-1} )$
for some $\upmu\geq\upnu$. The time-like and space-like elements of $\bast\Z (\uptheta|x^{-1} )$ are exactly the elements
of $\bast\Z^{\rm sym} (\uptheta )$.  If we reverse our clocks and view diophantine approximations as material particles ``departing from'' 
$\uptheta$ ($\bast n^{-1}=$ time and $\upvarepsilon (\bast n)=$
space) then $|\bast n|_{\uptheta}$ is nothing more than the initial speed.
When $\uptheta\in\mathfrak{B}$, we have the inequality 
\begin{align}\label{HUPforBA}
|\boldsymbol[ \bast m\boldsymbol]|^{2}_{\uptheta}\geq C_{\uptheta}
\end{align}
where $C_{\uptheta}$ is the corresponding element of the Lagrange spectrum.  We may interpret $C_{\uptheta}$ 
as being Planck's constant for the system defined by the symmetric diophantine approximations of $\uptheta$, where the inequality (\ref{HUPforBA})
plays the role of Heisenberg's uncertainty principle.


\begin{theo}  For all $\bast m_{1},\bast m_{2}\in\bast\Z_{\upnu}^{\rm sym} (\uptheta )$,
\[    \boldsymbol[ \bast m_{1}+\bast m_{2},  \bast n\boldsymbol]_{\uptheta} =\boldsymbol[ \bast m_{1},  \bast n\boldsymbol]_{\uptheta} 
+\boldsymbol[\bast m_{2},  \bast n\boldsymbol]_{\uptheta} .\]
If moreover $\bast m_{1},\bast m_{2}$ are time-like and point in the same direction, then they satisfy the {\rm reverse} triangle inequality:
\[  \boldsymbol[ \bast m_{1}+\bast m_{2}\boldsymbol]_{\uptheta}\geq \boldsymbol[ \bast m_{1}\boldsymbol]_{\uptheta} +\boldsymbol[ \bast m_{2}\boldsymbol]_{\uptheta}.   \]
\end{theo}

\begin{proof}  The first statement is immediate.  To prove the second it will suffice to prove the following reverse Cauchy inequality:
\[ \boldsymbol[ \bast m,  \bast n\boldsymbol]_{\uptheta}\geq \boldsymbol[ \bast m_{1}\boldsymbol]_{\uptheta}\boldsymbol[ \bast m_{2}\boldsymbol]_{\uptheta}.\]
Indeed, suppose the latter is true, and  $\bast m_{1},\bast m_{2}$ are time-like of the same sign, then immediately:
\begin{align*}
  \boldsymbol[ \bast m_{1}+\bast m_{2}\boldsymbol]_{\uptheta}^{2} =(\bast m_{1}+\bast m_{2})(\upvarepsilon(\bast m_{1})+\upvarepsilon(\bast m_{2})) & =
  \boldsymbol[ \bast m_{1}\boldsymbol]_{\uptheta}^{2} + \boldsymbol[ \bast m_{2}\boldsymbol]_{\uptheta}^{2} +2 \boldsymbol[ \bast m_{1},  \bast m_{2}\boldsymbol]_{\uptheta} \\
  & \geq ( \boldsymbol[ \bast m_{1}\boldsymbol]_{\uptheta} +\boldsymbol[ \bast m_{2}\boldsymbol]_{\uptheta})^{2}.
\end{align*}
But the reverse Cauchy inequality follows from:
\begin{align*} 
\boldsymbol[ \bast m_{1},  \bast m_{2}\boldsymbol]_{\uptheta}^{2} -\boldsymbol[\bast m_{1}\boldsymbol]_{\uptheta}^{2}\boldsymbol[ \bast m_{2}\boldsymbol]_{\uptheta}^{2} &= 
\frac{1}{4}(   \bast m_{1}\upvarepsilon (\bast m_{2})-   \bast m_{2}\upvarepsilon (\bast m_{1}) )^{2}\geq 0.
\end{align*}
\end{proof}

We call $\bast\Z^{\rm sym} (\uptheta )=\{ \bast\Z_{\upnu}^{\rm sym} (\uptheta )\}$ equipped with the Minkowskian pairing
$\boldsymbol[ \cdot, \cdot\boldsymbol]_{\uptheta}$ a {\bf  Lorentzian approximate group}.  
The Minkowskian norm has the following compatibility with the symmetric approximate ideal product:

\begin{prop}\label{innerprodprod}   Suppose that $\bast m\in  \bast\Z_{\upnu}^{\rm sym} (\uptheta ),
\bast n\in  \bast\Z_{\upnu}^{\rm sym} (\upeta )$ are either both time-like or both space-like.  
Then $ |\boldsymbol[\bast m\bast n\boldsymbol]_{\uptheta\upeta}|=\infty$ so that
\[   |\boldsymbol[\bast m\bast n\boldsymbol]_{\uptheta\upeta}| >    \boldsymbol[\bast m\boldsymbol]_{\uptheta}  \boldsymbol[\bast n\boldsymbol]_{\upeta}. \]
\end{prop}

\begin{proof}  First note that $\upvarepsilon (\bast m\bast n)=\bast n\upvarepsilon (\bast m) + \bast m\upvarepsilon (\bast n) 
+\upvarepsilon (\bast m)\upvarepsilon (\bast n)$ and by hypothesis $\bast n\upvarepsilon (\bast m)\simeq r$, $\bast n\upvarepsilon (\bast m)\simeq s$ with $r,s$ non-zero reals.  These non-zero reals will be of the same sign if $\bast m,\bast n$ are either both time-like
or both space-like: in this event, $|\upvarepsilon (\bast m\bast n)|\simeq |r|+|s|>0$, hence
$|\boldsymbol[\bast m\bast n\boldsymbol]_{\uptheta\upeta}|$ is infinite and the result is trivially true.
\end{proof}

\begin{theo}\label{LorIsom}  The group ${\rm PGL}_{2}(\Z )$ acts by Lorentzian isometries: that is, for all $\bast m,\bast n\in\bast\Z_{\upnu}^{\rm sym} (\uptheta )$,
\[  \boldsymbol[ \bast m, \bast n\boldsymbol]_{\uptheta} = \boldsymbol[ A(\bast m),A(\bast n)\boldsymbol]_{A(\uptheta )}. \] 
In particular, if $\uptheta\Bumpeq \upeta$ then $\bast\Z^{\rm sym}(\uptheta)\cong\bast\Z^{\rm sym}(\upeta)$
as Lorentzian approximate groups.
\end{theo}

\begin{proof}  If $A=\left(\begin{array}{cc}
a & b \\
c& d
\end{array}      \right)$ then as was seen in Theorem \ref{triisomorphism}, $A(\bast m)=c\bast m^{\perp}+d\bast m$ and $\upvarepsilon (A(\bast m))=(c\uptheta+d)^{-1}\upvarepsilon (\bast m)$, with similar formulas for $A(\bast n)$ and $\upvarepsilon (A(\bast n))$.  Thus
\[  A (\bast m)\upvarepsilon (A(\bast m))  = \frac{c\bast m^{\perp}+d\bast m}{c\uptheta+d}\upvarepsilon (\bast m) = 
\frac{c(\bast m^{\perp}/\bast m)+d}{c\uptheta+d}  \bast m\upvarepsilon (\bast m) \simeq \bast m\upvarepsilon (\bast m)\]
and the result follows.
\end{proof}

We now define a family of norms indexed by general exponents.
For each $\upkappa >1$ consider the function $x^{-\upkappa}$ and the set $\bast\Z (\uptheta |x^{-\upkappa})\subset \bast\Z (\uptheta |x^{-1})$ with
its associated ``norm'' function
\[ |\bast m|_{\uptheta,\upkappa} = |\bast m^{\upkappa}\upvarepsilon(\bast n)|\mod \bast\R_{\upvarepsilon} \in\R .\]
Much of the discussion developed above for the case $\upkappa =1$ extends to $\upkappa >1$.   We summarize the situation for
$\upkappa >1$ leaving the straightforward verifications to the reader.

\begin{itemize}
\item[a.] The set of {\small $\boldsymbol\upkappa$ {\bf -symmetric diophantine approximations}}, defined 
\[ \bast\Z_{\upkappa}^{\rm sym} (\uptheta )=\{ \bast\Z^{\rm sym}_{\upnu,\upkappa} (\uptheta )\}:=|\cdot |_{\uptheta,\upkappa}^{-1}(0,\infty)
\subset\bast\Z (\uptheta |x^{-1}) ,\]  
 satisfies the obvious analogue of Theorem \ref{symmgroupologystruct}.   
 \item[b.]  If $\upkappa\not=\upkappa'$ then $\bast\Z_{\upkappa}^{\rm sym} (\uptheta )\cap\bast\Z_{\upkappa'}^{\rm sym} (\uptheta )=\{ 0\}$.
\item[c.] If $\uptheta$ is
$\upkappa$-bad then $\bast\Z_{\upkappa}^{\rm sym} (\uptheta ) = \bast\Z (\uptheta |x^{-\upkappa})$ is a family of groups, and $\bast\Z_{\upkappa'}^{\rm sym} (\uptheta )=0$
for all $\upkappa'>\upkappa$.
\item[d.]  If $\uptheta\Bumpeq \upeta$ by $A\in{\rm PGL}_{2}(\Z)$ then we have the following analogue of Theorem \ref{LorIsom}: for all $\bast m\in \bast\Z (\uptheta |x^{-\upkappa})$,  $|A(\bast m)|_{A(\uptheta),\upkappa}=|\bast m|_{\uptheta,\upkappa}$.  
\end{itemize}

We may also consider nonstandard exponents: for any $\bast\upkappa\in\bast\R_{+}$,  $\bast\upkappa>1$, we may define in the obvious way
\[ \bast\Z (\uptheta |x^{-\bast \upkappa})\;\subset\;
\bigcap_{\bast \upkappa>\upkappa>1} \bast\Z (\uptheta |x^{-\upkappa}),\] equipped with its associated norm function $|\cdot |_{\uptheta,\bast\upkappa}$ with which we may define the $\bast\upkappa$ symmetric diophantine
approximations $\bast\Z_{\bast \upkappa}^{\rm sym} (\uptheta )=\{ \bast\Z^{\rm sym}_{\upnu,\bast\upkappa} (\uptheta )\}$.
Note that if $\uptheta\not\in\mathfrak{W}_{\infty}$ then  $\bast\Z_{\bast \upkappa}^{\rm sym} (\uptheta )=0$ for all $\bast\upkappa$ infinite nonstandard.
In general we have
\[ \bast\Z (\uptheta |x^{-1}) = \bigcup_{\bast\upkappa>1} \bast\Z_{\bast \upkappa}^{\rm sym} (\uptheta ).  \]

We refer to $\bast\Z (\uptheta )$, with the family 
$\{(\bast\Z (\uptheta |x^{-\bast \upkappa}), |\cdot |_{\uptheta,\bast\upkappa})\}$
as a {\bf  Frechet Lorentzian approximate group}: for any $\bast m\in\bast\Z (\uptheta )$, $|\bast m |_{\uptheta,\bast\upkappa} =0$ for all $\bast\upkappa$ $\Leftrightarrow $
$\bast m=0$.  We call a approximate module homomorphism
\[f:\bast \Z(\uptheta )\rightarrow \bast \Z(\upeta )\]
 a {\bf  Frechet Lorentzian isometry} or simply an {\bf  isometry} if it preserves the Frechet Lorentzian norms.

\begin{theo} $\uptheta\Bumpeq \upeta$ $\Rightarrow $ $\bast\Z (\uptheta )\cong\bast\Z (\upeta )$ by isometric approximate module isomorphism.
\end{theo}

\begin{proof}  The proof of Theorem \ref{LorIsom} follows through identically to show that $A\in{\rm PGL}_{2}\Z$ acts
by Frechet Lorentzian isometries.   
\end{proof}

\begin{conj} $\uptheta\Bumpeq \upeta$ $\Leftrightarrow $ $\bast\Z (\uptheta )\cong\bast\Z (\upeta )$ by an isometric approximate module isomorphism.
\end{conj}


\section{Matrix Approximate Ideal Arithmetic}\label{matideoarith}

Let $\Uptheta$ be a real matrix of size $r\times s$.  In \cite{Ge4} we defined the 
{\bf inhomogeneous diophantine approximation group} of $\Uptheta$ to be
\[ \bast\Z^{s}(\Uptheta) = \{ \bast {\boldsymbol n}\in \bast\Z^{s} |\; \exists  \bast {\boldsymbol n}^{\perp}\in\bast\Z^{r}\text{ s.t. }  
\upvarepsilon (\bast {\boldsymbol n}):=\Uptheta\bast {\boldsymbol n}-\bast {\boldsymbol n}^{\perp}\ \in\bast\R^{r}_{\upvarepsilon} \}.\] 
The corresponding {\bf homogeneous diophantine approximation group} was defined by
\[ \bast\widetilde{\Z}^{s}(\Uptheta) = {\rm Ker}(\perp ) = \{  \bast {\boldsymbol n}\in \bast\Z^{s}(\Uptheta) |\;  \Uptheta \bast{\boldsymbol n}
\in\bast\R^{r}_{\upvarepsilon} \} < \bast\Z^{s}(\Uptheta).\]  Thus, if $\bast\Z^{r}(\Uptheta)^{\perp}$ denotes the group of duals then
$\bast\Z^{s}(\Uptheta)/\bast\widetilde{\Z}^{s}(\Uptheta)\cong \bast\Z^{r}(\Uptheta)^{\perp}$.  Note that if $\Uptheta$ is square
invertible then $\bast\widetilde{\Z}^{s}(\Uptheta) =0$.

In this section we will develop approximate ideal arithmetic for the groups $\bast\Z^{s}(\Uptheta)$.  
 First, for $\bast \boldsymbol n=(\bast n_{1},\dots ,\bast n_{s})\in\bast\Z^{s}$, the {\bf house norm} is defined 
\begin{align}\label{hatnorm}  \overline{|\bast\boldsymbol n|} := \max_{j=1,\dots ,s} |\bast n_{j}| 
\end{align}
and for $\bast\boldsymbol n\not=\boldsymbol 0$ write $\upmu (\bast\boldsymbol n ) :=\langle \overline{|\bast\boldsymbol n|}\rangle^{-1}\in\bstar\PR\R$.
Then for $\upmu\in\bstar\PR\R_{\upvarepsilon}$ the set 
\[ (\bast\Z^{s})^{\upmu}=\big\{\bast\boldsymbol n\not=\boldsymbol 0\big|\; \upmu^{1/s}<\upmu (\bast\boldsymbol n )\big\}\cup \{\boldsymbol 0\}=
\big\{\bast\boldsymbol n\big|\;  \overline{|\bast\boldsymbol n|}\cdot\upmu^{1/s}\in\bstar\PR\R_{\upvarepsilon}\big\}
\]
forms a group: for $\bast\boldsymbol m,\bast\boldsymbol n\in(\bast\Z^{s})^{\upmu}$,
\[  \overline{|\bast\boldsymbol m+\bast\boldsymbol n|}\cdot\upmu^{1/s}\leq
\big( \overline{|\bast\boldsymbol m|}+ \overline{|\bast\boldsymbol n|}\big)\cdot \upmu^{1/s} \in \bstar\PR\R_{\upvarepsilon}.\]
The set $(\bast\Z^{s})^{\upmu[\upiota]}$ of elements of $(\bast\Z^{s})^{\upmu}$ which in addition satisfy $ \overline{|\bast\boldsymbol n|}\cdot\upmu^{1/s}<\upiota$ forms a
subgroup.

For $\bast\boldsymbol \upvarepsilon =(\bast\upvarepsilon_{1},\dots ,\bast\upvarepsilon_{r})\in\bast\R^{r}_{\upvarepsilon}$, define
 $\overline{|\bast\boldsymbol  \upvarepsilon|}$ as in (\ref{hatnorm}). Given $\bast\boldsymbol n\in\bast\Z^{s}(\Uptheta)$ denote
by $\upnu (\bast\boldsymbol n)= \langle\overline{|\upvarepsilon(\bast\boldsymbol n )  |}\rangle$: then
\[ (\bast\Z^{s})_{\upnu}(\Uptheta)=\{\bast\boldsymbol n\in \bast\Z^{s}(\Uptheta)|\; \upnu (\bast\boldsymbol n)\leq\upnu^{1/r}\}\] is also a group.  
We have thus defined the {\bf matrix approximate ideal}
\[ \bast\Z^{s}(\Uptheta) = \{(\bast\Z^{s})_{\upnu}^{\upmu}(\Uptheta)\}=\{(\bast\Z^{s})_{\upnu}^{\upmu[\upiota]}(\Uptheta)\}.\]

In general, the dual group $(\bast\Z^{r})(\Uptheta)^{\perp}$ admits a filtration by growth only, $\{(\bast\Z^{r})^{\upmu}(\Uptheta)^{\perp}\}$,
except when $\bast\widetilde{\Z}^{s}(\Uptheta)=0$ e.g.\ when $\Uptheta$ is square invertible.  Nevertheless, we have

\begin{lemm}\label{matrixdualgrowth}  Let $\bast \boldsymbol n\in (\bast\Z^{s})^{\upmu[\uplambda]}(\Uptheta)$.  Then 
\[ \overline{|\bast\boldsymbol n^{\perp}|}\cdot\upmu^{1/s}<\uplambda,\]
that is, $\bast \boldsymbol n^{\perp}\in (\bast\Z^{r})^{\upmu^{r/s}[\uplambda]}$. 
If $\Uptheta$ is invertible of dimension $r\times r$ then $\bast \boldsymbol n\in (\bast\Z^{r})^{\upmu[\uplambda]}(\Uptheta)\Leftrightarrow \bast \boldsymbol n^{\perp}\in (\bast\Z^{r})^{\upmu[\uplambda]}(\Uptheta)^{\perp}$.
\end{lemm}

\begin{proof}  
We must show that $\langle \overline{|\bast\boldsymbol n^{\perp}|}\rangle\cdot(\upmu^{r/s})^{1/r} =
\langle \overline{|\bast\boldsymbol n^{\perp}|}\rangle\cdot\upmu^{1/s}
\in\bstar\PR\R_{\upvarepsilon}$.  But this follows immediately from:
\[   \overline{|\bast\boldsymbol n^{\perp}|} = \overline{|\Uptheta\bast\boldsymbol n |} =\max_{i=1,\dots, r}|\sum_{j=1}^{s} \Uptheta_{ij}\bast n_{j}|
\leq s\overline{|\Uptheta|}\cdot\overline{|\bast\boldsymbol n |} \]
where $\overline{|\Uptheta|} = \max_{ij}|\Uptheta_{ij}|$.  
The last statement follows from symmetry of argument.
\end{proof}

The relevant arithmetic operations for matrix approximate ideal arithmetic
derive from the Kronecker product and sum \cite{HJ}.
 Let $\Uptheta$, $\Uptheta'$ be real matrices of dimensions $r\times s$, $r'\times s'$.     Denote by $\Uptheta\otimes\Uptheta'$ 
 the {\bf Kronecker (or tensor) product} i.e.\ the $rr'\times ss'$ block matrix
 \[  \Uptheta\otimes\Uptheta' = \left(\begin{array}{ccc}
 \uptheta_{11}\Uptheta' & \cdots & \uptheta_{1s}\Uptheta' \\
 \vdots & \ddots &\vdots \\
  \uptheta_{r1}\Uptheta' & \cdots & \uptheta_{rs}\Uptheta' 
  \end{array}
 \right)  \]
 where $\Uptheta = (\uptheta_{ij})$.
If $r=s$ and $r'=s'$ the {\bf Kronecker sum} and {\bf Kronecker difference}
are the $rr'\times rr'$ matrices 
\[   \Uptheta\oplus\Uptheta' = \Uptheta\otimes I_{r'} + I_{r}\otimes \Uptheta',\quad  \Uptheta\ominus\Uptheta' = \Uptheta\otimes I_{r'} - I_{r}\otimes \Uptheta'. \] 
Neither the Kronecker product nor the Kronecker sum/difference are commutative, nor do they satisfy the distributive law.   If we assume $\Uptheta,\Uptheta'$
are square and denote by $\upsigma (\Uptheta )$, $\upsigma (\Uptheta' )$ their spectra 
then $\upsigma (\Uptheta \otimes\Uptheta')= \upsigma (\Uptheta )\cdot \upsigma (\Uptheta' )$,
$\upsigma (\Uptheta \oplus\Uptheta')= \upsigma (\Uptheta )+ \upsigma (\Uptheta' )$ and $\upsigma (\Uptheta \ominus\Uptheta')= \upsigma (\Uptheta )- \upsigma (\Uptheta' )$.

Denote by $\tilde{\mathcal{M}}(\R)$
the monoid of all real matrices with respect to the Kronecker product, and by $\mathcal{M}(\R)$ the submonoid of square matrices, equipped
further with the Kronecker sum/difference.  Note that there is a monomorphism $(\R,+,\times)\hookrightarrow (\mathcal{M}(\R),\oplus,\otimes)$.
Observe that if $\boldsymbol m$, $\boldsymbol n$ are vectors
of size $s$ resp. $s'$ then 
\begin{align}\label{matrixdenominatorlaw}  (\Uptheta\otimes\Uptheta')(\boldsymbol m\otimes\boldsymbol n ) = 
 \Uptheta(\boldsymbol m)\otimes\Uptheta'(\boldsymbol n);
 \end{align}
if $r=s$ and $r'=s'$ and $\boldsymbol m$, $\boldsymbol n$ are vectors of size $r$ resp. $r'$ then
\begin{align}\label{matrixnumeratorlaw}  (\Uptheta\oplus\Uptheta')(\boldsymbol m\otimes\boldsymbol n ) = 
 \Uptheta(\boldsymbol m)\otimes \boldsymbol n +  \boldsymbol m\otimes\Uptheta'(\boldsymbol n),
 \end{align}
 with a similar formula for the Kronecker difference.
 If we think of $(\Uptheta\boldsymbol m,\boldsymbol m)$ as a ``vector numerator denominator pair'', then 
 (\ref{matrixdenominatorlaw}) is the formula for the numerator of the product and
 (\ref{matrixnumeratorlaw}) is the formula for the numerator of the sum.

\begin{theo}[Matrix Approximate Ideal Arithmetic]\label{matrixideoarith}   Let $\Uptheta$, $\Uptheta'$ be real matrices of dimensions $r\times s$, $r'\times s'$. 
Then there is a well-defined bilinear pairing:
\begin{align}\label{KParith}
\otimes: (\bast \Z^{s})^{\upmu^{s}[\upiota]}_{\upnu^{r}} (\Uptheta)\times(\bast \Z^{s'})^{\upnu^{s'}[\uplambda]}_{\upmu^{r'}} (\Uptheta')
\longrightarrow  (\bast \Z^{ss'})^{\upmu^{ss'}\cdot\upnu^{ss'}[\upiota\cdot\uplambda]}_{\upiota+\uplambda} (\Uptheta\otimes\Uptheta') 
 \end{align}
 defined by the Kronecker product of vectors.  If $r=s$ and $r'=s'$ then the Kronecker product also defines a pairing
  \begin{align} \label{KSarith}
 \otimes: (\bast \Z^{r})^{\upmu^{r}[\upiota]}_{\upnu^{r}} (\Uptheta)\times (\bast \Z^{r'})^{\upnu^{r'}[\uplambda]}_{\upmu^{r'}} (\Uptheta')
\longrightarrow  (\bast \Z^{rr'})^{\upmu^{rr'}\cdot\upnu^{rr'}[\upiota\cdot\uplambda]}_{\upiota+\uplambda} (\Uptheta\oplus\Uptheta') 
\cap (\bast \Z^{rr'})^{\upmu^{rr'}\cdot\upnu^{rr'}[\upiota\cdot\uplambda]}_{\upiota+\uplambda} (\Uptheta\ominus\Uptheta').
\end{align}
In particular, when $r=r'=s=s'=1$ we recover Theorem \ref{productformula}.
\end{theo}

\begin{proof}  We will prove (\ref{KParith}); the proof of (\ref{KSarith}) is similar and is left to the reader.  The argument amounts
to replacing scalar product by tensor product in the
calculations found in Theorem \ref{productformula}, taking into account the dimensional normalizations used to define the filtrations.  Let
$\bast \boldsymbol m\in (\bast \Z^{s})^{\upmu^{s}[\upiota]}_{\upnu^{r}}(\Uptheta)$, $\bast \boldsymbol n\in(\bast \Z^{s'})^{\upnu^{s'}[\uplambda]}_{\upmu^{r'}} (\Uptheta')$. 
First observe that 
\[ \upmu (\bast \boldsymbol m\otimes  \bast \boldsymbol n) = \upmu( \bast \boldsymbol m)\cdot \upmu( \bast \boldsymbol n)>\upmu\cdot\upnu = 
(\upmu^{ss'}\cdot\upnu^{ss'})^{1/ss'}
\]
which implies $\bast \boldsymbol m\otimes  \bast \boldsymbol n\in (\bast \Z^{ss'})^{\upmu^{ss'}\cdot\upnu^{ss'}[\upiota\cdot\uplambda]}$.
Denote
by $\Uptheta_{i}$ the $i$th row of $\Uptheta$, $\Uptheta'_{j}$ the $j$th row of $\Uptheta'$ and by $\odot$ the dot product.
The $rr'$ vector $\Uptheta\otimes\Uptheta' (\bast \boldsymbol m\otimes\bast \boldsymbol n)$ has coordinates indexed by $(i,j)\in I\times J$
where the latter is given the linear dictionary order.  The $(i,j)$ coordinate is given by 
\begin{align*} (\Uptheta_{i}\odot \bast \boldsymbol m)\cdot (\Uptheta_{j}\odot \bast \boldsymbol n) & = 
(\bast\boldsymbol m^{\perp}_{i}+ \upvarepsilon (\bast \boldsymbol m)_{i})\cdot (\bast\boldsymbol n^{\perp}_{j}+ \upvarepsilon (\bast \boldsymbol n)_{j}) \\
& =
\bast\boldsymbol m^{\perp}_{i}  \bast\boldsymbol n^{\perp}_{j}+\upvarepsilon (\bast \boldsymbol m)_{i}\bast\boldsymbol n^{\perp}_{j} +
\bast\boldsymbol m^{\perp}_{i}\upvarepsilon (\bast \boldsymbol n)_{j}+\upvarepsilon (\bast \boldsymbol m)_{i}\upvarepsilon (\bast \boldsymbol n)_{j}
\end{align*}
i.e.
\begin{align}\label{matgderror}  \Uptheta\otimes\Uptheta' (\bast \boldsymbol m\otimes \bast\boldsymbol n) -  \bast\boldsymbol m^{\perp} \otimes \bast\boldsymbol n^{\perp} & =
\upvarepsilon (\bast \boldsymbol m)\otimes\bast \boldsymbol n^{\perp}+\bast\boldsymbol m^{\perp}\otimes\upvarepsilon (\bast \boldsymbol n)+\upvarepsilon (\bast \boldsymbol m)\otimes\upvarepsilon (\bast \boldsymbol n).
   \end{align}
The term $\upvarepsilon (\bast \boldsymbol m)\otimes\upvarepsilon (\bast \boldsymbol n)$ can be disregarded because upon passage
to $\bstar\PR\R_{\upvarepsilon}$ it is strictly less than the absolute values of the images of the other two terms on the right hand side of (\ref{matgderror}). 
Now
\[ \langle \overline{|\upvarepsilon (\bast \boldsymbol m)\otimes \bast\boldsymbol n^{\perp}|} \rangle =
\overline{|\bast \boldsymbol n^{\perp}|}\cdot\upnu (\bast \boldsymbol m)\leq \overline{| \bast\boldsymbol n^{\perp}|}\cdot\upnu <\uplambda,
 \]
where we have used Lemma \ref{matrixdualgrowth}.  Similarly one shows that
$ \langle \overline{| \bast\boldsymbol m^{\perp}\otimes\upvarepsilon (\bast \boldsymbol n)|} \rangle<\upiota$
so that $\bast \boldsymbol m\otimes \bast\boldsymbol n\in
(\bast \Z^{ss'})_{\upiota+\uplambda} (\Uptheta\otimes\Uptheta') $ as claimed.
   \end{proof}

 The characterization of classes of real numbers in terms of the nonvanishing spectrum found in \S \ref{nonvanspec} can be generalized to real matrices using the nonvanishing spectrum
 
 \[ {\rm Spec}(\Uptheta )= \{(\upmu,\upnu)|\;(\bast\Z^{s})_{\upnu}^{\upmu}(\Uptheta)\not=0 \}.\]  By viewing $\Uptheta$ as the family of linear forms $\{\Uptheta_{i}\}_{i=1}^{r}$ in $s$ variables, we will obtain the familiar
 correspondence between the geometry of  ${\rm Spec}(\Uptheta )$ and approximation classes.  
 
 We say that $\Uptheta$ is {\bf rational} if there exists a vector $\boldsymbol m\in\Z^{s}$ such that 
 $\Uptheta(\boldsymbol m)\in\Z^{r}$; otherwise we say that $\Uptheta$ is {\bf irrational}.   Denote the set of rational real matrices
 by $\tilde{\mathcal{Q}}(\R)$ and by $\mathcal{Q}(\R)$ the subset of square rational real matrices.  Then both $\tilde{\mathcal{Q}}(\R)$ and by 
 $\mathcal{Q}(\R)$
 are closed with respect to $\otimes$, and $\mathcal{Q}(\R)$ is closed with respect to $\oplus,\ominus$ as well.  Of course, the restriction of
 the monomorphism $\R\hookrightarrow \mathcal{M}(\R)$ to $\Q$ lies in $\mathcal{Q}(\R)$.
 
 \begin{prop}  For all $\Uptheta\in\tilde{\mathcal{M}}(\R )$, ${\rm Spec}(\Uptheta )\supset \{(\upmu,\upnu)|\upmu<\upnu\}$.
 If  $\Uptheta\in\tilde{\mathcal{Q}}(\R )$ then ${\rm Spec}(\Uptheta )=\bstar\PR\R_{\upvarepsilon}^{2}$.
 \end{prop}
 
 \begin{proof}  For $\Uptheta\in\tilde{\mathcal{M}}-\tilde{\mathcal{Q}}(\R)$, this is an straightforward adaptation of the proof of Theorem \ref{gennonvan} using Schmidt's ``Dirichlet Theorem'' for families of linear forms (Theorem 3A on page 36 of \cite{Sch}).
 If $\Uptheta\in\tilde{\mathcal{Q}}(\R )$ then $(\bast\Z^{s})_{-\infty}(\Uptheta)\not=0$ from which we get ${\rm Spec}(\Uptheta )=\bstar\PR\R_{\upvarepsilon}^{2}$.
 \end{proof}
 
A matrix 
 $\Uptheta$ is {\bf badly approximable} if there exists a constant $C>0$ such that for all $\bast \boldsymbol n\in
  \bast\Z^{s}(\Uptheta)$,
 \[ \overline{|\bast\boldsymbol n|}^{s} \cdot \overline{|\upvarepsilon (\bast\boldsymbol n)|}^{r}>C.  \]
 This definition is equivalent to the definition given in \cite{Sch} in terms of the family of linear forms $\{\Uptheta_{i}\}_{i=1}^{r}$.
In what follows, ${\rm Spec}_{\rm flat}(\Uptheta)$ = the intersection of ${\rm Spec}(\Uptheta)$ with the diagonal.
 
 \begin{theo}  $\Uptheta\in\tilde{\mathcal{M}}(\R)-\tilde{\mathcal{Q}}(\R)$ is badly approximable $\Leftrightarrow$ ${\rm Spec}(\Uptheta)=\{(\upmu,\upnu)|\upmu<\upnu\}$
 $\Leftrightarrow$
 ${\rm Spec}_{\rm flat}(\Uptheta)=\emptyset$.
 \end{theo}
 
 \begin{proof}  This follows directly from the definition of being badly approximable.
 \end{proof}
 
 Denote by $\tilde{\mathcal{B}}(\R)\supset\mathfrak{B}$ the set of badly approximable real matrices, by 
 $\tilde{\mathcal{W}}(\R)  =\tilde{\mathcal{M}}(\R)-(\tilde{\mathcal{B}}(\R)\cup \tilde{\mathcal{Q}}(\R)) \supset\mathfrak{W}$
 the set of well approximable matrices.  In addition, we have the notion of {\bf very well approximable matrices} $\tilde{\mathcal{W}}_{\upkappa}(\R)_{r,s}$ of exponent $\upkappa>s$: those for which
 \begin{enumerate}
\item there exists $\bast\boldsymbol n$ such that for any $\uplambda\in (s,\upkappa)$
  \[ \overline{|\bast\boldsymbol n|}^{\uplambda} \cdot \overline{|\upvarepsilon (\bast\boldsymbol n)|}^{r}\simeq 0.  \] 
  \item The infinitesimal equation 
  \[ \overline{|\bast\boldsymbol n|}^{\upkappa'} \cdot \overline{|\upvarepsilon (\bast\boldsymbol n)|}^{r}\simeq 0  \]
  has no solution.
  \end{enumerate}
  Write $\tilde{\mathcal{W}}_{>s}(\R)_{r,s}$ for the set of very well approximable matrices of size $r\times s$.
  When $\upkappa=\infty$ we get the {\bf Liouville matrices} $\tilde{\mathcal{W}}_{\infty}(\R)_{r,s}$.  In addition, let $\tilde{\mathcal{B}}(\R)_{r,s}$
 denote the space of badly approximable matrices of size $r\times s$.  
 
 \begin{theo}  Let $\tilde{\mathcal{C}}(\R )_{r,s}$ be any of the classes $\tilde{\mathcal{B}}(\R)_{r,s}$, 
 $\tilde{\mathcal{W}}_{>s}(\R)_{r,s}$, $\tilde{\mathcal{W}}_{\infty}(\R)_{r,s}$ described in the previous paragraph.  Then 
 \[ \tilde{\mathcal{C}}(\R)_{r,s}^{T}:=\big\{\Uptheta^{T}\big|\; \Uptheta\in \tilde{\mathcal{C}}(\R )_{r,s}\big\}=\tilde{\mathcal{C}}(\R)_{s,r}.\]
 \end{theo}
 
 \begin{proof}  This follows from the Khintchine Transference Principle (Theorem 5C of \S 4.5 of \cite{Sch}).
 \end{proof}
 
 We now introduce a generalized notion of projective linear equivalence appropriate to the matrix setting.     Let
 ${\rm GL}_{r,s}(\Z ):={\rm GL}_{r+s}(\Z)$ be the group of $(r+s)\times (r+s)$ integral invertible matrices, partitioned in the following way
 \begin{align}\label{blockformofGL} 
 M= 
  \left(\begin{array}{cc}
   A & B \\
   C & D \\
 \end{array}
  \right)=
  \left(\begin{array}{cc}
   A_{r\times r} & B_{r\times s} \\
   C_{s\times r} &  D_{s\times s} \\
 \end{array}
  \right) 
  \end{align}
 (where the block subindices indicate the dimension). Note that the product of matrices $M,M'\in{\rm GL}_{r,s}(\Z)$ viewed in block form follows the familiar formula for
 $2\times 2$ matrix multiplication  e.g. 
  \begin{align}\label{blockprodform}  M'M =\left(\begin{array}{cc}
   A' & B' \\
   C' & D' \\
 \end{array}
  \right)\left(\begin{array}{cc}
   A & B \\
   C & D \\
 \end{array}
  \right)=\left(\begin{array}{cc}
   A'A+B'C & A'B +B'D\\
   C'A+D'C & C'B+D'D \\
 \end{array}
  \right) .
  \end{align}
  If $M'=M^{-1}$ then
  $ A'A+B'C=I_{r}$, $C'B+D'D=I_{s}$, $A'B +B'D=O_{r,s}$ and $C'A+D'C =O_{s,r}$, where
$ I_{r}$ is the $r\times r$ identity matrix, $O_{r,s}$ is the $r\times s$ zero matrix, etc.  
  
 Let $\Uptheta,\Uptheta'$ be $r\times s$ real matrices. 
 We write 
 \[ \Uptheta\Bumpeq_{r,s} \Uptheta'\] if there exists $M\in {\rm GL}_{r,s}(\Z )$ with 
 \[    (A\Uptheta  + B )= \Uptheta' (C \Uptheta + D), \]
or  equivalently, if 
  \begin{align}\label{matrixversionofequiv}  M \left(\begin{array}{c}
   \Uptheta \\
   I_{s} \\
 \end{array}\right)  = \left(\begin{array}{c}
   \Uptheta' \\
   I_{s} \\
 \end{array}\right) \big(C\Uptheta + D\big). 
 \end{align}
 Note that $\Bumpeq_{1,1}$ is just $\Bumpeq$, the usual relation of projective linear equivalence for scalars.

 \begin{theo} $\Bumpeq_{r,s}$ is an equivalence relation.
 \end{theo}
 
 \begin{proof} We begin with transitivity.  Suppose that $\Uptheta\Bumpeq_{r,s} \Uptheta'$ and $\Uptheta'\Bumpeq_{r,s} \Uptheta''$
 by matrices $M, M'$.  Then $\Uptheta' (C \Uptheta + D)= (A\Uptheta  + B )$ and 
 $\Uptheta'' (C' \Uptheta' + D')= (A'\Uptheta'  + B' )$.  We want to show that $\Uptheta\Bumpeq_{r,s} \Uptheta''$ via $M'M$.
 By the product formula (\ref{blockprodform}), this amounts to showing
 \[   \Uptheta'' \bigg(  \big(C'A+D'C\big)\Uptheta +  \big(C'B+D'D\big) \bigg) =   \big(A'A+B'C\big)\Uptheta +  \big(A'B+B'D\big).\]
 However this follows from:
 \begin{align*}
  \Uptheta'' \bigg(  \big(C'A+D'C\big)\Uptheta +  \big(C'B+D'D\big) \bigg)  &=  \Uptheta'' \bigg(  C'\big(A\Uptheta+B\big) +  D'\big(C\Uptheta+D\big) \bigg)  \\
  & = \Uptheta''\big(C'\Uptheta' +D'\big)\big(C\Uptheta +D\big) \\
  &= \big(A'\Uptheta' + B'\big)\big(C\Uptheta + D\big) \\
  & = A'\big(A\Uptheta +B\big) + B'\big(C\Uptheta +D\big) \\
  & = \big(A'A+B'C\big)\Uptheta +  \big(A'B+B'D\big).
 \end{align*}
 As for symmetry: suppose that $\Uptheta\Bumpeq_{r,s} \Uptheta'$ by $M$, understood in the sense of
  (\ref{matrixversionofequiv}).  Let $M'=M^{-1}$; we want to show that 
  \[ M' \left(\begin{array}{c}
   \Uptheta' \\
   I_{s} \\
 \end{array}\right)  =
 \left(\begin{array}{c}
   \Uptheta \\
   I_{s} \\
 \end{array}\right) \big(C'\Uptheta' + D'\big).  \]
 Multiplying both sides of (\ref{matrixversionofequiv}) by $M'$ gives 
 \[  \left(\begin{array}{c}
   \Uptheta \\
   I_{s} \\
 \end{array}\right) = M'  \left(\begin{array}{c}
   \Uptheta' \\
   I_{s} \\
 \end{array}\right) \big(C\Uptheta + D\big).\]
 Therefore we must show that $(C\Uptheta + D)(C'\Uptheta' + D')=I_{s}$: in fact, it suffices
 to show that $(C'\Uptheta' + D')(C\Uptheta + D)=I_{s}$.  This follows from:
 \begin{align*}
 (C'\Uptheta' + D')(C\Uptheta + D) = & C'(A\Uptheta+B)+D'(C\Uptheta +D) \\
 =& (C'A+D'C)\Uptheta + (C'B+D'D) \\
 = & O_{s,r}\Uptheta  + I_{s} = I_{s}.
 \end{align*}
 \end{proof}
 
  The matrix $M$ implying $\Uptheta\Bumpeq_{r,s}\Uptheta'$ is clearly a projective invariant;
 by abuse of notation
 we will sometimes write $M(\Uptheta )=\Uptheta'$.  We have therefore produced a partially defined action of ${\rm PGL}_{r,s}(\Z)$ on
 $\tilde{\mathcal{M}}(\R)_{r,s}$ generalizing the projective linear action of ${\rm PGL}_{2}(\Z)$ on $\R$: $M$ acts on $\Uptheta$ provided
 there exists $\Uptheta'$ for which $\Uptheta\Bumpeq_{r,s}\Uptheta'$ via $M$.

 Recall from \cite{Ge4} the group of vector numerator denominator pairs 
 \[  \bast\Z^{r,s}(\Uptheta)=\{ (\bast\boldsymbol m^{\perp}, \bast\boldsymbol m)| \;  \bast\boldsymbol m\in \bast\Z^{s}(\Uptheta) \}.\]
 
 \begin{theo}  Suppose that $M\in {\rm PGL}_{r,s}(\Z )$ acts on $\Uptheta\in\tilde{\mathcal{M}}(\R)_{r,s}$. Then the map
 \begin{align}\label{theinducedmap}  (\bast\boldsymbol m^{\perp}, \bast\boldsymbol m)\longmapsto  (\bast\boldsymbol n^{\perp}, \bast\boldsymbol n):=
 (A\bast\boldsymbol m^{\perp} + B\bast\boldsymbol m,
 C\bast\boldsymbol m^{\perp} + D\bast\boldsymbol m)
   \end{align}
   defines an isomorphism
   of groups of numerator denominator pairs
    $ \bast\Z^{r,s}(\Uptheta)\cong \bast\Z^{r,s}(M(\Uptheta))$.  In particular, 
   \[  (\bast\Z^{s})^{\upmu[\upiota]}_{\upnu}(\Uptheta)\cong
   (\bast\Z^{s})^{\upmu[\upiota]}_{\upnu}(M(\Uptheta)) \] for all $\upmu,\upnu,\upiota\in\bstar\PR\R_{\upvarepsilon}$.
 \end{theo}

\begin{proof}  Suppose that $\Uptheta (\bast \boldsymbol m ) =\bast\boldsymbol m^{\perp} + \upvarepsilon (\bast\boldsymbol m )$.  Then
we may write
\[ (\bast\boldsymbol n^{\perp}, \bast\boldsymbol n) = 
(A(\Uptheta (\bast \boldsymbol m )- \upvarepsilon (\bast\boldsymbol m )) + B\bast\boldsymbol m,
 C(\Uptheta (\bast \boldsymbol m )- \upvarepsilon (\bast\boldsymbol m )) + D \bast\boldsymbol m) .\]
 Therefore, 
 \begin{align*} 
 \Uptheta ' \bast\boldsymbol n  & \simeq \Uptheta' (C \Uptheta + D)\bast\boldsymbol m \\
  & =  (A\Uptheta  + B)\bast\boldsymbol m \\
  & \simeq \bast\boldsymbol n^{\perp}.
 \end{align*} 
 We leave the proof that (\ref{theinducedmap}) is invertible and respects growth-decay parameters to the reader.
\end{proof}

Define the relation 
\[ \Uptheta {}_{\upmu}\!\owedge_{\upnu}\Uptheta'\] when the groups occurring in (\ref{KParith}) (with their dimensional normalizations
 of $\upmu$, $\upnu$) are nontrivial.  The extent to which elements of $\tilde{\mathcal{M}}(\R)$ are involved in approximate ideal factorizations can be delineated according to class e.g.\
 the elements of $\tilde{\mathcal{B}}(\R)$ are anti prime, the elements of $\tilde{\mathcal{W}}_{>s}(\R)_{r,s}$ are the omnidivisors etc.  
 We leave it to the reader to formulate the matrix analogue of the divisibility discussion found at the end of \S \ref{gdarith}. 
 
 \begin{coro}  Let $\Uptheta, \Uptheta'\in\tilde{\mathcal{M}}(\R)-\tilde{\mathcal{Q}}(\R)$ be real matrices.   If $M\in {\rm PGL}_{r,s}(\Z ), N\in  {\rm PGL}_{r',s'}(\Z )$ act on $\Uptheta, \Uptheta'$ then $\Uptheta {}_{\upmu}\!\!\owedge_{\upnu}\Uptheta'$ $\Leftrightarrow$
$M(\Uptheta) {}_{\upmu}\!\!\owedge_{\upnu}N(\Uptheta')$.
 \end{coro}

 We end this section with a few notes regarding the possible further development of the matrix theory along the lines of the scalar theory:
 
 \begin{enumerate}
 \item A theory of flat matrix arithmetic would appear to be, for the moment, out of reach since we do not have yet at our disposable a
 matrix continued fraction algorithm
yielding a good notion of best approximations (in the sense of satisfying the matrix analogue of Dirichlet's Theorem), see for example \cite{Schw}. 
\item One may define on $\bast \Z^{s}(\Uptheta )$ the Minkowskian norm $\boldsymbol[ \bast \boldsymbol m\boldsymbol]_{\Uptheta}:= \overline{|\bast \boldsymbol m |}^{s}\cdot \overline{|\upvarepsilon(\bast \boldsymbol m ) |}^{r}$
as well as the set of symmetric diophantine approximations
\[ (\bast\Z^{s})^{\rm sym}(\Uptheta )=\{\bast\boldsymbol m \in\bast \Z^{s}(\Uptheta )|\;  0<\boldsymbol[ \bast \boldsymbol m\boldsymbol]_{\Uptheta}< \infty\}\cup \{ 0\}.\] 
For $\Uptheta\in\tilde{\mathcal{B}}(\R)$, $(\bast\Z^{s})^{\rm sym}(\Uptheta )$ is non trivial;  again, due to the lack of a good notion of
matrix best approximations, we can only conjecture that
for general $\Uptheta\in\tilde{\mathcal{M}}(\R)$, $(\bast\Z^{s})^{\rm sym}(\Uptheta )\not=0$.
\end{enumerate}


\section{Approximate Ideal Arithmetic of $\mathcal{O}$-Approximation Groups}\label{KIdeologicalArith}

 Let $K/\Q$ be a finite extension of degree $d$ with ring of integers $\mathcal{O}$.  Since $K$ possesses $d$ places $\uptau:K\hookrightarrow \C$
 we index the coordinates of $ \boldsymbol z =(z_{\uptau})\in\C^{d}$ using the places of $K$.  Let
\[ \K:=\{ \boldsymbol z =(z_{\uptau})\in\C^{d}|\; \bar{z}_{\uptau} = z_{\bar{\uptau}} \} \cong\R^{r}\times\C^{s}\cong\R^{d}\] be the Minkowski space: the archimedean part of the $K$-adeles, a finite-dimensional $\R$-algebra. 
 $\K\subset\C^{d}$ receives the restriction of the hermitian metric on $\C^{d}$, and
we regard $\R\subset\K$ via the diagonal embedding.  If $K/\Q$ is Galois then the Galois group ${\rm Gal}(K/\Q)$ acts on $\K$ via
$\upsigma (\boldsymbol z )=(z_{\upsigma\cdot \uptau})$, where $ \upsigma\cdot \uptau := \uptau\circ \upsigma^{-1}$ for $\upsigma\in {\rm Gal}(K/\Q)$.  In particular, ${\rm Gal}(K/\Q)$ acts via hermitian isometries on $\K$ since its action is by coordinate permutation; moreover, it acts trivially on $\R=$ the Minkowski space of $\Q$.  Therefore, the induced action on $\bast\K$ is bicontinuous.
Denote by ${\rm N}:\bast \K\rightarrow\bast\R$ the norm map: $\bast \boldsymbol z = (\bast z_{i})\mapsto {\rm N}(\bast \boldsymbol z)=\bast z_{1}\cdots \bast z_{d}$.

\vspace{3mm}

\noindent  {\em $K$-Tropical Semi-ring}

\vspace{3mm}

\noindent Consider the ring $\bast\K_{\rm fin}\subset\bast\K$ of elements all of whose coordinates are bounded.  The group of units is the subgroup $\bast\K^{\times}_{\rm fin}$ of elements all of whose coordinates are
non-infinitesimal and non-infinite.  The multiplicative quotient
\[     \bstar\PR\K  =\bast\K/ \bast\K^{\times}_{\rm fin} \] is partially ordered
and directed along the coordinates.  We will denote elements of $ \bstar\PR\K$ by $\boldsymbol \upmu$.
There is a diagonal inclusion $\bstar\PR\R\hookrightarrow \bstar\PR\K$; the image of $\upmu$ will be denoted $\upmu$ (not bold).

As in the case of $K=\Q$, $\bstar\PR\K$ has the structure
of a tropical semi-ring with respect to the induced product, and the sum defined
\[  \boldsymbol \upmu + \boldsymbol \upmu' = (\upmu_{1}+\upmu'_{1},\dots ,\upmu_{d}+\upmu_{d}' ). \] 
Note that $ \boldsymbol \upmu + \boldsymbol \upmu'$ is the least element greater than or equal to $ \boldsymbol \upmu, \boldsymbol \upmu'$.
The neutral element for $+$ is $-\boldsymbol \infty = (-\infty,\dots ,-\infty )$.
Elements of $\bast K$ act by multiplication on the left of $\bstar\PR\K$, respecting in the style
of Proposition \ref{troprop}  the tropical structure.  The norm (product of coordinates) map induces 
a multiplicative map ${\rm N}:  \bstar\PR\K\rightarrow \bstar\PR\R$.

Note that $\bstar\PR\C := \bast\C /\bast\C_{\rm fin}^{\times}$ is isomorphic to $\bstar\PR\R$ since 
$\bast\C_{\rm fin}^{\times}\supset \SI^{1}$.  Thus the non archimedean valuation $\langle\cdot\rangle:\bast\R \rightarrow \bstar\PR\R$,
extended in the obvious way to $\bast\C$, takes values in $\bstar\PR\R$.  We have
\[ \bstar\PR\K\cong (\bstar\PR\R )^{r}\times   (\bstar\PR\C )^{s}\cong \bstar\PR\R^{r+s},\]
so that every element $\boldsymbol \upmu$ may be written in the reduced coordinate form $(\upmu_{1},\dots ,\upmu_{r+s})$.

Let $\langle \cdot \rangle : \bast\K\rightarrow \bstar\PR\K$ be the canonical projection.  Let $\bast\K_{\upvarepsilon}\subset \bast\K_{\rm fin}$ be the (neither maximal nor prime) ideal of elements all of whose coordinates are infinitesimal and
denote $\bstar\PR\K_{\upvarepsilon}  =\bast\K_{\upvarepsilon}/ \bast\K^{\times}_{\rm fin}$.
Define also
\[ \bstar\PR\K_{\rm N(\upvarepsilon)}=\{ \boldsymbol \upmu \in  \bstar\PR\K |\; {\rm N}( \boldsymbol \upmu ) \in \bstar\PR\R_{\upvarepsilon}\}\supset \bstar\PR\K_{\upvarepsilon} .\]  While $\bstar\PR\K_{\rm N(\upvarepsilon)}$ is not closed with respect to $+$, 
it is downwardly closed with respect to the partial order ($\boldsymbol \upmu\in \bstar\PR\K_{\rm N(\upvarepsilon)}$ and $\boldsymbol \upmu'<\boldsymbol \upmu$ implies $\boldsymbol \upmu'\in  \bstar\PR\K_{{\rm N}(\upvarepsilon)}$) and moreover, each interval
$[-\boldsymbol\infty,\boldsymbol \upmu ] = \{ \boldsymbol \upmu'|\;  \boldsymbol \upmu'\leq\boldsymbol\upmu\}\subset \bstar\PR\K_{\rm N(\upvarepsilon)}$ is a sub tropical semi ring (w/o unit).
Note that if we define 
\[\bstar\PR\K_{{\rm Tr}(\boldsymbol\upvarepsilon)}=\{ \boldsymbol\upmu \in  \bstar\PR\K  |\;
{\rm Tr}(\boldsymbol\upmu ):= \upmu_{1}+ \cdots +\upmu_{r+s}=\max ( \upmu_{i})\in\bstar\PR\R_{\upvarepsilon}\},\] then 
$ \bstar\PR\K_{{\rm Tr}(\boldsymbol\upvarepsilon)}= \bstar\PR\K_{\upvarepsilon}$.
 
Let $\langle\langle\cdot \rangle\rangle:\bstar\PR\K\rightarrow \bstar\PR\R$ be the map
induced by composing the euclidean norm $\|\cdot \|$ on $\bast\K$ (valued in $\bast\R_{+}$)
with $\langle\cdot\rangle$: in other words,
$\langle\langle\cdot\rangle\rangle = \langle \|\cdot \| \rangle$.  In particular, If $\boldsymbol\upmu$ is represented
by $(\bast x_{1},\dots ,\bast x_{r+s})$, we have
\begin{align*} 
\langle\langle\boldsymbol\upmu\rangle\rangle & =\left\langle \sqrt{ \bast x^{2}_{1}+\dots + \bast x^{2}_{r+s} }\right\rangle 
  = \sqrt{\langle\bast x^{2}_{1}+\dots + \bast x^{2}_{r+s} \rangle} \\
 & = \sqrt{ \max (\upmu_{i}^{2}) } 
 =  \upmu_{1}+ \cdots + \upmu_{r+s} =  
{\rm Tr}(\boldsymbol\upmu ).
\end{align*}

\vspace{3mm}

\noindent  {\em $K$-approximate ideals}

\vspace{3mm}

Let $\boldsymbol z\in\K$.  Define the group of {\bf  $\tv{O}$-diophantine approximations} as
\[ \bast\mathcal{O}(\boldsymbol z) = \{\bast\upalpha\in\bast\mathcal{O} |\; \exists \bast\upalpha^{\perp} \text{ s.t. } \bast\upalpha\boldsymbol z-\bast\upalpha^{\perp}
\in  \bstar\PR\K_{\upvarepsilon}
\}.\]  This group was first introduced in \cite{Ge4}.

 \noindent For each ${\boldsymbol \upmu}\in \bstar\PR\K_{{\rm N}(\upvarepsilon)} $ define 
 \[ \bast\mathcal{O}^{\boldsymbol \upmu } = \{\bast\upalpha\in\bast\mathcal{O}|\; \bast\upalpha\cdot \boldsymbol \upmu \in \bstar\PR\K_{\upvarepsilon}\} \subset \bast\mathcal{O} .\]
If
$\boldsymbol \upmu <\boldsymbol \upmu'$ then $\bast\mathcal{O}^{\boldsymbol \upmu }\supset \bast\mathcal{O}^{\boldsymbol \upmu '}$ however if $\boldsymbol \upmu$ and $\boldsymbol \upmu'$ are unrelated by the order, the associated groups are unrelated by inclusion.
The fine growth subfiltration $\bast\mathcal{O}^{\boldsymbol \upmu [\boldsymbol \upiota] }$ is defined by 
$\bast\upalpha\cdot\boldsymbol\upmu <\boldsymbol\upiota$ where $\boldsymbol\upiota\in\bstar\PR\K_{\upvarepsilon}$, a subgroup of $\bast\mathcal{O}^{\boldsymbol \upmu }$.  

The use of elements ${\boldsymbol \upmu}\in  \bstar\PR\K_{\rm N(\upvarepsilon)}$ to index growth is required as there
exist nonstandard integers $\bast \upalpha\in\bast\mathcal{O}$ having $\K$-coordinates $(\bast\upalpha_{1},\dots ,\bast\upalpha_{d})$
exhibiting inhomogeneous growth.  For example, if $\upalpha$ is a Salem number, 
the class of the sequence $\{ \upalpha^{i}\}$
has a bounded coordinate lying on $\bast \SI^{1}$, an infinitesimal coordinate and an infinite coordinate.

Now given $\boldsymbol z\in\K$ and $\boldsymbol \upnu\in \bstar\PR\K_{\upvarepsilon}$  
we define $\bast\mathcal{O}_{\boldsymbol \upnu }(\boldsymbol z )\subset \bast\mathcal{O}(\boldsymbol z )$ as the subgroup of 
$\bast\upalpha$ for which 
$\langle\boldsymbol\upvarepsilon (\bast\upalpha )\rangle \leq \boldsymbol\upnu$, where
$\boldsymbol\upvarepsilon (\bast\upalpha )  = \bast\upalpha\boldsymbol z -\bast\upalpha^{\perp}$.  We have
\[ \langle\boldsymbol\upvarepsilon (\bast\upalpha + \bast\upalpha' )\rangle \leq 
\langle\boldsymbol\upvarepsilon (\bast\upalpha )\rangle + \langle\boldsymbol\upvarepsilon (\bast\upalpha ' )\rangle \leq \boldsymbol \upnu
  \]
so $\bast\mathcal{O}_{\boldsymbol \upnu }(\boldsymbol z)$ is a group.
Here we note that we can extend the definition of $\bast\mathcal{O}_{\boldsymbol \upnu }(\boldsymbol z)$ 
to indices $\boldsymbol \upnu\in \bstar\PR\K_{{\rm N}(\upvarepsilon )}$, also obtaining a group, since $[-\boldsymbol\infty, \boldsymbol\upnu )$ is a sub tropical 
semi-ring. 
We have thus defined the $K$-approximate ideal structure 
\[ \bast\mathcal{O}({\boldsymbol z})=\Big\{ \bast\mathcal{O}^{\boldsymbol \upmu[\boldsymbol\upiota]}_{\boldsymbol \upnu}({\boldsymbol z})\Big\},
\quad  \bast\mathcal{O}^{\boldsymbol \upmu[\boldsymbol\upiota]}_{\boldsymbol \upnu}({\boldsymbol z})=
\bast\mathcal{O}^{\boldsymbol \upmu[\boldsymbol\upiota]}\cap\bast\mathcal{O}_{\boldsymbol \upnu}({\boldsymbol z})\]
for $\boldsymbol z\in\K$ and $\boldsymbol \upmu,\boldsymbol \upnu,\boldsymbol \upiota\in \bstar\PR\K_{{\rm N}(\upvarepsilon )}$. 
In particular, $\bast\mathcal{O}({\boldsymbol z})$ is a $\bast\mathcal{O}$-approximate module, and we may speak of approximate module homomorphisms
between such $K$-approximate ideals.
 Note the compatibility 
$\bast\Z_{\upnu}^{\upmu}(\uptheta)\subset\bast\mathcal{O}_{ \upnu}^{\upmu}(\uptheta)$.

\vspace{3mm}

\noindent  {\em $K$-Nonvanishing Spectra}

\vspace{3mm}

\noindent The $K$-nonvanishing spectrum is 
\[ {\rm Spec}_{K}(\boldsymbol z )=\{ (\boldsymbol\upmu,\boldsymbol\upnu)|\; \bast\mathcal{O}^{\boldsymbol \upmu}_{\boldsymbol \upnu}({\boldsymbol z})\not=\boldsymbol 0  \}\subset \bstar\PR\K_{{\rm N}(\upvarepsilon)}^{2}.\] 
 Note first that trivially
\begin{prop}\label{triviallynonvacO}   Let $\uptheta\in\R$.  If $\boldsymbol \upmu <\boldsymbol \upnu\in\bstar\PR\K_{\upvarepsilon}$  then 
$\bast\mathcal{O}_{\boldsymbol \upnu}^{\boldsymbol \upmu}(\uptheta)\not=0$.  If $\boldsymbol z = \upgamma\in K$ (diagonally embedded
in $\K$) then $ {\rm Spec}_{K}(\upgamma )=\bstar\PR\K_{{\rm N}(\upvarepsilon)}^{2}$.
\end{prop}

\begin{proof}Let   $\upmu <\upnu\in \bstar\PR\R_{\upvarepsilon}$ be such that 
$ \boldsymbol \upmu <\upmu<\upnu< \boldsymbol \upnu$ . 
Then 
$ 0\not=\bast\Z_{ \upnu}^{\upmu}(\uptheta)
\subset\bast\mathcal{O}_{\upnu}^{ \upmu}(\uptheta)\subset \bast\mathcal{O}_{\boldsymbol \upnu}^{\boldsymbol \upmu}(\uptheta)$.
The second statement is obvious.
\end{proof}

 Recall that $\uptheta\in\R$ is called a Pisot-Vijayaraghavan number
 if it is a real algebraic integer greater than $1$ for which all of its conjugates have absolute value $<1$.  In what follows,  we denote $P_{K,\uptheta}(X)=\prod (X-\uptheta_{\uptau})\in\Z[X]$ where the product is over the archimedean places $\uptau$ of the field $K$ and $\uptheta_{\uptau}:=\uptau(\uptheta )$.

 \begin{theo}\label{PVnumber}  Let $\uptheta\in\R$ be a Pisot-Vijayaraghavan number 
with $\uptheta\in \uptau (\mathcal{O})$ for some place $\uptau$ and for which $P_{K,\uptheta}(X)$ is the minimal polynomial of $\uptheta$.  Then there exists $\boldsymbol \upmu\in \bstar\PR\K_{\upvarepsilon}$ such that
$\bast\mathcal{O}_{\boldsymbol \upmu}^{\boldsymbol \upmu}(\uptheta )\not=0$.
\end{theo}

\begin{proof}  Let $\vartheta \in\mathcal{O}$ be such that $\vartheta_{\uptau_{1}}=\uptheta$, so that
$|\vartheta_{\uptau_{i}}|<1$ for $i=2,\dots d$.  Let $\bast\upalpha\in\bast\mathcal{O}$ be the class
associated to the sequence $\vartheta, \vartheta^{2},\vartheta^{3},\dots $ and let $\bast\upalpha^{\perp}=\vartheta\bast\upalpha\in\bast\mathcal{O}$.
  Then $\bast\upalpha \in \bast\mathcal{O}(\uptheta )$ with dual $\bast\upalpha^{\perp}$:
indeed $\bast\upalpha\uptheta -\bast\upalpha^{\perp} $ is the class of the vector sequence
\[\left\{  \left(0,\vartheta_{\uptau_{2}}^{n}(\uptheta-\vartheta_{\uptau_{2}}),\dots ,  \vartheta_{\uptau_{d}}^{n}(\uptheta-\vartheta_{\uptau_{d}}) \right ) \right\}_{n=1}^{\infty} \]
which is infinitesimal since $|\vartheta_{\uptau_{i}}|<1$ for $i=2,\dots d$.  Let $\boldsymbol\upmu\in \bstar\PR\K_{\upvarepsilon}$
be the class of $\bast\upalpha\uptheta -\bast\upalpha^{\perp}$.  Then
\[   \bast\upalpha\cdot \boldsymbol\upmu = (0,\bast\upalpha_{\uptau_{2}}\cdot \upmu_{2},\dots , \bast\upalpha_{\uptau_{d}}\cdot \upmu_{d} )
\in \bstar\PR\K_{\upvarepsilon}\]
since the components $\bast\upalpha_{\uptau_{2}},\dots , \bast\upalpha_{\uptau_{d}}$ are themselves infinitesimal.  Thus
$\bast\mathcal{O}_{\boldsymbol \upmu}^{\boldsymbol \upmu}(\uptheta )\not=0$.
\end{proof}

Theorem \ref{PVnumber} reveals that there are infinitely many antiprimes $\uptheta$ (quadratic Pisot Vijayaraghavan numbers) which possess a non-trivial flat spectra provided that we expand the field of approximants to one minimally containing $\uptheta$.  In particular, 
such a $\uptheta$ ceases to be antiprime, a phenomenon which may be described as
the ``splitting'' of the nonvanishing spectrum c.f.\ the {\it $K$-Classification} subsection below.

\vspace{3mm}

\noindent  {\em $K$-Approximate Ideal Arithmetic}

\vspace{3mm}

We have the following exact analogue of Theorem \ref{productformula}:

\begin{theo}[$K$-Approximate Ideal Arithmetic]\label{Ogrowthdecayprod} Let $\boldsymbol z ,\boldsymbol w \in\K$.  Then
\begin{align}\label{grdecprod}
 \bast\mathcal{O}_{\boldsymbol \upnu}^{\boldsymbol \upmu[\boldsymbol\upiota]}(\boldsymbol z)\cdot  \bast\mathcal{O}_{\boldsymbol \upmu}^{\boldsymbol \upnu[\boldsymbol \uplambda]}(\boldsymbol w) & \subset  
\bast\mathcal{O}^{\boldsymbol\upmu\cdot\boldsymbol\upnu [\boldsymbol\upiota\cdot\boldsymbol \uplambda]}_{\boldsymbol\upiota+\boldsymbol \uplambda}(\boldsymbol z\boldsymbol w ) \cap 
 \bast\mathcal{O}^{\boldsymbol\upmu\cdot\boldsymbol\upnu [\boldsymbol\upiota\cdot\boldsymbol \uplambda]}_{\boldsymbol\upiota+\boldsymbol \uplambda}(\boldsymbol z+\boldsymbol w ) \cap  \bast\mathcal{O}^{\boldsymbol\upmu\cdot\boldsymbol\upnu [\boldsymbol\upiota\cdot\boldsymbol \uplambda]}_{\boldsymbol\upiota+\boldsymbol \uplambda}(\boldsymbol z-\boldsymbol w ).
\end{align}
\end{theo}

\begin{proof}  If $\bast \upalpha\in  \bast\mathcal{O}_{\boldsymbol \upnu}^{\boldsymbol \upmu[\boldsymbol\upiota]}(\boldsymbol z)$
and $\bast \upbeta\in\bast\mathcal{O}_{\boldsymbol \upmu}^{\boldsymbol \upnu[\boldsymbol \uplambda]}(\boldsymbol w)$
then
\[    \bast\upalpha\bast \upbeta\cdot\boldsymbol z\boldsymbol w = (\bast\upalpha_{1}\bast \upbeta_{1}z_{1}w_{1}, \dots ,\bast\upalpha_{d}\bast \upbeta_{d}z_{d}w_{d}) . \]
 The proof proceeds as in that of Theorem \ref{productformula}, implemented along the coordinates
of $\bast\K$. 
\end{proof}

The remarks following Theorem \ref{productformula} apply just as well to $\mathcal{O}$-Approximate Ideal arithmetic.
Here however non principal ideals are absent, not surprising since the definition of $\mathcal{O}$-diophantine approximation
groups is made with regard to single elements of $\K$.  In \S \ref{ideologicalclasssection} we
will produce the analogues of (classes) of two generator ideals by ``decoupling'' numerator denominator pairs.
 Non principal ideals also appear naturally as dual diophantine approximations of vectors, see \S \ref{matideoarith}.

For $\boldsymbol z,  \boldsymbol w\in\K$, we write
\begin{align}\label{prodnotation2} \boldsymbol z {}_{\boldsymbol\upmu}\text{\circled{$\mathcal{O}$}}{}_{\boldsymbol\upnu} \boldsymbol w
\end{align} when the groups $\bast\mathcal{O}_{\boldsymbol \upnu}^{\boldsymbol \upmu}(\boldsymbol z)$, 
$\bast\mathcal{O}_{\boldsymbol \upmu}^{\boldsymbol \upnu}(\boldsymbol w)$
are nontrivial. 
Thus when $\mathcal{O}=\Z$, $\text{\circled{$\Z$}}=\owedge$.
 The symbol
\[ \bast \upalpha \;{}_{\boldsymbol\upmu}\circled{\small $\boldsymbol z\boldsymbol w$}{}_{\,\boldsymbol\upnu}\bast  \upbeta \]
will indicate that the product $\bast\upalpha\bast \upbeta$ is defined as one of diophantine approximations, subject to the condition that
$\bast\upalpha\in\bast\mathcal{O}_{\boldsymbol \upnu}^{\boldsymbol \upmu}(\boldsymbol z)$ resp.\ $\bast\upbeta\in \bast\mathcal{O}_{\boldsymbol \upmu}^{\boldsymbol \upnu}(\boldsymbol w)$.   The notions of $\mathcal{O}$-fast, $\mathcal{O}$-slow and $\mathcal{O}$-flat divisors are defined as in \S \ref{gdarith}.  
 When
  $\boldsymbol \upmu$ and $\boldsymbol \upnu$ are not related by the order, we say that the factors are
{\bf  oscillatory divisors}.

Let ${\rm PGL}_{2}(\mathcal{O})$ be the projective linear group with entries in $\mathcal{O}$.  Then ${\rm PGL}_{2}(\mathcal{O})$
partially\footnote{Or rather acts fully on a suitable compactification of $\K$.  In any case, ${\rm PGL}_{2}(\mathcal{O})$ acts fully on $\K-K$.} acts on $\K$ and we write
\[  \boldsymbol z \Bumpeq_{K}\boldsymbol z'  \]
if there exists $A\in {\rm PGL}_{2}(\mathcal{O})$ such that $A( \boldsymbol z)= \boldsymbol z'$.

\begin{theo}  If $ \boldsymbol z \Bumpeq_{K}\boldsymbol z' $ by $A\in {\rm PGL}_{2}(\mathcal{O})$ then $A$ induces an approximate module isomorphism
\[  A:\bast\mathcal{O}(\boldsymbol z)\longrightarrow \bast\mathcal{O}(\boldsymbol z').\]
If in addition we have $\boldsymbol w\in\K$ and $B\in {\rm PGL}_{2}(\mathcal{O})$ then
\[  \boldsymbol z {}_{\boldsymbol\upmu}\text{\circled{$\mathcal{O}$}}{}_{\boldsymbol\upnu} \boldsymbol w
\quad \Longleftrightarrow \quad A(\boldsymbol z) {}_{\boldsymbol\upmu}\text{\circled{$\mathcal{O}$}}{}_{\boldsymbol\upnu} B(\boldsymbol w).
 \]
\end{theo}

\begin{proof}  Same idea as the proof of Theorem \ref{triisomorphism}, implemented along place coordinates.
\end{proof}


\begin{theo}\label{galrespects}  Suppose that $K/\Q$ is Galois and $\upsigma\in{\rm Gal}(K/\Q)$.   Then
\[\upsigma \bigg(\bast\mathcal{O}_{\boldsymbol \upnu}^{\boldsymbol \upmu[\boldsymbol\upiota]}(\boldsymbol z)\bigg) 
= \bast\mathcal{O}_{\upsigma(\boldsymbol \upnu)}^{\upsigma(\boldsymbol \upmu)[\upsigma(\boldsymbol\upiota)]}\big(\upsigma(\boldsymbol z)\big).
\]
In particular,
\[\boldsymbol z {}_{\boldsymbol\upmu}\text{\circled{$\mathcal{O}$}}{}_{\boldsymbol\upnu} \boldsymbol w \quad \Longleftrightarrow\quad
\upsigma(\boldsymbol z) {}_{\upsigma(\boldsymbol\upmu)}\text{\circled{$\mathcal{O}$}}{}_{\upsigma(\boldsymbol\upnu)} \boldsymbol \upsigma(w)
  \]
  and
 \[ \bast \upalpha\, {}_{\boldsymbol\upmu}\circled{\small $\boldsymbol z\boldsymbol w$}{}_{\boldsymbol\upnu}\bast  \upbeta \quad \Longleftrightarrow \quad
 \upsigma\big(\bast \upalpha \big) {}_{ \upsigma(\boldsymbol\upmu)}\circled{\tiny $\upsigma(\boldsymbol z)\upsigma(\boldsymbol w)$}{}_{ \upsigma(\boldsymbol\upnu)}\upsigma \big(\bast  \upbeta \big).  \]
Therefore an elements status as an $\mathcal{O}$-fast, $\mathcal{O}$-slow, $\mathcal{O}$-flat or $\mathcal{O}$-oscillatory divisor is preserved
 by the action of $\upsigma\in{\rm Gal}(K/\Q)$.
\end{theo}
\begin{proof} Let $\upsigma\in{\rm Gal}(K/\Q)$, then 
\[\upsigma \bigg(\bast\mathcal{O}_{\boldsymbol \upnu}^{\boldsymbol \upmu[\boldsymbol\upiota]}(\boldsymbol z)\bigg) 
= \bast\mathcal{O}_{\upsigma(\boldsymbol \upnu)}^{\upsigma(\boldsymbol \upmu)[\upsigma(\boldsymbol\upiota)]}\big(\upsigma(\boldsymbol z)\big).
\]
 It follows immediately that $\upsigma$ respects the $K$-approximate ideal product as indicated in the statement of the Proposition.
\end{proof}



    If we take  if $\boldsymbol z = \uptheta=(\uptheta ,\dots ,\uptheta )$ then the trace map ${\rm Tr}:\K\rightarrow \R$
defines a well-defined homomorphism of groups ${\rm Tr}:\bast\mathcal{O}(\uptheta )\rightarrow
\bast\Z(\uptheta )$.  In addition, we have a well-defined map
of projective classes
$  {\rm Tr}: \bstar \PR\K_{\upvarepsilon}\longrightarrow \bstar\PR\R_{\upvarepsilon}
 $
so that
\[   {\rm Tr}\bigg(\bast\mathcal{O}_{\boldsymbol \upnu}(\uptheta )\bigg)\subset
\bast\Z_{ {\rm Tr}(\boldsymbol \upnu)}(\uptheta ).
\]
Note that the trace map does {\it not} map 
$\bast\mathcal{O}^{\boldsymbol \upmu}(\uptheta )$ to $\bast\Z^{ {\rm Tr}(\boldsymbol \upmu)}(\uptheta )$.

On the other hand, the norm map, as we have seen, induces 
$   {\rm N}: \bstar \PR\K_{{\rm N}(\upvarepsilon )}\longrightarrow \bstar\PR\R_{\upvarepsilon}$, however it does not define a map
from $\bast\mathcal{O}_{\boldsymbol \upnu}(\uptheta )$ to $\bast\Z_{ {\rm N}(\boldsymbol \upnu)}(\uptheta )$.  Instead
it yields a map of fine growth filtrations
\[{\rm N}:\bast \mathcal{O}^{\boldsymbol \upmu[\boldsymbol \upiota]}\longrightarrow
\bast\Z^{{\rm N}(\boldsymbol \upmu)[{\rm N}(\boldsymbol \upiota )]}. \]
As an immediate corollary, we may deduce that for $\boldsymbol \upmu$ to be a growth index for a nonstandard integer $\bast\upalpha\in\bast\mathcal{O}$,
the product of the infinitesimal coordinates of $\boldsymbol \upmu$ must dominate the product of the non infinitesimal coordinates of $\boldsymbol \upmu$:

\begin{prop}  If $\boldsymbol \upmu\in\bstar\PR\K$ and $\bast\mathcal{O}^{\boldsymbol \upmu}\not=0$
then $\boldsymbol \upmu\in\bstar\PR\R_{{\rm N}(\upvarepsilon)}$.
\end{prop}

\begin{proof}  No element of $\bast\Z$ is infinitesimal, so 
${\rm N}(\bast\upalpha )\cdot {\rm N}(\boldsymbol \upmu )\in \bstar\PR\R_{\upvarepsilon}$ can only occur if ${\rm N}(\boldsymbol \upmu )$
is infinitesimal.
\end{proof}

From the above paragraphs, we deduce the following broad principle: 
\begin{quotation}
Growth is multiplicative but not additive. Decay is additive but not multiplicative.
\end{quotation}
In general, if we seek to return to the ground ring $\bast\Z$ using either the norm or the trace, one of the growth-decay parameters must be sacrificed.
Nevertheless, there exist specific situations when the anomalous parameter can be controlled.

 \begin{prop}\label{normgrowthdecay} Let $K/\Q$ be of degree $2$ and let $\sigma$ be the nontrivial element of its Galois group.
  Then for any $\boldsymbol z \in\K$,
  \[  {\rm N}\bigg(  \bast\mathcal{O}^{\boldsymbol\upmu[\boldsymbol \upiota]}_{\sigma(\boldsymbol\upmu)}(\boldsymbol z )    \bigg)
  \subset \bast\Z^{{\rm N}(\boldsymbol\upmu )[{\rm N}(\boldsymbol\upiota )]}_{{\rm Tr}(\boldsymbol\upiota )}\big({\rm N}(\boldsymbol z)\big).
  \]
 \end{prop}

\begin{proof}  Let $\bast\upalpha\in \bast\mathcal{O}^{\boldsymbol\upmu[\boldsymbol \upiota]}_{\sigma(\boldsymbol\upmu)}(\boldsymbol z )$.  Write 
$\bast\upalpha = (\bast\upalpha_{1},\bast\upalpha_{2}), \bast\upalpha^{\perp} = (\bast\upalpha_{1}^{\perp},\bast\upalpha_{2}^{\perp}) $
as well as $\upvarepsilon (\bast\upalpha ) =$ $(\upvarepsilon (\bast\upalpha_{1} ), \upvarepsilon (\bast\upalpha_{2} ))$.
It is immediate that ${\rm N}(\bast\upalpha ) =\bast\upalpha_{1}\bast\upalpha_{2}\in\bast\Z^{{\rm N}(\boldsymbol\upmu )[{\rm N}(\boldsymbol\upiota )]}$.
On the other hand,
\begin{align}\label{normerrorterm} {\rm N}(\bast\upalpha )\cdot {\rm N}(\boldsymbol z)-{\rm N}(\bast\upalpha^{\perp} ) & =
{\rm N}(\upvarepsilon (\bast\upalpha ) )+ \bast\upalpha_{1}^{\perp}(\bast\upalpha_{2}z_{2}-\bast\upalpha_{2}^{\perp})+
\bast\upalpha_{2}^{\perp}(\bast\upalpha_{1}z_{1}-\bast\upalpha_{1}^{\perp}) \nonumber \\
& = {\rm N}(\upvarepsilon(\bast\upalpha ) ) + \bast\upalpha_{1}^{\perp}\cdot\upvarepsilon (\bast\upalpha_{2} )+ \bast\upalpha_{2}^{\perp}\cdot \upvarepsilon (\bast\upalpha_{1} )  .
\end{align}
Since $\bast\upalpha\in \bast\mathcal{O}_{\sigma(\boldsymbol\upmu)}(\boldsymbol z ) $, the image of (\ref{normerrorterm}) by $\langle\cdot\rangle$
belongs to $\bstar\PR\R_{\upvarepsilon}$.
Since 
\[ \langle{\rm N}(\upvarepsilon (\bast\upalpha ) ) \rangle < \langle\bast\upalpha_{1}^{\perp}\cdot \upvarepsilon (\bast\upalpha_{2} )\rangle,
\langle \bast\upalpha_{2}^{\perp}\cdot\upvarepsilon (\bast\upalpha_{1} )\rangle\] we may disregard $\langle{\rm N}(\upvarepsilon (\bast\upalpha ) ) \rangle$,
and therefore the image of (\ref{normerrorterm}) by $\langle\cdot \rangle$ is bounded by ${\rm Tr}(\boldsymbol \upiota )$.
\end{proof}

\begin{note}  In view of the nature of the image growth-decay indices occurring in Proposition \ref{normgrowthdecay},
we cannot use the norm to push products down of the form presented in Theorem \ref{Ogrowthdecayprod}.
\end{note}

There is a similar sort of result for the trace.  Given $\boldsymbol\upmu\in\bstar\PR\K$, define the {\bf  lower trace} to be
\[   {\rm tr}( \boldsymbol\upmu) := \min \upmu_{i} . \] Note that if 
$\upmu\in\bstar\PR\R$ then ${\rm tr}(\upmu) ={\rm Tr}( \upmu)=\upmu $.

\begin{prop}\label{TraceKDA} Let $K/\Q$ be of degree $d$.
  Then for any $\uptheta\in\R$,
  \[  {\rm Tr}\bigg(  \bast\mathcal{O}^{\boldsymbol\upmu[\boldsymbol \upiota]}_{\boldsymbol\upnu}( \uptheta )    \bigg)
  \subset \bast\Z^{{\rm tr}(\boldsymbol\upmu )[{\rm Tr}(\boldsymbol\upiota )]}_{{\rm Tr}(\boldsymbol\upnu )}(\uptheta).
  \]
  If $\boldsymbol\upnu=-\boldsymbol\infty$, the result is valid for $\uptheta$ replaced by $\upgamma\in K\subset\K$.
 \end{prop}
 
 \begin{proof}  We have already observed that the trace map preserves decay for $\uptheta\in\R$.  If
 the decay is $-\boldsymbol\infty$ then the trace map preserves the decay for $\upgamma\in K$.  On the other hand, the inequality
 $\bast\upalpha \cdot\boldsymbol\upmu <\boldsymbol\upiota$ may be rewritten
$ (\bast\upalpha_{1}\cdot \upmu_{1},\cdots ,\bast\upalpha_{d}\cdot \upmu_{d})< (\upiota_{1},\dots ,\upiota_{d})$.  
 It follows that
 \[{\rm Tr}(\bast\upalpha)\cdot {\rm tr}(\upmu ) =(\upalpha_{1}+\cdots +\upalpha_{d}) \cdot \min (\upmu_{i})
 <\upiota_{1}+\cdots + \upiota_{d} = {\rm Tr}(\boldsymbol\upiota ).\]
 \end{proof}
 
 
 \begin{coro}  If $\upmu,\upnu\in \bstar\PR\R_{\upvarepsilon}$, $\uptheta,\upeta\in\R$ and
$ \bast \upalpha \;{}_{\upmu}\circled{\small $\uptheta\upeta$}{}_{\upnu}\bast  \upbeta $
then ${\rm Tr}( \bast \upalpha ) \;{}_{\upmu}\circled{\small $\uptheta\upeta$}{}_{\upnu}{\rm Tr}(\bast  \upbeta )$.
If $\upnu=-\infty$ the result holds for $\upgamma,\updelta\in K$.
 \end{coro}
 
 Thus the trace map respects the {\it approximate ideal arithmetic} of classical (principal) ideals.

\vspace{3mm}

\noindent  {\em $K$-Classification}

\vspace{3mm}

Diophantine approximation by the $K$-integers $\mathcal{O}$, as formulated in \cite{Ge4}, is {\it global},
since it is performed with respect to all the archimedean places at once --  in contrast to the classical notion of 
Diophantine approximation
by elements of $K$ \cite{Bu}, which is {\it local}, framed with respect to a fixed archimedean place.
It is therefore reasonable to define $K$-versions of the usual linear classification of the reals
as described in \S \ref{nonvanspec}.  Note that the classification theory of Koksma, though based on
approximations by algebraic numbers, 
is not field specific, and is therefore not relevant to the considerations of this section.

We say that $\boldsymbol z\in\K-K$ is 
\begin{itemize}
\item[-] {\bf $\boldsymbol K$-badly approximable} if $\bast\mathcal{O}^{\boldsymbol \upmu}_{\boldsymbol \upmu}(\boldsymbol z)=0$
for all $\boldsymbol \upmu\in \bstar\PR\K_{\upvarepsilon }$.
\item[-] {\bf  $\boldsymbol K$-well approximable} if it is not $K$-badly approximable.
\item[-] {\bf  $\boldsymbol K$-very well approximable of exponent $\boldsymbol\upkappa$} if there exists $\boldsymbol \upmu\in  \bstar\PR\K_{\upvarepsilon }$ such that $\bast\mathcal{O}^{\boldsymbol \upmu}_{\boldsymbol \upmu^{\upkappa'}}(\boldsymbol z)=0$ for all
$\upkappa'>\upkappa$ and
\[ \bigcap_{\uplambda\in [1,\upkappa)} \bast\mathcal{O}^{\boldsymbol \upmu}_{\boldsymbol \upmu^{\uplambda}}(\boldsymbol z)\not=0.\]
\item[-] {\bf  $\boldsymbol K$-Liouville} if it is $K$-very well approximable of exponent $\upkappa$ for all $\upkappa>1$. \end{itemize}

Denote these classes by $\mathfrak{B}(K)$, $\mathfrak{W}(K)$, $\mathfrak{W}_{\upkappa}(K)$ and $\mathfrak{W}_{\infty}(K)$, 
any one of which is denoted $\mathfrak{C}(K)$; the $K=\Q$ counterpart
is simply denoted $\mathfrak{C}$.  Note that 
for any class $\mathfrak{C}$ there exists $\uptheta\in\mathfrak{C}$ such that $\uptheta\not\in\upsigma (K)$
for every archimedean place $\upsigma$ of $K$.

\begin{theo}  Let $\uptheta\in\mathfrak{C}$ be such that $\uptheta\not\in\upsigma (K)$
for every archimedean place $\upsigma$ of $K$.
Then $\uptheta\in\mathfrak{C}(K)$.  In particular, the class $\mathfrak{C}(K)$ is nontrivial for all $K/\Q$ of finite degree.
\end{theo}

\begin{proof} The only nontrivial case is when $\mathfrak{C}(K)=\mathfrak{B}(K)$. Suppose that there exists $\boldsymbol\upmu\in\bstar\PR\K_{\upvarepsilon}$
with $\bast\mathcal{O}^{\boldsymbol \upmu}_{\boldsymbol \upmu}(\boldsymbol \uptheta)\not=0$.  Since $\uptheta\not\in\upsigma (K)$
for every archimedean place $\upsigma$, none of the coordinates of $\boldsymbol\upmu$ are equal to $-\infty$.  Therefore,
we may find a diagonal element $\upmu = (\upmu,\dots ,\upmu )$, where $\upmu\in\bstar\PR\R_{\upvarepsilon}$, such that
$\boldsymbol\upmu<\upmu<1$ in which $\bast\upalpha\cdot\upmu\in \PR\K_{\upvarepsilon}$.  For such and element
we have {\it a fortiori} that $\boldsymbol\upnu (\bast \upalpha )\leq \upmu$ so that $\bast\upalpha\in
\bast\mathcal{O}^{ \upmu}_{ \upmu}(\boldsymbol \uptheta)\not=0$ as well.  Applying
Proposition \ref{TraceKDA}, and using the fact that ${\rm Tr}(\upmu)={\rm tr}(\upmu )=\upmu$ we obtain
$0\not= {\rm Tr}(\bast\upalpha )\in\bast\Z^{\upmu}_{\upmu}(\uptheta )$, which is a contradiction.
\end{proof}

The $K$-classes defined above are Galois natural: 
\begin{theo}  If $K/\Q$ is a Galois extension, then for any $K$-class $\mathfrak{C}(K)$ and $\upsigma\in {\rm Gal}(K/\Q)$, $\upsigma (\mathfrak{C}(K))=\mathfrak{C}(K)$, acting as the identity on $\mathfrak{C}\subset\mathfrak{C}(K)$.
\end{theo}

\begin{proof}  This follows from Theorem \ref{galrespects} and the continuity of the Galois action on $\bstar\PR\K$.
\end{proof}

If $\uptheta\in\mathfrak{C}$ but $\uptheta\not\in \mathfrak{C}(K)$ then we say that $\uptheta$ {\bf  splits} in $K$.  For example,Theorem \ref{PVnumber} says that
any quadratic Pisot-Vijayaraghavan number $\uptheta$ with a conjugate in $K$ splits in $K$.
The factorization symbols $\Uparrow_{K}$, $\Downarrow_{K}$, $\Updownarrow_{K}$ and $\talloblong_{K}$ have the expected meanings
in the $K$-context, are Galois natural, and give rise to notions of $K$-antiprimes $K$-omnidivisors, etc.  The treatment of these and more advanced topics
(such as $K$-flat arithmetic) will be deferred to another study. 

 The $\mathcal{O}$-approximation theory may be merged with the matrix theory in the more or less expected way.  Given $K/\Q$ a finite extension of degree $d$, 
consider the sets $\tilde{\mathcal{M}}(\K)\supset\mathcal{M}(\K)$ of matrices and square matrices with entries
in $\K$, equipped with the Kronecker product as well as the Kronecker sum in the square case.   For $\boldsymbol\Uptheta\in \tilde{\mathcal{M}}(\K)$ of dimension $r\times s$, the diophantine approximation group is denoted $\bast\mathcal{O}^{s}(\boldsymbol\Uptheta)$ (see \cite{Ge4}).

 If 
$\bast\boldsymbol\upalpha = (\bast\upalpha_{1},\dots ,\bast\upalpha_{s})\in\mathcal{O}^{s}$, let
$\bast\upalpha_{i}\in\K$ have place-indexed coordinates 
\[ \bast\upalpha_{i}=(\bast\upalpha_{i,\tau_{1}},\dots ,\bast\upalpha_{i,\tau_{d}})\]
and define
\[ \overline{|\bast\boldsymbol\upalpha|} = (\max_{i=1,\dots ,s} |\bast\upalpha_{i,\tau_{1}}  |,\dots ,\max_{i=1,\dots ,s} |\bast\upalpha_{i,\tau_{d}}  |  ) 
\in\bast\R_{+}^{d}. \]
Write $\boldsymbol\upmu (  \bast\boldsymbol\upalpha )=\langle \overline{|\bast\boldsymbol\upalpha|} \rangle^{-1}$ (well-defined provided
$\bast\boldsymbol\upalpha$ is not the zero vector) and similarly
 $\boldsymbol\upnu (  \bast\boldsymbol\upalpha ) =  \overline{|\upvarepsilon(\bast\boldsymbol\upalpha)|}$.
With these definitions we obtain the approximate ideal 
\[ \bast\mathcal{O}^{s}(\boldsymbol\Uptheta)=\{(\bast\mathcal{O}^{s})^{\boldsymbol\upmu[\boldsymbol\upiota]}_{\boldsymbol\upnu}(\boldsymbol\Uptheta)\}.  \]
Approximate ideal arithmetic in this setting takes the anticipated form upon combining Theorems \ref{Ogrowthdecayprod} and \ref{matrixideoarith}.

Restrict to the case of row vectors:
let $\boldsymbol\uptheta = (\uptheta_{1},\dots, \uptheta_{s})\in\K^{s}$, so that 
\[\bast\mathcal{O}^{s} (\boldsymbol\uptheta) = \{\bast\boldsymbol \upalpha\in\mathcal{O}^{s}|\; \exists\bast \upalpha^{\perp}\in
\bast\mathcal{O}\text{ s.t. }\boldsymbol\uptheta\odot\bast\boldsymbol \upalpha\simeq \bast \upalpha^{\perp}\in\bast\mathcal{O}\}\] 
where $\odot$ is the dot product.   The dual group 
\[ \bast \mathcal{O} (\boldsymbol\uptheta)^{\perp}=\{ \bast \upalpha^{\perp}|\;\exists  \bast\boldsymbol \upalpha\in\bast\mathcal{O}^{s} \text{ s.t. }\boldsymbol\uptheta\odot\bast\boldsymbol \upalpha\simeq \bast \upalpha^{\perp}   \}\subset\bast\mathcal{O} \] 
is  the {\bf nonprincipal diophantine approximation group} generated by the coordinates of $\boldsymbol\uptheta$.   Note
that if $\boldsymbol\upgamma = (\upgamma_{1},\upgamma_{2} )$ for $\upgamma_{1},\upgamma_{2}\in\mathcal{O}$ 
then $\bast \mathcal{O} (\boldsymbol\upgamma)^{\perp}$ is the ideal generated
by $\upgamma_{1},\upgamma_{2}$ in $\bast\mathcal{O}$.   If the Kronecker
product $\bast\boldsymbol\upalpha\otimes\bast\boldsymbol\upbeta$ for $\bast\boldsymbol\upalpha\in \bast\mathcal{O}^{s} (\boldsymbol\uptheta)$, $\bast\boldsymbol\upbeta\in \bast\mathcal{O}^{s'} (\boldsymbol\upeta)$ defines via approximate ideal arithmetic
an element of $\bast\mathcal{O}^{ss'} (\boldsymbol\uptheta\otimes\boldsymbol\upeta)$ then the corresponding
duals multiply as elements of $\bast\mathcal{O}$:
\[   (\bast\boldsymbol\upalpha\otimes\bast\boldsymbol\upbeta)^{\perp_{\boldsymbol\uptheta\otimes\boldsymbol\upeta}} 
=\bast\upalpha^{\perp_{\boldsymbol\uptheta}}\bast\upbeta^{\perp_{\boldsymbol\upeta}}.  \]
In particular, if we restrict to vectors whose entries belong to $K$ we recover the multiplicative arithmetic of ideals:
\begin{itemize}
\item[-] If $\boldsymbol\upgamma\in K^{s}$ then $\bast\mathcal{O}(\boldsymbol\upgamma)^{\perp}$ is the ideal generated by
the coordinates of $\boldsymbol\upgamma$.  Thus there exists an element $\boldsymbol\upgamma_{0}\in K^{2}$ such
that $\bast\mathcal{O}(\boldsymbol\upgamma)^{\perp}=\bast\mathcal{O}(\boldsymbol\upgamma_{0})^{\perp}$.
\item[-] The approximate ideal arithmetic of the Kronecker product, on the level of dual groups, is the pairing
\[   \bast\mathcal{O}(\boldsymbol\upgamma)^{\perp}\times  \bast\mathcal{O}(\boldsymbol\upgamma')^{\perp} 
\overset{\cdot}{\longrightarrow}  \bast\mathcal{O}((\boldsymbol\upgamma\otimes \boldsymbol\upgamma')_{0})^{\perp}.
 \]
 defined by the ordinary product in $\bast\mathcal{O}$.
\end{itemize}

\section{Decoupled Approximate Ideals and the Approximate Ideal Class Monoid}\label{ideologicalclasssection}

In this section we introduce the decoupled approximate ideals, the set of which extends the classical 
ideal class group.  The development offered is prefatory, intended primarily to provide structural harmony in these closing pages, relating 
approximate ideal arithmetic with the classical product of (nonprincipal) ideals.
A more evolved version of the ideas in this section will appear in the third installment in this series of papers \cite{Ge5}, in the context of Galois 
and Class Field Theory.

Let $K/\Q$ be of finite degree, $\K$ the associated Minkowski space. 
For $\boldsymbol z\in\K$ define the {\bf  decoupled diophantine approximation group}
by
\[ \bast[\mathcal{O}](\boldsymbol z) :=  \bast\mathcal{O}(\boldsymbol z)+  \bast\mathcal{O}(\boldsymbol z)^{\perp}\subset \bast \mathcal{O}.
 \]
If $\boldsymbol z\in\K$ is invertible then  
  \[ \bast[\mathcal{O}](\boldsymbol z) =   \bast[\mathcal{O}](\boldsymbol z^{-1})=\bast[\mathcal{O}](\boldsymbol z) + \bast[\mathcal{O}](\boldsymbol z^{-1})\subset \bast\mathcal{O}(\boldsymbol z ,\boldsymbol z^{-1})^{\perp}.\]  One can regard
 $ \bast[\mathcal{O}](\boldsymbol z)$ as the outcome of ``decoupling'' (making independent) numerator denominator pairs associated to $\boldsymbol z$.   Note that $\bast[\mathcal{O}](\boldsymbol z)$ has a natural approximate ideal structure defined by the subgroups
 \begin{align}\label{decoupledideology}
 \bast[\mathcal{O}]^{\boldsymbol \upmu[\boldsymbol \upiota]}_{\boldsymbol \upnu}(\boldsymbol z) = \{ \bast \upalpha + \bast\upbeta^{\perp}|
\bast \upalpha,\bast \upbeta\in \bast\mathcal{O}^{\boldsymbol \upmu[\boldsymbol \upiota]}_{\boldsymbol \upnu}(\boldsymbol z)
 \}.
 \end{align}
 
For $\upgamma = \upalpha/\upbeta\in K$, $\upalpha,\upbeta\in\mathcal{O}$, $\bast[\mathcal{O}](\upgamma)$
is an ideal and
\begin{align}\label{specialcase} 
\bast K\supset \bast (\upgamma ,1)\supset\bast[\mathcal{O}](\upgamma) \supset \bast (\upalpha,\upbeta )=\bast\mathcal{O}(\upalpha,\upbeta)^{\perp} .
\end{align}
(In the above, $\bast (x,y)$ is the ultrapower of the fractional ideal $(x,y)$: the $\bast\mathcal{O}$-module generated by $x,y$.)
Note that for $\upalpha\in\mathcal{O}$,
\[ \bast[\mathcal{O}](\upalpha)=\bast\mathcal{O},\] so that the construction
$\bast[\mathcal{O}](\cdot )$ already identifies
the ultrapowers of classical principal ideals with the identity ideal $\bast (1)=\bast \mathcal{O}$.    
Moreover,  the association $\bast (\upalpha,\upbeta )\mapsto \bast[\mathcal{O}](\upalpha/\upbeta)$ is a projective invariant:
if $x\in\mathcal{O}$ then $\bast (x\upalpha,x\upbeta )$ is assigned the ideal $\bast[\mathcal{O}](\upalpha/\upbeta)$ as well.

The Theorem which follows emphasizes the importance of the approximate ideal structure in distinguishing decoupled diophantine approximation groups.
Its proof requires Kronecker's criterion (\cite{Ca}, pg. 53) for the existence of inhomogeneous simultaneous approximations, which we recast
in the language of matrix diophantine approximations.
Consider an $r\times s$ real matrix $\Uptheta$ and a vector $\boldsymbol w = (w_{1},\dots ,w_{r})^{T}\in\R^{r}$.
Kronecker's criterion declares that the following statements are equivalent:

\begin{enumerate}  
\item[K1.] There exists $\bast\boldsymbol n^{\perp}\in\bast\Z^{r}$, $\bast\boldsymbol n\in\bast\Z^{s}$ such that 
\[ \Uptheta \bast\boldsymbol n -\bast\boldsymbol n^{\perp} \simeq \boldsymbol w.\]
\item[K2.] If $\boldsymbol m = (m_{1},\dots ,m_{r})\in\Z^{r}$ satisfies $\boldsymbol m \cdot\Uptheta\in \Z^{s}$ then $\boldsymbol m\cdot \boldsymbol w\in\Z$.
\end{enumerate}

\begin{theo}\label{cornerstructure}  For  $\uptheta\in\R-\Q$, $\bast\Z (\uptheta ,\uptheta^{-1})^{\perp}=\bast\Z$.  Moreover, $\bast[\Z]( \uptheta) \subsetneqq\bast\Z$ as groups
$\Leftrightarrow$ $\uptheta$ is a quadratic irrationality.  
\end{theo}

\begin{proof} 

That $\bast\Z (\uptheta ,\uptheta^{-1})^{\perp}=\bast\Z$ is a consequence of Kronecker's criterion.  Indeed, for $\Uptheta =(\uptheta ,\uptheta^{-1})$
there is no $m\in\Z$ with $ m\cdot \Uptheta\in \Z^{2}$, so condition K2 is vacuously true: hence for any $w\in\Z$ we may find $\bast\boldsymbol n\in\bast\Z^{2}$, $\bast\boldsymbol n^{\perp}\in\bast\Z$ satisfying K1.  Therefore, by a diagonal argument, 
for any $\bast  w =\bast\{ w_{i}\}\in\bast\Z$ we may solve K1 with $\bast w$ in place of $w$.
Suppose now that $\uptheta$ is a quadratic irrationality.  In this case we will show that
$ \bast[\Z]( \uptheta)\cap\Z =\{ 0\}$: in other words, we will show that for all $0\not=N\in\Z$, the equation
\begin{align}\label{quadirreqn} \bast m +\bast n^{\perp} = N ,\quad \bast m\in\bast\Z (\uptheta ), \bast n^{\perp}\in\bast\Z (\uptheta^{-1} )\end{align}
has no solution.  To this end we consider the $2\times 2$ diagonal matrix $\Uptheta = {\rm diag} (\uptheta, \uptheta^{-1})$.  If there were a solution
$\bast n^{\perp} = N-\bast m$ to (\ref{quadirreqn}) then the vectors $\bast \boldsymbol m = (\bast m,\bast m)$, $\bast \boldsymbol m^{\perp} = (\bast m^{\perp},-\bast n)$ would provide a solution to the inhomogeneous equation
\[ \Uptheta \bast\boldsymbol m -\bast\boldsymbol m^{\perp} \simeq \boldsymbol w=(0,N\uptheta^{-1} )^{T}.\]
On the other hand, $\uptheta$ is quadratic $\Leftrightarrow$ there exists a vector $\boldsymbol p = (p_{1},p_{2})\in\Z^{2}$ with $\boldsymbol p\cdot\Uptheta \in\Z^{2}$.
But such a vector could never satisfy $\boldsymbol p\cdot \boldsymbol w = p_{2}N\uptheta^{-1}\in\Z$.  Therefore there can be no solutions to K1
for the choice $\boldsymbol w=(0,N\uptheta^{-1} )^{T}$.  If $\uptheta$ is not quadratic then condition K2 is vacuously satisfied since there exist
no vectors  $\boldsymbol p = (p_{1},p_{2})\in\Z^{2}$ with $\boldsymbol p\cdot\Uptheta \in\Z^{2}$.  Thus $\Z\subset \bast[\Z]( \uptheta)$.  By a diagonal
argument we can then solve $\Uptheta \bast\boldsymbol n -\bast\boldsymbol n^{\perp} \simeq \bast w$ for any $\bast w\in\bast\Z$.
\end{proof}

In sum: Theorem \ref{cornerstructure} does not imply that for $\uptheta\in\R-\Q$ not a quadratic irrationality, $\bast[\Z]( \uptheta) =\bast\Z$
as approximate ideals i.e.\ it is not the case that $\bast[\Z]^{ \upmu}_{ \upnu}(\uptheta)=\bast\Z^{\upmu}$
for all $\upmu,\upnu$.  Indeed, if $(\upmu,\upnu)\not\in {\rm Spec}(\uptheta)$ then $\bast[\Z]^{ \upmu}_{ \upnu}(\uptheta)=0$.  In general,
for $\boldsymbol z,\boldsymbol z'\in\K$,
if ${\rm Spec}(\boldsymbol z)\not= {\rm Spec}(\boldsymbol z')$ then   $ \bast[\mathcal{O}](\boldsymbol z)\not=  \bast[\mathcal{O}](\boldsymbol z')$
as approximate ideals.

 Denote by $ \mathcal{C}l (K )$ the ideal class group of $K$.  Recall \cite{Fre} that 
ideals $\mathfrak{a}=(\upalpha,\upbeta ),\mathfrak{a}' =(\upalpha ',\upbeta')\subset\mathcal{O}$ define the same ideal class
 $\Leftrightarrow$  there exists
$A\in {\rm PGL}_{2}(\mathcal{O})$ with $A(\upgamma )=\upgamma'$ where $\upgamma = \upalpha/\upbeta$ and $\upgamma' = \upalpha'/\upbeta'$. 
That is, if we denote $\PR K= K\cup\{\infty \}$ then there is a bijection 
\[  \mathcal{C}l (K ) \longleftrightarrow {\rm PGL}_{2}(\mathcal{O})\backslash \PR K\]
where $[1]\leftrightarrow \infty$.

\begin{prop}\label{EquivImpDecoupEqual}  For $\boldsymbol z,\boldsymbol z'\in\K$, $\boldsymbol z \Bumpeq_{K}\boldsymbol z'$ $\Rightarrow$ $\bast[\mathcal{O}](\boldsymbol z )=\bast[\mathcal{O}](\boldsymbol z ')$ as approximate ideals.
For $\upgamma, \upgamma'\in K$, $ \upgamma \Bumpeq_{K}\upgamma'$
 $\Leftrightarrow$ $\bast[\mathcal{O}](\upgamma)=\bast[\mathcal{O}](\upgamma ')$ as ideals.
\end{prop}

\begin{proof}  Suppose that $\boldsymbol z'=A (\boldsymbol z)$ with $A=\left(\begin{array}{cc}
a & b \\
c& d
\end{array}      \right)$ for $a,b,c,d\in\mathcal{O}$. Then every element of 
$\bast[\mathcal{O}](\boldsymbol z ')$ is of the form
\[  (c\bast\upalpha^{\perp} + d\bast\upalpha) +   (a\bast\upbeta^{\perp} + b\bast\upbeta) = (d\bast\upalpha+
b\bast\upbeta)+(c\bast\upalpha+a\bast\upbeta)^{\perp} \in \bast[\mathcal{O}](\boldsymbol z )
\]
for some $\bast\upalpha,\bast\upbeta\in \bast[\mathcal{O}](\boldsymbol z )$.  
Thus $\bast[\mathcal{O}](\boldsymbol z )=\bast[\mathcal{O}](\boldsymbol z ')$ as groups.
If moreover $(c\bast\upalpha^{\perp} + d\bast\upalpha)\in \mathcal{O}^{\boldsymbol\upmu}_{\boldsymbol\upnu}(\boldsymbol z')$
then this implies that $\bast\upalpha\in  \mathcal{O}^{\boldsymbol\upmu}_{\boldsymbol\upnu}(\boldsymbol z)$.  Similarly,
$\bast\upbeta\in  \mathcal{O}^{\boldsymbol\upmu}_{\boldsymbol\upnu}(\boldsymbol z)$ from which it follows
that $(d\bast\upalpha+
b\bast\upbeta)\in  \mathcal{O}^{\boldsymbol\upmu}_{\boldsymbol\upnu}(\boldsymbol z)$.  By a similar argument
$(c\bast\upalpha+a\bast\upbeta)^{\perp}\in  \mathcal{O}^{\boldsymbol\upmu}_{\boldsymbol\upnu}(\boldsymbol z^{-1})$
and the first claim follows.  For $\upgamma,\upgamma'\in K$ with $\bast[\mathcal{O}]( \upgamma )=\bast[\mathcal{O}]( \upgamma')$, the latter 
is an ultrapower of a standard ideal $\mathfrak{a}\subset\mathcal{O}$.  If $\upgamma=\upalpha/\upbeta, \upgamma'=\upalpha'/\upbeta'$ 
then the ideals $\mathfrak{a}\supset (\upalpha,\upbeta ),  (\upalpha',\upbeta' )$ differ by multiples of elements of $\mathcal{O}$ and hence define the same
ideal class.  From this it follows that $\upgamma, \upgamma'$ differ by the action of an element of
${\rm PGL}_{2}(\mathcal{O})$.
\end{proof}

We define the {\bf  decoupled approximate ideal class set} as 
\[ \mathcal{C}l (\K ) = \left\{  [\mathcal{O}](\boldsymbol z)|\; \boldsymbol z\in \K \right\}.  \]
By Proposition \ref{EquivImpDecoupEqual}, there is a surjective function ${\rm PGL}_{2}(\mathcal{O})\backslash \PR\K\rightarrow  \mathcal{C}l (\K )$
where $\PR\K = \K\cup \{\infty\}$.

\begin{conj}  If $\bast[\mathcal{O}](\boldsymbol z )=\bast[\mathcal{O}](\boldsymbol z ')$ as approximate ideals then $\boldsymbol z \Bumpeq_{K}\boldsymbol z'$.
In particular, ${\rm PGL}_{2}(\mathcal{O})\backslash \PR\K\leftrightarrow  \mathcal{C}l (\K )$.
\end{conj}

By Proposition \ref{EquivImpDecoupEqual} we have
\[ \mathcal{C}l (K )\hookrightarrow \mathcal{C}l (\K ), \quad [\mathfrak{a}]\mapsto   \bast [\mathcal{O}](\upgamma)\]
where $\mathfrak{a}=(\upalpha,\upbeta )$ and $\upgamma =\upalpha/\upbeta$.  The image $ \bast [\mathcal{O}](\upgamma)$ is
of course a class with slightly more structure: the trivial ``growth only'' approximate ideal 
$\{\bast [\mathcal{O}]^{\boldsymbol \upmu[\boldsymbol \upiota]}(\upgamma)\}$.

We now introduce the approximate ideal product of decoupled approximate ideals associated to invertible elements of $\K$:
 
 \begin{theo}[Decoupled Approximate Ideal Product]\label{DeCoupProd}  Let $\boldsymbol z,\boldsymbol w\in\K^{\times}$.  There exists a bilinear map
 \[  \bast[\mathcal{O}]^{\boldsymbol \upmu[\boldsymbol \upiota]}_{\boldsymbol \upnu}(\boldsymbol z)\times
  \bast[\mathcal{O}]^{\boldsymbol \upnu[\boldsymbol \uplambda]}_{\boldsymbol \upmu}(\boldsymbol w)\stackrel{\cdot}{\longrightarrow}
   \bast[\mathcal{O}]^{\boldsymbol \upmu\cdot\boldsymbol\upnu[\boldsymbol \upiota\cdot\boldsymbol\uplambda]}_{\boldsymbol \upiota+\boldsymbol\uplambda}(\boldsymbol z\cdot\boldsymbol w)+
    \bast[\mathcal{O}]^{\boldsymbol \upmu\boldsymbol\cdot\upnu[\boldsymbol \upiota\cdot\boldsymbol\uplambda]}_{\boldsymbol \upiota+\boldsymbol\uplambda}(\boldsymbol z\cdot\boldsymbol w^{-1})   \]
    given by the ordinary product in $\bast\mathcal{O}$, whose restriction to $\mathcal{C}l (K)$ with the choices
    $\boldsymbol\upmu=\boldsymbol\upnu=-\boldsymbol\infty$ coincides with the ideal (class) product.
 \end{theo}
 
 \begin{proof}  Let $\bast\upalpha_{1} + \bast\upbeta_{1}^{\perp}\in  \bast[\mathcal{O}]^{\boldsymbol \upmu[\boldsymbol \upiota]}_{\boldsymbol \upnu}(\boldsymbol z)$, $\bast\upalpha_{2} + \bast\upbeta_{2}^{\perp}\in  \bast[\mathcal{O}]_{\boldsymbol \upmu}^{\boldsymbol \upnu[\boldsymbol \uplambda]}(\boldsymbol w)$.  Then 
 \[ \bast\upalpha_{1}\bast\upalpha_{2}+ \bast\upbeta_{1}^{\perp}\bast\upbeta_{2}^{\perp} = 
 \bast\upalpha_{1}\bast\upalpha_{2}+ (\bast\upbeta_{1}\bast\upbeta_{2})^{\perp} \in  \bast[\mathcal{O}]_{[\boldsymbol \upiota+\boldsymbol\uplambda]}^{\boldsymbol \upmu\cdot\boldsymbol\upnu[\boldsymbol \upiota\cdot\boldsymbol\uplambda]}(\boldsymbol z\cdot\boldsymbol w) . \]
 On the other hand, since $  \bast[\mathcal{O}]^{\boldsymbol \upnu[\boldsymbol \upiota]}_{\boldsymbol \upmu}(\boldsymbol z)^{\perp} =   \bast[\mathcal{O}]^{\boldsymbol \upnu[\boldsymbol \upiota]}_{\boldsymbol \upmu}(\boldsymbol z^{-1})$,
 $  \bast[\mathcal{O}]^{\boldsymbol \upnu[\boldsymbol \uplambda]}_{\boldsymbol \upmu}(\boldsymbol w)^{\perp} =   \bast[\mathcal{O}]^{\boldsymbol \upnu[\boldsymbol \uplambda]}_{\boldsymbol \upmu}(\boldsymbol w^{-1})$, we have as well
 \[  \bast\upalpha_{1}\bast\upbeta^{\perp}_{2} +\bast\upbeta^{\perp}_{1} \bast\upalpha_{2} = 
  \bast\upalpha_{1}\bast\upbeta^{\perp}_{2} +(\bast\upbeta_{1} \bast\upalpha_{2}^{\perp})^{\perp} \in
   \bast[\mathcal{O}]^{\boldsymbol \upmu\cdot\boldsymbol\upnu[\boldsymbol \upiota\cdot\boldsymbol\uplambda]}_{[\boldsymbol \upiota+\boldsymbol\uplambda]}(\boldsymbol z\cdot\boldsymbol w^{-1}) .\]
   Now if $\upgamma_{1}
   $, $\upgamma_{2}
   \in K$ then 
   $\bast[\mathcal{O}]^{-\boldsymbol\infty[\boldsymbol \upiota_{i}]}_{-\boldsymbol\infty}(\upgamma_{i})=\bast\mathfrak{a}_{i}$ for any 
   $\boldsymbol\upiota_{i}> -\boldsymbol\infty$ where $\mathfrak{a}_{i}=(\upalpha_{i},\upbeta_{i})$
   and  $\upgamma_{1}
   =\upalpha_{1}/\upbeta_{1}
   $, $\upgamma_{2}
   =\upalpha_{2}/\upbeta_{2}$.
   The image of the product $\bast[\mathcal{O}]^{-\boldsymbol\infty[\boldsymbol \upiota_{1}]}_{-\boldsymbol\infty}(\upgamma_{1})\cdot 
   \bast[\mathcal{O}]^{-\boldsymbol\infty[\boldsymbol \upiota_{2}]}_{-\boldsymbol\infty}(\upgamma_{2})$ is an ideal, generated by
   $\upalpha_{1}\upalpha_{2}, \upbeta_{1}\upbeta_{2},\upalpha_{1}\upbeta_{2},\upbeta_{1}\upalpha_{2}$, and so is
   equal to $\bast\mathfrak{a}_{1}\cdot \bast\mathfrak{a}_{2}$.
 \end{proof}

 In fact, we will consider approximate ideal products of decoupled groups corresponding to the approximate ideal index pair
$ \left((\boldsymbol\upmu_{1} [\boldsymbol\upiota_{1}],\boldsymbol\upnu_{1}),(\boldsymbol\upmu_{2} [\boldsymbol\upiota_{2}],\boldsymbol\upnu_{2})\right)  $
 in the same sense of {\it Note} \ref{genlprodnote}.  Note in this case that the approximate ideal product is defined when
$ \boldsymbol\upmu_{2}\geq \boldsymbol\upnu_{1},\quad \boldsymbol\upnu_{2}\leq \boldsymbol\upmu_{1}$, 
 in which case the product belongs to the group(s) with index 
$  (\boldsymbol\upmu_{1}\boldsymbol\upmu_{2} [\boldsymbol\upiota_{1}\boldsymbol\upiota_{2}],\boldsymbol\upiota_{1}+\boldsymbol\upiota_{2})$: that is
\begin{align}\label{genldecoupprod}
 \bast[\mathcal{O}]^{\boldsymbol \upmu_{2}[\boldsymbol \upiota_{2}]}_{\boldsymbol \upnu_{2}}(\boldsymbol z)\times
  \bast[\mathcal{O}]^{\boldsymbol \upmu_{1}[\boldsymbol \upiota_{1}]}_{\boldsymbol \upnu_{1}}(\boldsymbol w)\stackrel{\cdot}{\longrightarrow}
   \bast[\mathcal{O}]^{\boldsymbol \upmu_{2}\cdot\boldsymbol\upmu_{1}[\boldsymbol \upiota_{2}\cdot\boldsymbol\upiota_{1}]}_{[\boldsymbol \upiota_{2}+\boldsymbol\upiota_{1}]}(\boldsymbol z\cdot\boldsymbol w)+
    \bast[\mathcal{O}]^{\boldsymbol \upmu_{2}\cdot\boldsymbol\upmu_{1}[\boldsymbol \upiota_{2}\cdot\boldsymbol\upiota_{1}]}_{[\boldsymbol \upiota_{2}+\boldsymbol\upiota_{1}]}(\boldsymbol z\cdot\boldsymbol w^{-1}).
\end{align}
If $\uptheta, \upeta\in\mathfrak{B}$ are badly approximable then any product of the form (\ref{genldecoupprod}) has at least
one factor $0$, giving a trivial product.


Although it would be obviously desirable to assert that the decoupled product gives rise to a product in
$\mathcal{C}l (\K)$, we will see shortly that it is not true that for any $\boldsymbol z,\boldsymbol w\in\K^{\times}$ there exists
some $\boldsymbol x\in\K$ satisfying $\bast[\mathcal{O}]^{\boldsymbol \upmu}_{\boldsymbol \upnu}(\boldsymbol z)\cdot
  \bast[\mathcal{O}]^{\boldsymbol \upnu}_{\boldsymbol \upmu}(\boldsymbol w) 
  \subset
  \bast[\mathcal{O}]^{\boldsymbol \upmu\cdot\boldsymbol \upnu}(\boldsymbol x)$.
In order to sidestep this complication, we make the following definition.  Let
  \[ \bast[\mathcal{O}](\boldsymbol z|\boldsymbol w)\subset \bast[\mathcal{O}](\boldsymbol z\cdot\boldsymbol w) +
  \bast[\mathcal{O}](\boldsymbol z\cdot\boldsymbol w^{-1})\]
  be the group generated by the images of the maps in (\ref{genldecoupprod}) as one ranges over all growth-decay parameters: the {\bf  2-correlator decoupled diophantine approximation group} associated to 
 $\boldsymbol z,\boldsymbol w$.
  We endow 
  $ \bast[\mathcal{O}]  (\boldsymbol z|\boldsymbol w)$ with the approximate ideal structure coming from its parts: that is, 
 $ \bast[\mathcal{O}]^{\boldsymbol \upmu [\boldsymbol\upiota ]}_{\boldsymbol\upnu}(\boldsymbol z|\boldsymbol w)$
 is the group generated by the images of the maps in (\ref{genldecoupprod}) which belong to 
$ \bast[\mathcal{O}]^{\boldsymbol \upmu [\boldsymbol\upiota ]}_{\boldsymbol\upnu}(\boldsymbol z\cdot\boldsymbol w) +
  \bast[\mathcal{O}]^{\boldsymbol \upmu [\boldsymbol\upiota ]}_{\boldsymbol\upnu}(\boldsymbol z\cdot\boldsymbol w^{-1})$
  
 
  \begin{prop} \begin{enumerate}  
  \item For all $\boldsymbol z,\boldsymbol w\in\K^{\times}$, $ \bast[\mathcal{O}](\boldsymbol z|\boldsymbol w)= \bast[\mathcal{O}](\boldsymbol w|\boldsymbol z)$.
\item  For all $\boldsymbol z\in\K^{\times}$
\[ \bast[\mathcal{O}]^{\boldsymbol\upmu [\boldsymbol\upiota]}_{\boldsymbol\upnu}(\boldsymbol z|1) =
 \bast[\mathcal{O}]^{\boldsymbol\upmu [\boldsymbol\upiota]}_{\boldsymbol\upnu}(1|\boldsymbol z)= \bast[\mathcal{O}]^{\boldsymbol\upmu[\boldsymbol\upiota]}_{\boldsymbol\upnu}(\boldsymbol z) . \]
\item If $\uptheta,\upeta\in\mathfrak{B}$ are badly approximable then there exists no $\upomega\in\R$ with $\bast[\Z](\uptheta|\upeta)= 
 \bast[\Z](\upomega) $. 
 \item $\bast[\mathcal{O}](\boldsymbol z|\boldsymbol w)$ is invariant under the action of ${\rm PGL}_{2}(\mathcal{O})$ in each
  of its arguments: 
  \[ \bast[\mathcal{O}]( A(\boldsymbol z)| B(\boldsymbol w)) =\bast[\mathcal{O}](\boldsymbol z|\boldsymbol w)\]
  for all $A,B\in {\rm PGL}_{2}(\mathcal{O})$.

 \end{enumerate}
 \end{prop}

\begin{proof} Item (1) follows immediately from the definitions. (N.B. The relation $ {}_{\upmu}\!\owedge_{\upnu}$ is not commutative because of the insistence that its first argument satisfies the condition
 $\upmu\geq\upnu$.  The definition of the 2-correlator group 
does not impose this condition.) Item (2) follows from the fact that $1\in [\mathcal{O}]^{\boldsymbol\upmu[\boldsymbol\upiota]}_{\boldsymbol\upnu}(1)$ for any choice of growth-decay indices.   For $\uptheta,\upeta\in\mathfrak{B}$, $\bast[\Z](\uptheta|\upeta) =0$ = the zero approximate ideal.  Note that the zero approximate ideal
  is {\em not} equal to $\bast[\Z]( 0) = \bast\Z$: this gives an example of a 2-corrrelator approximate ideal which is not equal to the decoupled approximate ideal
  of some element in $\R$.  Item (4) follows essentially the same proof as Proposition \ref{EquivImpDecoupEqual}.
\end{proof}

  The definition of the 2-correlator decoupled diophantine approximation group can be extended to give meaning to an {\bf  $\boldsymbol n$-correlator decoupled diophantine approximation group}.  
 Fix $(\boldsymbol z_{n},\dots ,\boldsymbol z_{1})\in\K^{n}$ and
  consider a sequence of fine growth-decay parameters
  \[ (\vec{\boldsymbol \upmu}[\vec{\boldsymbol\upiota}],\vec{\boldsymbol\upnu}[\vec{\boldsymbol\uplambda}] ):= 
  (\boldsymbol\upmu_{n} [\boldsymbol\upiota_{n}],\boldsymbol\upnu_{n}[\boldsymbol\uplambda_{n}]),\dots,(\boldsymbol\upmu_{1} [\boldsymbol\upiota_{1}],\boldsymbol\upnu_{1}[\boldsymbol\uplambda_{1}])\]
   for which some iterated derived product (association of products) of the corresponding ordinary diophantine approximation groups can be performed.  
 Forming the corresponding iterated product of decoupled groups produces a subset of the sum of decoupled approximate ideals 
\[\sum_{\upsigma_{n},\dots ,\upsigma_{1}\in \{\pm 1\}}\bast[\mathcal{O}](\boldsymbol z_{n}^{\upsigma_{n}}\cdots\boldsymbol z_{1}^{\upsigma_{1}}).\]
The group generated by the images of all such iterated products is denoted 
\[ \bast[\mathcal{O}](\boldsymbol z_{n}|\cdots| \boldsymbol z_{1}),\]
endowed with an approximate ideal structure coming from its parts.
As in the case of the 2-correlator approximate ideal, the effect of an insertion of $1$ anywhere in $(\boldsymbol z_{n},\dots ,\boldsymbol z_{1})$ 
produces the same approximate ideal. 
The $n$-correlator decoupled approximate ideal 
 is invariant by the action of ${\rm PGL}_{2}(\mathcal{O})$ in each of its arguments; 
  we denote by $ \mathcal{C}l_{n}  (\K)$ the set of $n$-correlator decoupled approximate ideals.
 
 The approximate ideal product between such general correlator decoupled groups 
 \begin{align}\label{correlatorideoprod}\bast[\mathcal{O}]^{\boldsymbol \upmu[\boldsymbol\uplambda]}_{\boldsymbol \upnu}(\boldsymbol z_{n}|\cdots| \boldsymbol z_{1})\cdot \bast[\mathcal{O}]^{\boldsymbol \upmu'[\boldsymbol\uplambda']}_{\boldsymbol \upnu'}(\boldsymbol w_{m}|\cdots| \boldsymbol w_{1})
\subset \bast[\mathcal{O}](\boldsymbol z_{n}|\cdots| \boldsymbol z_{1}| \boldsymbol w_{m}|\cdots| \boldsymbol w_{1})
\end{align}
is defined, approximate ideal structures permitting.
It is clear that as one varies over all possible sequential parameters 
$(\vec{\boldsymbol \upmu}[\vec{\boldsymbol\uplambda}],\vec{\boldsymbol\upnu} ),
(\vec{\boldsymbol \upmu}'[\vec{\boldsymbol\uplambda}'],\vec{\boldsymbol\upnu}' )$ the set of images of the product (\ref{correlatorideoprod})
generates the group $ \bast[\mathcal{O}](\boldsymbol z_{n}|\cdots| \boldsymbol z_{1}| \boldsymbol w_{m}|\cdots| \boldsymbol w_{1})$.
We symbolize this state of affairs by writing
\[  \bast[\mathcal{O}](\boldsymbol z_{n}|\cdots| \boldsymbol z_{1})
\varoast
\bast[\mathcal{O}](\boldsymbol w_{m}|\cdots| \boldsymbol w_{1})
=
 \bast[\mathcal{O}](\boldsymbol z_{n}|\cdots| \boldsymbol z_{1}| \boldsymbol w_{m}|\cdots| \boldsymbol w_{1}).\]
If we write $ \bast[\mathcal{O}](\emptyset ):= (0)$ = the zero approximate ideal and
\[ \mathcal{C}l_{\infty} (\K) =
\bigcup_{n=0}^{\infty}\bigcup_{\boldsymbol z_{n},\dots , \boldsymbol z_{1}\in\PR\K}  \bast[\mathcal{O}](\boldsymbol z_{n}|\cdots| \boldsymbol z_{1}) \] then the approximate ideal product as defined above gives rise to 
 an associative binary operation 
 \[\varoast : \mathcal{C}l_{\infty}  (\K)\times\mathcal{C}l_{\infty}  (\K)\longrightarrow \mathcal{C}l_{\infty}  (\K), \]
 making of $ \mathcal{C}l_{\infty}  (\K)$ a monoid with unit $\bast[\mathcal{O}](1)$ {\it and} nullity $ \bast[\mathcal{O}](\emptyset )$
 : the {\bf  correlator approximate ideal class monoid}.

\begin{note} There is a surjective
 function 
 \[ \widetilde{\mathcal{C}l}_{n}(\K):= \K^{n}/{\rm PGL}_{2}(\mathcal{O})^{n}\longrightarrow \mathcal{C}l _{n}(\K )\]
 for each $n\geq 0$ (when $n=0$, $\widetilde{\mathcal{C}l}_{0}(\K):= \ast $ maps to $ \bast[\mathcal{O}](\emptyset )$).
 In addition, each insertion of $1$ gives an embedding $\widetilde{\mathcal{C}l}_{n}(\K)\hookrightarrow \widetilde{\mathcal{C}l}_{n+1}(\K)$
 so that taking limits gives a surjection
 \begin{align*}\label{classsurj} \widetilde{\mathcal{C}l}_{\infty}(\K):= \lim_{\longrightarrow}\widetilde{\mathcal{C}l} _{n}(\K )\longrightarrow
  \mathcal{C}l_{\infty}  (\K). 
 \end{align*}
 \end{note}
 We end with a few remarks on nilpotency and annihilation in $ \mathcal{C}l_{\infty}  (\K)$.
 
 \begin{theo}  Let $\uptheta, \uptheta'\in\mathfrak{B}$, $\upeta\in\mathfrak{W}_{1^{+}}$. 
 Then
 \[ \bast[\mathcal{O}](\uptheta)\varoast \bast[\mathcal{O}](\uptheta')= \bast[\mathcal{O}](\emptyset)
 =\bast[\mathcal{O}](\uptheta)\varoast \bast[\mathcal{O}](\upeta).
 \]
 \end{theo}
 
 \begin{proof}  Immediate since all possible approximate ideal products have a zero factor.  
 \end{proof}

 \begin{conj}  Let $\uptheta\in\R$ be of exponent $\upkappa$.  Then $\bast[\mathcal{O}](\uptheta)$ is $\lfloor\upkappa+2\rfloor$-step nilpotent. 
 If $\uptheta\in\mathfrak{W}_{\infty}$, $\bast[\mathcal{O}](\uptheta)$ is not nilpotent and is of infinite order.
 \end{conj}
 
 To prove the nilpotency statements in the conjecture, we need to have a better understanding of the {\it efficiency} of the growth decay filtrations.  
 In particular, the following questions must be addressed:
\begin{ques}[Decay Efficiency]\label{decayeff}   Let $\uptheta\in\R-\Q$, and suppose that we have $\upmu,\upnu$ and $\upnu'<\upnu$ with $\bast\Z^{\upmu}_{\upnu'}(\uptheta )\not=0$.
Is it the case that
$\bast\Z^{\upmu}_{\upnu}(\uptheta )-\bast\Z^{\upmu}_{\upnu'}(\uptheta )\not=\emptyset$?
\end{ques}

If we switch the roles of growth and decay in the above question, we obtain a corresponding question for the efficiency of growth indices.
Notice that this growth efficiency question has a positive response for the choice $\upmu'=\upnu$, by
Theorem \ref{symnotempty} of \S \ref{metrical} and the definition of symmetric diophantine approximations:
\[\bast\Z^{\rm sym}_{\upnu}(\uptheta) \subset
 \bast\Z^{\upmu}_{\upnu} (\uptheta ) - 
 \bast\Z^{\upnu}_{\upnu} (\uptheta).\]

\begin{ques}[Fine Growth Efficiency]\label{fineeff} Given $\upmu,\upnu$
 with $\bast\Z^{\upmu}_{\upnu}(\uptheta )\not=0$, do all values of $\upiota>\upmu$ yield $\bast\Z^{\upmu[\upiota]}_{\upnu}(\uptheta )\not=0$?
For $\upiota>\upiota'>\upmu$, is it the case that $\bast\Z^{\upmu[\upiota]}_{\upnu}(\uptheta )-\bast\Z^{\upmu[\upiota']}_{\upnu}(\uptheta )\not=\emptyset$?
\end{ques}

Question \ref{fineeff} is of interest in that it may be relevant to the following 

\begin{ques}[Product Efficiency]\label{prodeff} Given a nontrivial growth decay product, 
 \[ \bast\Z^{\upmu_{1}[\upiota_{1}]}_{\upnu_{1}}(\uptheta )\cdot \bast\Z^{\upmu_{2}[\upiota_{2}]}_{\upnu_{2}}(\upeta )\subset 
 \bast\Z^{\upmu_{1}\upmu_{2}[\upiota_{1}\upiota_{2}]}_{\upiota_{1}+\upiota_{2}}(\uptheta\upeta )\]
is it the case that for all $\upnu<\upiota_{1}+\upiota_{2}$,
 \[\bast\Z^{\upmu_{1}[\upiota_{1}]}_{\upnu_{1}}(\uptheta )\cdot \bast\Z^{\upmu_{2}[\upiota_{2}]}_{\upnu_{2}}(\upeta )\not\subset
\bast\Z^{\upmu_{1}\upmu_{2}[\upiota_{1}\upiota_{2}]}_{\upnu}(\uptheta\upeta )? \]
 \end{ques}
 
 For example, consider $\uptheta\in\mathfrak{W}_{1^{+}}$, which has exponent $\upkappa=1$.  Let us suppose that the answer to Question \ref{decayeff} is positive.
 Then the only nontrivial approximate ideal products that we may form are flat.  If the answer to Question \ref{prodeff} is also positive, then the image of any non zero 
 (flat) product belongs strictly to a growth-decay group with slow indices, so no further non zero products can be performed. This implies that 
 $\bast[\Z] (\uptheta )$ is 3-step nilpotent.

\end{document}